\documentclass[10pt, leqno]{amsart}
\makeatletter
\def\LaTeX{\leavevmode L\raise.42ex
    \hbox{\kern-.3em\size{\sfsize}{0pt}\selectfont A}\kern-.15em\TeX}
\makeatother

\newcommand{\BibTeX}{{\rm B\kern-.05em{\sc
          i\kern-.025emb}\kern-.08em\TeX}}

\makeatletter
\label{e:dispaa}
\label{e:dispau}
\label{e:dispav}
\label{e:dispaw}
\label{e:dispax}
\makeatother

\usepackage[showonlyrefs]{mathtools}
\mathtoolsset{showonlyrefs=true}

\usepackage{amsmath}
\usepackage{mathrsfs}    
\usepackage{amssymb}
\usepackage{amsthm}
\usepackage{graphicx,color}
\usepackage{hyperref}
\usepackage{verbatim} 

\usepackage[applemac]{inputenc}

\newtheorem{theorem}{Theorem}[section]
\newtheorem{proposition}[theorem]{Proposition}
\newtheorem{corollary}[theorem]{Corollary}
\newtheorem{lemma}[theorem]{Lemma}

\theoremstyle{definition}
\newtheorem{definition}[theorem]{Definition}
\newtheorem{remark}[theorem]{Remark}

\numberwithin{equation}{section}

\def\ep{\varepsilon}
\def\vphi{\varphi}

\def\div{\textrm{div}}

\def\Bcal{\mathcal{B}}

\def\Feh{\mathcal{F}}
\def\Geh{\mathcal{G}}

\def\Mcal{\mathcal{M}}
\def\Meh{\mathcal{M}}

\def\Ncal{\mathcal{N}}
\def\Oeh{\mathcal{O}}
\def\Peh{\mathcal{P}}
\def\Seh{\mathcal{S}}
\def\Reh{\mathcal{R}}

\def\Ucal{\mathcal{U}}
\def\Vcal{\mathcal{V}}

\def\N{\mathbb{N}}

\def\R{\mathbb{R}}

\def\Mr{\mathscr{M}}
\def\Nr{\mathscr{N}}

\def\Hh{\mathscr{H}}

\def\diag{\mathrm{diag}}

\def\hbar{\bar{h}}

\def\utilde{{\tilde u}}
\def\vtilde{{\tilde v}}
\def\atilde{{\tilde a}}
\def\btilde{{\tilde b}}
\def\mutilde{{\tilde \mu}}
\def\nutilde{{\tilde \nu}}

\def\dist{{\rm dist}}
\def\interior{{\rm int}}

\def\leq{\leqslant}
\def\geq{\geqslant}

\definecolor{verde}{rgb}{0,0.35,0.1} 
\definecolor{rosso}{rgb}{0.7,0,0}
\definecolor{blue}{rgb}{0, 0, 1}
\definecolor{viola}{rgb}{0.6,0,0.4}

\newcommand{\todo}[1][TO DO]{\typeout{#1}{\large\ttfamily\textbf{\fbox{**#1**}}}}
\newcommand{\rst}[1]{\ensuremath{{\mathbin |}%
\raise-.5ex\hbox{$#1$}}}

\begin{document}

\title[Optimal partition problems involving Laplacian eigenvalues]{Extremality conditions and regularity of solutions to optimal partition problems involving Laplacian eigenvalues}

\author[M. Ramos]{Miguel Ramos}\thanks{}

\author[H. Tavares]{Hugo Tavares}\thanks{}
\address{Hugo Tavares \newline \indent Center for Mathematical Analysis, Geometry and Dynamical Systems \newline \indent
Instituto Superior T\'ecnico, Universidade de Lisboa, \newline \indent
Av. Rovisco Pais, 1049-001 Lisboa, Portugal}
\email{htavares@math.ist.utl.pt}

\author[S. Terracini]{Susanna Terracini}
\address{Susanna Terracini \newline \indent Dipartimento di Matematica ``Giuseppe Peano'', Universit\`a di Torino \newline \indent
Via Carlo Alberto 10, \newline \indent 20123 Torino, Italy }
\email{susanna.terracini@unito.it}


\thanks{{Acknowlegments.}  The authors would like to thank Sandro Salsa for giving a hint that enabled a very short and elegant proof of Lemma \ref{prop:u_i/u_1_alpha_Holder_generalcase}. H. Tavares was supported by Funda\c c\~ao para a Ci\^encia e Tecnologia through the program Investigador FCT and through the project UID/MAT/04459/2013. H. Tavares and S. Terracini are partially supported through the project ERC Advanced Grant  2013 n. 339958 ``Complex Patterns for Strongly Interacting Dynamical Systems - COMPAT''}

\date{\today}

\begin{abstract}
Let $\Omega\subset \mathbb{R}^N$ be an open bounded domain and $m\in \mathbb{N}$. Given $k_1,\ldots,k_m\in \mathbb{N}$, we consider a wide class of optimal partition problems involving Dirichlet eigenvalues of elliptic operators, of the following form
\[
\inf\left\{F(\lambda_{k_1}(\omega_1),\ldots, \lambda_{k_m}(\omega_m)):\ (\omega_1,\ldots, \omega_m)\in \mathcal{P}_m(\Omega)\right\},
\]
where $\lambda_{k_i}(\omega_i)$ denotes the $k_i$--th eigenvalue of $(-\Delta,H^1_0(\omega_i))$ counting multiplicities, and $\mathcal{P}_m(\Omega)$ is the set of all open partitions of $\Omega$, namely
\[
\mathcal{P}_m(\Omega)=\left\{(\omega_1,\ldots,\omega_m):\ \omega_i\subset \Omega \text{ open},\ \omega_i\cap \omega_j=\emptyset\ \forall i\neq j\right\}.
\]
While existence of a \emph{quasi-open} optimal partition $(\omega_1,\ldots, \omega_m)$ follows from a general result by Bucur, Buttazzo and Henrot [Adv. Math. Sci. Appl. 8, 1998], the aim of this paper is to associate with such minimal partitions and their eigenfunctions some suitable extremality conditions and to exploit them,  
proving as well the Lipschitz continuity of some eigenfunctions, and regularity of the partition in the sense that the free boundary $\cup_{i=1}^m \partial \omega_i\cap \Omega$ is, up to a residual set, locally a $C^{1,\alpha}$ hypersurface. This last result extend the ones in the paper by Caffarelli and Lin [J. Sci. Comput. 31, 2007] to the case of higher eigenvalues.

\end{abstract}

\maketitle

\section{Introduction}
Let $\Omega$ be an open bounded domain of $\R^N$ ($N\geq 2$) and let $m\in \N$ and $k_1,\ldots, k_m\in \N$. The starting point and first aim of this paper is to investigate the extremality conditions and the regularity properties of solutions to problems of type:
\begin{equation}\label{eq:OPPmain}
\inf_{(\omega_1,\ldots,\omega_m)\in \Peh_m(\Omega)} \sum_{i=1}^m \lambda_{k_i}(\omega_i),
\end{equation}
where 
\[
\Peh_m(\Omega)=\left\{(\omega_1,\ldots,\omega_m)\subseteq \Omega^m:\ \omega_i \text{ open }\forall i,\ \omega_i\cap \omega_j=\emptyset\ \forall i\neq j\right\}
\]
denotes the class of all partitions of $\Omega$ with $m$ open sets,  and $\lambda_k(\omega)$ is the $k$--th eigenvalue of $-\Delta$ in $H^1_0(\omega)$, counting multiplicities. Observe that \eqref{eq:OPPmain} is an example of an optimal partition problem (see for instance \cite[Section 5.5]{BucurButtazzoBook} or \cite{BucurButtazzoHenrot}) with cost function $\Phi(\omega_1,\ldots,\omega_m):=\sum_{i=1}^m \lambda_{k_i}(\omega_i)$, and class of admissible sets given by $\mathcal{A}(\Omega):=\{\omega\subseteq \Omega: \ \omega \text{ open set}\}$. 

Optimal partition problems belong to the wide family of shape optimization problems, and have many applications in physics and engineering. They also play a fundamental role in the study of the nodal sets of eigenfunctions of Schr\"odinger operators \cite{HHT}, as well as in the proof of monotonicity formulas \cite{BBHU,BKS,BHV,helffer,HHOT2,HHOT3,polacik,TavaresTerracini1}.  In general, these kind of problems may only have solution in a relaxed sense \cite{ButtazzoDalMaso, ButtazzoTimofte}, except when one imposes certain geometric constraints on the admissible domains, or some monotonicity properties on the cost function (we refer the reader to the book by Bucur and Buttazzo \cite{BucurButtazzoBook} for a good survey on these issues). Problem \eqref{eq:OPPmain} falls into the second case: as $\Phi$ is monotone decreasing with respect to set inclusion, as well as lower semi continuous for the $\gamma$--convergence, a general abstract result by Bucur, Buttazzo and Henrot (\cite{BucurButtazzoHenrot}) implies the existence of solution for \eqref{eq:OPPmain} in the class of \emph{quasi-open sets}. In general, it is hard to pass from quasi-open to open sets. By using a penalization technique with partition of unity functions, Bourdin, Bucur and Oudet \cite{BourdinBucurOudet} give a different proof for the existence of quasi-open solutions, while proving the existence of open solutions for the bidimensional case $N=2$ (by using a compactness result by {\v{S}}ver{{\'a}}k \cite{Sverak} which only holds in dimension two). 

Let us consider a function $F:\R^m\to \R$ such that
\begin{enumerate}
\item[(F1)] $F\in C^1((\R^+)^m)$, and $\frac{\partial F}{\partial x_i}>0$ in $(\R^+)^m$ for every $i=1,\ldots,m$;
\item[(F2)] Given $i$, and for every fixed $\bar x_j>0$, $j\neq i$, we have 
\[
F(\bar x_1,\ldots, \bar x_{i-1}, x_i, \bar x_{i+1}\ldots, \bar x_{m})\to +\infty \qquad \text{ as } x_i\to +\infty.
\]
\end{enumerate}
Consider
\begin{equation}\label{eq:OPPmain_general_with_F}
\inf_{(\omega_1,\ldots,\omega_m)\in \Peh_m(\Omega)} F(\lambda_{k_1}(\omega_1),\ldots, \lambda_{k_m}(\omega_m)).
\end{equation}

A key point in the qualitative theory of solutions to the optimal partition problem is be to derive proper extremality conditions and initial regularity for the eigenfunctions associated with the elements of the partition. In this aspect, the case of the first eigenvalue is much simpler, yet non trivial, and extremality conditions were first derived by Conti, Terracini, Verzini:

\begin{theorem}[Conti, Terracini, Verzini, \cite{ContiTerraciniVerziniOPP}]\label{thm:CTV}
Assume $k_i=1$, $i=1\dots, m$, and let $F(x_1,\dots,x_m)=\sum_ix_i^p$, for some $p>0$. Then the optimal partition problem \eqref{eq:OPPmain_general_with_F} admits an open and connected solution $(\omega_1,\ldots,\omega_m)\in \Peh_m(\Omega)$, such that,  for each $i=1,\ldots, m$, the ($L^2$--normalized and nonnegative) eigenfunctions $u_{i}$ associated with the eigenvalues $\lambda_{1}(\omega_i)$ are Lipschitz continuous.  In addition, there are numbers  $a_i>0$ such that

\begin{equation}\label{eq:CTV}
-\Delta\left(a_i u_i-\sum_{j\neq i}a_ju_j\right) \geq \lambda_{1}(\omega_i) a_iu_i-\sum_{j\neq i} \lambda_{1}(\omega_j)a_ju_j\,,\qquad i=1,\dots,m\,. 
\end{equation}
\end{theorem}

It is worthwhile noticing that this simple condition does not hold in the case of optimal partitions related with higher eigenvalues. The lack of extremality conditions  is a well known problem in shape optimization, as highlighted  for instance in the recent paper  \cite{BMPV2014}.  In facts, no domain variation formula is available in the unfortunate (yet possible) case of degenerate eigenvalues (observe that we will take any prescribed $k_1,\ldots, k_m\in \N$). This fact, together with the initial lack of regularity of the interfaces, makes Hadamard type conditions unavailable for this problem.  

In particular, the system of inequalities  \eqref{eq:CTV} immediately yields regularity of the nodal set in regions when only two components are active, and proportionality of their gradients at the interfaces. In the case of the first eigenvalues,  when $F(x_1,\dots,x_m)=\sum_{i}x_i$, much stronger regularity results were proved in \cite{Aram,CaffarelliLin2,CaffarelliLinJAMS} and will be discussed later.

Our first main result applies to  a wide class of optimal partition problems which includes \eqref{eq:OPPmain}, providing  openness of the optimal partition, Lipschitz continuity of eigenfunctions and  extremality conditions for the corresponding eigenfunctions.

\begin{theorem}\label{thm:extremality_result}
The optimal partition problem \eqref{eq:OPPmain_general_with_F} admits a solution $(\widetilde\omega_1,\ldots,\widetilde\omega_m)\in \Peh_m(\Omega)$, such that,  for each $i=1,\ldots, m$, there exist $1\leq l_i \leq k_i$,  Lipschitz continuous eigenfunctions $\tilde u_{l_i}^i, \tilde u_{l_i+1}^i\ldots \tilde u_{k_i}^i$ associated to the eigenvalues $\lambda_{l_i}(\widetilde \omega_i)=\ldots=\lambda_{k_i}(\widetilde \omega_i)$, such that
\[
\tilde u^i_{l_i},\tilde u^i_{l_i+1},\ldots, \tilde u^i_{k_i} \qquad \text{ are $L^2$ orthonormal.}
\]
In addition, there are numbers  $\tilde a^i_{l_i},\ldots, \tilde a^i_{k_i}>0$
such that, defining, for $x_0\in \Omega$ and $r\in (0,\dist(x_0,\partial \Omega))$,

\[
E(x_0,r)=\sum_{i=1}^m \frac{1}{r^{N-2}}  \int_{B_r(x_0)}  \sum_{j=l_i}^{k_i}  \tilde a^i_j\left(|\nabla \tilde u^i_j|^2 -\lambda_{k_i}(\tilde u^i_j)^2\right) \,dx,
\]
then there holds:

\begin{equation}\label{eq:extremality}
\dfrac{d E}{dr} 
=\sum_{i=1}^m \left(
\frac{2}{r^{N-2}} \int_{\partial B_r(x_0)} 
 \sum_{j=l_i}^{k_i}  \tilde a^i_j|\partial_n \tilde u^i_j|^2
d\sigma-\frac{2}{r^{N-1}}  \int_{B_r(x_0)}  \sum_{j=l_i}^{k_i} \tilde a^i_j\lambda_{k_i}(\tilde u^i_j)^2\, dx\right).
\end{equation}
\end{theorem}

\begin{remark}
We stress that the step from the first to higher eigenvalues imposes a change of perspective. Indeed, the extremality condition \eqref{eq:extremality} involves two non trivial selection facts: at first, it may involve only  some of the $\lambda_{k_i}$--eigenfunctions, and secondly the weights $\tilde a_{n}^i$ are not given directly by the original problem and we only know  some partial information - see \eqref{eq:weightsbounds} ahead.  The selection of  the eigenfunctions and coefficients appears in a rather mysterious way as a consequence of our strategy, which consists in a double penalization/approximation of the problem.  The regularity of the eigenfunctions shows close similarities with the analogue result of \cite{BMPV2014}, in the context of shape optimization. 
\end{remark}

The extremality condition \eqref{eq:extremality} is a weak reflection law in the sense of \cite{TavaresTerracini1} and will play the same role here. Indeed, following the approach introduced by Caffarelli and Lin in their celebrated paper \cite{CaffarelliLinJAMS}, the main application of Theorem \ref{thm:extremality_result} concerns  the regularity of the optimal partition, or, more specifically, of the common boundaries of the elements of the partition itself (see also \cite{TavaresTerracini1}, where the case of non minimizing configuration is treated). Following \cite{CaffarelliLinJAMS}, we shall adopt the following definition of regular partition.

\begin{definition}\label{definition:regular_partition} An open partition $(\omega_1,\ldots,\omega_m)\in \Peh_m(\Omega)$ is called \emph{regular} if:
\begin{enumerate}
\item[1.] denoting $\Gamma=\Omega\setminus \bigcup_{i=1}^m\omega_i$, there holds\footnote{Here, $\Hh_\text{dim}(\cdot)$ denotes the Hausdorff dimension of a set.} $\Hh_\text{dim}(\Gamma)\leq N-1$; 
\item[2.] there exists a set $\Reh\subseteq \Gamma$, relatively open in $\Gamma$, such that
\begin{itemize}
\item[-] $\Hh_\text{dim} (\Gamma\setminus \Reh)\leq N-2$;
\item[-] $\Reh$ is a collection of  hypersurfaces of class $C^{1,\alpha}$ (for some $0<\alpha<1$), each one separating two different elements of the partition. 
\end{itemize}
\end{enumerate}
Moreover we ask that, if $N=2$, then actually $\Gamma\setminus \Reh$ is a locally finite set, and that $\Reh$ consists of a locally finite collection of curves meeting at singular points with equal angles.
\end{definition} 

\begin{remark}
In particular, observe that a regular partition exhausts the whole domain, as it satisfies $\overline \Omega=\cup_{i=1}^m \overline{\omega_i}$. 
\end{remark}

We will show the following.

\begin{theorem}\label{thm:first_main_result}
The optimal partition problem \eqref{eq:OPPmain_general_with_F} admits a regular solution $(\widetilde\omega_1,\ldots,\widetilde\omega_m)\in \Peh_m(\Omega)$.  Moreover, for each $i=1,\ldots, m$, there exist $1\leq l_i \leq k_i$ and eigenfunctions $\tilde u_{l_i}^i, \tilde u_{l_i+1}^i\ldots \tilde u_{k_i}^i$ associated to the eigenvalues $\lambda_{l_i}(\widetilde \omega_i)=\ldots=\lambda_{k_i}(\widetilde \omega_i)$, such that
\[
\tilde u^i_{l_i},\tilde u^i_{l_i+1},\ldots, \tilde u^i_{k_i} \qquad \text{ are $L^2$ orthonormal and Lipschitz continuous,}
\]
and
\[
\widetilde \omega_i=\interior\left(\, \overline{ \left\{\sum_{j=l_i}^{k_i}  (\tilde u^i_j)^2>0\right\} }\, \right).
\]
Finally, on the regular part of the common boundary, we have the following extremality condition: for each $i=1,\ldots, m$ there exist coefficients $\tilde a^i_{l_i},\ldots, \tilde a^i_{k_i}>0$ such that
\begin{quote}
given $x_0\in \Reh$, denoting by $\widetilde \omega_i$ and $\widetilde \omega_j$ the two adjacent sets of the partition at $x_0$,
\begin{equation}\label{eq:extremality_main}
\mathop{\lim_{x\to x_0}}_{x\in \widetilde \omega_i} \sum_{n=l_i}^{k_i} \tilde a_{n}^i |\nabla \tilde u^i_n(x)|^2=\mathop{\lim_{x\to x_0}}_{x\in \widetilde \omega_j} \sum_{n=l_j}^{k_j} \tilde a_n^j |\nabla \tilde u^j_n(x)|^2\neq 0.
\end{equation}
\end{quote} 
\end{theorem}

\begin{remark}
Both in this theorem and in the intermediate results of Section \ref{sec:intermediate}, the regularity of the free boundary can be extended up to the boundary $\partial \Omega$ under appropriate regularity assumptions on the set $\Omega$. We refer the reader to \cite[Section 7]{TavaresTerracini1} for the details.
\end{remark}

Theorem \ref{thm:first_main_result} extends the regularity result of \cite{CaffarelliLin2} to the case of higher eigenvalues and more general cost functions. As already pointed out, the case of the first eigenvalues,  $k_1=\ldots=k_m=1$ is much simpler, as it can be characterized by an absolute minimization of an energy functional in a singular space, namely
\begin{equation}\label{eq:equivalent_characterization_of_sum_lambda1}
\inf \left\{\sum_{i=1}^m \int_\Omega |\nabla u_i|^2\, dx:\ u_i\in H^1_0(\Omega),\ \int_\Omega u_i^2\, dx=1,\ u_i\cdot u_j\equiv 0 \ \forall i\neq j\right\}.
\end{equation}

For this case, the first partial regularity results track back to  \cite{ContiTerraciniVerziniOPP}, while the general theorem was then established in \cite{CaffarelliLin2}, using techniques that were previously developed in \cite{CaffarelliLinJAMS} (where the authors deal with a minimization of a free energy in a singular space).  We cite also \cite{TavaresTerracini1} for the extension from minimizing to critical configurations was performed.  Although we will follow the same general strategy of \cite{CaffarelliLinJAMS},  
 (use Almgren's monotonicity formulas, Hausdorff Reduction Principle, or Boundary Harnack Principles in NTA domains), we stress that the case of higher eigenvalues differs from that of the first one (or the minimizers of the free energy) and requires the elaboration of different ideas and strategies. We finally refer to the papers \cite{ContiTerraciniVerziniNehari,ContiTerraciniVerziniOptimalNonlinear}, where related optimal partition problems involving sums of first nonlinear eigenvalues are considered.

\section{Strategy of the proofs and intermediate results}\label{sec:intermediate}

To simplify the presentation, we deal from now on with $F(x_1,\ldots,x_m)=\sum_{i=1}^m x_i$, hence with \eqref{eq:OPPmain}, which already presents all the main features of the problem. The general case follows exactly in the same way, as the proofs only use properties (F1) and (F2).

Another natural approach to the problem is to consider limiting profiles of solutions to the singularly perturbed problem 
\begin{equation}\label{eq:BEC}
-\Delta u_i=\lambda_i u_i -\beta \mathop{\sum_{j=1}^m}_{j\neq i} u_iu_j^2,\quad u_i\in H^1_0(\Omega), \qquad i=1,\ldots, m,
\end{equation}
under the constraints
\[
\int_\Omega u_i^2\, dx=1,\qquad i=1,\ldots, m.
\]
These systems have been the object of an intensive study in the last decade, in particular in the case of competitive interaction $\beta>0$, mainly because of its interesting mathematical features, as well as for its physical applications (for instance in the study of Bose-Einstein condensation). Their relation with optimal partition problems has also been addressed, for instance in \cite{BTWW,CaffarelliLin2,CLLL,ContiTerraciniVerziniOPP,HHT,TavaresTerracini2}. It has been shown in \cite{CaffarelliLinJAMS,NTTV1, TavaresTerracini2} that, in some situations, phase separation occurs between different components as the competition parameter increases, i.e., $\beta\to \infty$. In particular it is shown in \cite{NTTV1, TavaresTerracini2} that, by taking an $L^\infty$ bounded family of solutions $(u_\beta)_\beta$, and corresponding bounded coefficients $(\lambda_{i\beta})_\beta$, then there exists a limiting profile $u_i:=\lim_{\beta\to +\infty} u_{i\beta}$ such that $(\{u_1\neq 0\},\ldots, \{u_m\neq 0\})\in \Peh_m(\Omega)$, and
\[
-\Delta \tilde u_i=\lambda_i \tilde u_i \qquad \text{ in } \{u_i\neq 0\}.
\]
This clearly illustrates the relation between optimal partitions involving eigenvalues and the system of Schr\"odinger equations \eqref{eq:BEC}. In particular, it is known that \eqref{eq:equivalent_characterization_of_sum_lambda1} can be well approximated (as $\beta\to \infty$) by the ground state (least energy) levels of \eqref{eq:BEC}, namely:
\[
\inf \left\{\int_\Omega \left(\sum_{i=1}^m |\nabla u_i|^2 + \beta \mathop{\sum_{i,j=1}}_{i<j}^m u_i^2 u_j^2\right)\, dx : \ u_i\in H^1_0(\Omega) \text{ and } \int_\Omega u_i^2\, dx=1 \right\}.
\]
Thus, using this approach and the results in \cite{TavaresTerracini1}, one proves once again the existence of a regular partition to the problem of summing first eigenvalues. However, passing to higher eigenvalues is not an easy task, as one needs to construct suitable minimax characterizations at higher energy levels of \eqref{eq:BEC}. Indeed in \cite{TavaresTerracini2}, by using a new notion of vector genus, several sign changing solutions are build for \eqref{eq:BEC}, and by taking the least energy sign-changing solution among these, one approaches the second eigenfunctions associated to the optimal partition problem $\inf_{\Peh_m(\Omega)} \lambda_2(\omega_i)$. By collecting the previous results, one then can actually solve \eqref{eq:OPPmain} for $1\leq k_i\leq 2$. In order to solve the general problem with higher eigenvalues, however, it does not seem completely clear to us which variational characterization for solutions of \eqref{eq:BEC} one could take (see \cite[Subsection 4.1]{TavaresTerracini2} for a conjecture). Thus here we decided to follow a different strategy relying on a double approximation procedure, which we will describe next. The relevant and surprising fact is that instead of taking minimax levels for a certain energy functional, we are able to approximate the problem \eqref{eq:OPPmain} for every $k_i\in \N$ through a symmetric constrained energy minimization.

\bigskip

In order to cope with the problem of not knowing a priori what is the multiplicity of each set of the optimal partition, our motivation was to try to find approximate solutions of \eqref{eq:OPPmain} through the minimization process of a certain energy functional. Partially inspired by \cite{HHT} (where a different problem is treated), we let $p\in \N$ and consider the problem
\begin{equation}\label{eq:OPPsecondary}
\inf_{(\omega_1,\ldots,\omega_m)\in \Peh_m(\Omega)} \sum_{i=1}^m \Bigl(\sum_{j=1}^{k_i} (\lambda_j(\omega_i))^p\Bigr)^{1/p}.
\end{equation}
Observe that \eqref{eq:OPPsecondary} is a good approximation for \eqref{eq:OPPmain} for large $p$, as,  given $k\in \N$ and any positive real numbers $a_1,\ldots, a_k$, there holds $(a_1^p+\ldots +a_k^p)^{1/p}\to \max\{a_1,\ldots, a_k\}$ as $p\to \infty$. Thus, for any given partition $(\omega_1,\ldots, \omega_m)\in \Peh_m(\Omega)$, 
$$
\sum_{i=1}^m \Bigl(\sum_{j=1}^{k_i} \lambda_j(\omega_i)^p\Bigr)^{1/p} \to \sum_{i=1}^m \lambda_{k_i}(\omega_i)\qquad \text{ as }p\to +\infty.
$$
We will prove that an optimal solution $(\omega_{1p},\ldots, \omega_{mp})$ of \eqref{eq:OPPsecondary} exists and approaches, as $p\to \infty$, a solution of our original problem \eqref{eq:OPPmain}. More precisely, our second main result is the following.

\begin{theorem}\label{thm:2ndmain}
Given $p\in \N$, the optimal partition problem \eqref{eq:OPPsecondary}  admits a regular solution $(\omega_{1p},\ldots,\omega_{mp})\in \Peh_m(\Omega)$. 

Moreover, under the previous notations, for each $i=1,\ldots, m$ there exist eigenfunctions $u_{1p}^i,u_{2p}^i,\ldots, u_{k_i p}^i$ associated to the eigenvalues $\lambda_{1}(\omega_{ip}),\lambda_{2}(\omega_{ip}),\ldots, \lambda_{k_i}(\omega_{k_i p})$, such that

\[
u^i_{1p}, u_{2p}^i,\ldots, u^i_{k_ip} \qquad \text{ are $L^2$ orthonormal and Lipschitz continuous}
\]
and
\[
\omega_{ip}= \left\{\sum_{j=1}^{k_i}  (u^i_{jp})^2>0\right\}=\interior\left(\, \overline{ \left\{\sum_{j=1}^{k_i}  (u^i_{jp})^2>0\right\} }\, \right).
\]
On the regular part of the common boundary, we have the following extremality condition: for each $i=1,\ldots, m$ there exist coefficients $a^i_{1p},\ldots, a^i_{k_ip}>0$ such that
\begin{quote}
given $x_0\in \Reh$, denoting by $\omega_{ip}$ and $\omega_{jp}$ the two adjacent sets of the partition at $x_0$,
\[
\mathop{\lim_{x\to x_0}}_{x\in \omega_{ip}} \sum_{n=1}^{k_i} a_{np}^i |\nabla u^i_{np}(x)|^2=\mathop{\lim_{x\to x_0}}_{x\in \omega_{jp}} \sum_{n=1}^{k_j} a_{np}^j |\nabla u^j_{np}(x)|^2\neq 0
\]
\end{quote} 
Finally, as $p\to +\infty$,
\[
a_{np}^i \to 0\ \forall n=1,\ldots, l_i-1,\qquad a_{np}^i\to \tilde a_n^i \ \forall n=l_i,\ldots, k_i,
\]
and we have strong convergence of the last $k_i-l_i+1$ eigenfunctions:
\[
u^i_{np}\to \tilde u^i_n \text{ strongly in } H^1_0(\Omega)\cap C^{0,\gamma}(\overline \Omega), \qquad \forall 0<\gamma<1,\  \forall n=l_i,\ldots, k_i
\]
(where $\tilde a_n^i$ and $\tilde u^i_n$ are the coefficients and functions appearing in Theorem \ref{thm:first_main_result}).
\end{theorem}
 
 The existence and regularity results of the previous theorem will come out as a particular case of a more general result, for optimal partition problems that depend on all eigenvalues up to a certain order.
More precisely, let $k\in \N$ and a function $\psi:(\R^+_0)^k \to \R$ that satisfies (F1), (F2), and is invariant under permutations of the variables, that is:
\begin{itemize}
\item[-]$\frac{\partial \psi}{\partial \xi_i}(\xi_1,\ldots, \xi_k)>0$ for every $i$, whenever $\xi_1,\ldots, \xi_k>0$;
\item[-] Given $i$, and for every fixed $\bar \xi_j>0$, $j\neq i$, we have 
\[
\psi(\bar \xi_1,\ldots, \bar \xi_{i-1}, \xi_i, \bar \xi_{i+1}\ldots, \bar \xi_{m})\to +\infty \qquad \text{ as } \xi_i\to +\infty.
\]
\item[-] $\psi(\sigma(\xi_1,\ldots, \xi_k))=\psi (\xi_1,\ldots, \xi_k)$, for every permutation $\sigma\in S_k$.
\end{itemize}
Now, given $m$ functions $\psi_1,\ldots, \psi_m$ satisfying these assumptions (eventually for different $k_1,\ldots, k_m$), we will consider the following class of optimal partition problems
\begin{equation}\label{eq:general_OPP_with_psi}
\inf_{(\omega_1,\ldots,\omega_m)\in \Peh_m(\Omega)} \sum_{i=1}^m \psi_i( \lambda_1(\omega_i),\ldots,\lambda_{k_i}( \omega_i)),
\end{equation}
As examples we have \eqref{eq:OPPsecondary}, the optimization with cost function $\sum_{i=1}^{m} \prod_{j=1}^{k_i} \lambda_j(\omega_i)$, or any combination of the two, among many other.

We will prove that \eqref{eq:general_OPP_with_psi} can be reformulated as a constrained minimization problem of an energy functional:
\begin{equation}\label{eq:c_infty_INTRO}
\inf \left\{ \sum_{i=1}^m \psi_i\left(\int_\Omega |\nabla u_1^i|^2\,dx,\ldots, \int_\Omega |\nabla u^i_{k_i}|^2\, dx\right):
\begin{array}{c}
u^i\in H^1_0(\Omega;\R^{k_i}),\ u^i_k \cdot u^j_l\equiv 0\ \forall k,l, i\neq j\\[6pt]
\displaystyle \text{and, for every $i$, } \int_\Omega u^i_k u^i_l\, dx=\delta_{kl}\ \forall k,l\\[6pt]
\displaystyle  \int_\Omega \nabla u^i_k \cdot \nabla u^i_l\, dx=0\ \forall k\neq l
\end{array}
\right\}
\end{equation}
and that this, in turn, can be approximated by a solution to a competitive system of $k_1+\ldots+ k_m$ equations, where competition occurs between groups of components. More precisely, let the penalized energy functional $E_\beta:\R^{k_1}\times\ldots \times \R^{k_m}\to \R$ defined as
\begin{align*}
E_\beta(u^1,\ldots,u^m)&= \sum_{i=1}^m \psi_i\left(\int_\Omega |\nabla u_1^i|^2\,dx,\ldots, \int_\Omega |\nabla u^i_{k_i}|^2\, dx\right)\\
					&+\mathop{\sum_{i,j=1}^m}_{i<j} \frac{2\beta}{q}\int_\Omega \left((u_1^i)^2+\ldots+ (u_{k_i}^i)^2 \right)^\frac{q}{2}\left((u_1^j)^2+\ldots+ (u_{k_j}^j)^2\right)^\frac{q}{2}\, dx
\end{align*}
where $2<2q<2^\ast$, and consider the energy level
\begin{equation}\label{eq:c_beta_INTRO}
c_\beta:=\inf \left\{ E_\beta(u^1,\ldots, u^m) :\
\begin{array}{c}
u^i=(u^i_1,\ldots, u^i_{k_i})\in H^1_0(\Omega;\R^{k_i}) \\[6pt]
\displaystyle \text{and, for every $i$, } \int_\Omega u^i_k u^i_l\, dx=\delta_{kl}\ \forall k,l,\\[6pt]
\displaystyle  \int_\Omega \nabla u^i_k \cdot \nabla u^i_l\, dx=0\ \forall k\neq l
\end{array}
\right\}.
\end{equation}

Our third and last main result reads as follows.

\begin{theorem}\label{thm:3rdmain}
Under the previous notations, the following holds.
\begin{enumerate}
\item[(i)] Given $\beta>0$ there exists $(u^1_\beta,\ldots, u_\beta^m)$ achieving $c_\beta$. Moreover, for some $\mu_{jl,\beta}^i>0$, $i=1,\ldots, m$, $j,l=1,\ldots, k_i$, it is a solution of the system
\begin{equation}\label{eq:system_system}
-a^i_{l,\beta}\Delta u^i_{l}=\sum_{j=1}^{k_i} \mu_{jl,\beta}^i u^i_{l}-\beta u^i_{l}\Bigl((u^i_{1})^2+\ldots+(u^i_{k_i})^2\Bigr)^{\frac{q}{2}-1}\mathop{\sum_{j\neq i}^m}_{j=1} \Bigl( (u^j_{1})^2+\ldots+(u^j_{k_j})^2\Bigr)^\frac{q}{2},
\end{equation}
with 
\[
a^i_{l,\beta}=\frac{\partial \psi_i}{\partial \xi_l}\left(\int_\Omega |\nabla u_{1,\beta}^i|^2\,dx,\ldots, \int_\Omega |\nabla u^i_{k_i,\beta}|^2\, dx\right)>0.
\]
\item[(ii)] There exists $(\bar u^1,\ldots, \bar u^m)$, with Lipschitz continuous components, such that
\[
u^i_\beta\to \bar u^i \qquad \text{ strongly in } H^1_0\cap C^{0,\gamma}(\overline \Omega; \R^{k_i})\ \forall 0<\gamma<1,\quad \text{ as } \beta\to \infty;
\]
and
\begin{equation}\label{eq:Eigenfunctions}
-a^i_l\Delta \bar u^i_l=\mu^i_l \bar u^i_l \qquad \text{ in } \left\{\sum_{j=1}^{k_i}  (\bar u_j^i)^2>0\right\}=\interior \left(\, \overline{\left\{\sum_{j=1}^{k_i}  (\bar u_j^i)^2>0\right\}}\, \right).
\end{equation}
for $a^i_l=\lim_{\beta\to\infty} a^i_{l,\beta}>0$ and $\mu_l^i=\lim_{\beta\to \infty} \mu_{ll,\beta}^i>0$.
\item[(iii)] The vector $(\bar u^1,\ldots, \bar u^m)$ achieves \eqref{eq:c_infty_INTRO}, and such level is the limit of $c_\beta$ as $\beta\to \infty$;
\item[(iv)] The partition 
\[
(\bar \omega_1,\ldots, \bar \omega_m):=\left(\left\{\sum_{j=1}^{k_1} (\bar u_j^1)^2>0 \right\},\ldots,\left\{ \sum_{j=1}^{k_m} (\bar u_j^m)^2>0\right\}    \right)\in\Peh_m(\Omega)
\]
is regular, and solves the optimal partition problem \eqref{eq:OPPsecondary}.
\item[(v)] For $a^i_l=\lim_{\beta\to \infty} a^i_{l,\beta}$, we have that, whenever $x_0$ belongs to the regular part of the common nodal set of $(\bar u^1,\ldots, \bar u^m)$, if $\bar \omega_i, \bar \omega_j$ are the corresponding adjacent sets, then
\begin{equation}\label{eq:extremality_for_general_problem}
\mathop{\lim_{x\to x_0}}_{x\in \bar \omega_i} \sum_{n=1}^{k_i} a_{n}^i |\nabla u^i_n(x)|^2=\mathop{\lim_{x\to x_0}}_{x\in \bar \omega_j} \sum_{n=1}^{k_j} a_n^j |\nabla u^j_n(x)|^2\neq 0
\end{equation}
\end{enumerate}
\end{theorem}
Observe that the existence of minimizers in \eqref{eq:c_infty_INTRO} and \eqref{eq:c_beta_INTRO} is not a trivial fact (unlike in \eqref{eq:equivalent_characterization_of_sum_lambda1}) due to restriction on the gradients $\int_\Omega \nabla u^i_k \cdot \nabla u^i_l\, dx=0$, $k\neq l$. We will be able to remove this condition by extending the $\psi_i$'s to symmetric  functions defined on the space of self-adjoint matrices. The coefficients $a^i_n$, which appear in the extremality condition \eqref{eq:extremality_for_general_problem} for the approximated problem, are related to the chosen functions $\psi_i$, thus are related to \eqref{eq:general_OPP_with_psi}. We recall that for the choice 
\begin{equation}\label{eq:label_p}
\psi_{i,p}:(\R_0^+)^{k_i}\to \R,\qquad \psi_{i,p}(\xi_1,\ldots, \xi_{k_i}):=\left(\sum_{j=1}^{k_i} (\xi_j)^p \right)^{1/p}
\end{equation}
we are dealing with \eqref{eq:OPPsecondary} and thus, with an approximation of \eqref{eq:OPPmain}. It turns out that the coefficients in the extremality condition \ref{eq:extremality_main} of Theorem \ref{thm:first_main_result} come out from the double approximation as being a limit as $\beta\to +\infty$, followed from a limit as $p\to +\infty$, of the coefficients provided in the previous theorem.

\bigskip

Finally, let us now come to other possible applications of our strategy and future directions of research. Recall that in \cite{HHT} spectral minimal partitions like
\begin{equation}\label{eq:OPP_auxiliarylast}
\inf_{(\omega_1,\ldots,\omega_m)\in P_m(\Omega)} \max_{i=1,\ldots, m}\{ \lambda_1(\omega_i)\}
\end{equation}
were introduced (see the survey \cite{helffer} for a description of the subject). We believe our approach can be applied to the following generalization of \eqref{eq:OPP_auxiliarylast} 
\[
\inf_{(\omega_1,\ldots,\omega_m)\in P_m(\Omega)} \max_{i=1,\ldots, m}\{ \lambda_{k_i}(\omega_i)\}, \qquad \text{ for } k_1,\ldots, k_m\in \N,
\]
(which is not contained in the assumptions of Theorem \ref{thm:first_main_result}) by taking the following approximating problem (for $p$ and $q$ large):
\[
\inf_{(\omega_1,\ldots,\omega_m)\in P_m(\Omega)} \left(\sum_{i=1}^m \left(\sum_{j=1}^{k_i}(\lambda_{j}(\omega_i))^q\right)^{p/q}\right)^{1/p}.
\]
This, in turn, can also be well approximated by constrained minimal energy solutions of a competitive Schr\"odinger system. 

Another interesting question, that is being tackled at the moment \cite{TavaresZilio}, is to prove that, in \eqref{eq:OPPmain}, \emph{every} optimal partition has a regular \emph{representative} (in the sense of a class of equivalence involving the capacity of sets). This was proved in \cite{HHT} for \eqref{eq:OPP_auxiliarylast} by using a penalization technique (check p.111 therein for more details), and it seems that the strategy followed there can be successfully adapted and combined with our approach for \eqref{eq:OPPmain}.

\bigskip

\begin{remark} To simplify the notation, we will present the proofs of the theorems for the simpler case $k_1=\ldots, k_m=:k$, and deal with partitions with $m=2$ components. It will become clear that the general case follows exactly in the same way, without any extra work.
\end{remark}

\noindent\textbf{Notation. } We will denote the $H^1_0$--norm simply by $\|\cdot \|$, namely $\|u\|^2:=\int_\Omega |\nabla u|^2\, dx$.

The paper is organized as follows. 
\tableofcontents

\section{A general existence result. An approximated problem depending on $p$}\label{sec:Existence}

\subsection{Preliminaries}\label{subsec:preliminaries}

As previously mentioned, to simplify the presentation we will work from now on with $m=2$ and $k_1=k_2=:k\in \N$.
Define $\Seh_k(\R)$ and $\Oeh_k(\R)$ as the spaces of symmetric and orthogonal $k\times k$ matrices with real entries, respectively. In order to deal with  \eqref{eq:OPPsecondary} and \eqref{eq:general_OPP_with_psi}, we will work with a special class of functions with domain in $\Seh_k(\R)$.

\begin{definition}\label{def:symmetric_function}
We say that $\vphi\in \Feh$ if $\vphi:\Seh_k(\R)\to \R$ is  $C^1$ in $\Seh_k(\R)\setminus\{0\}$ and
$$
\vphi(M)=\vphi(P^TMP)\qquad \text{ for all }M\in \Seh_k(\R) \text{ and }P\in \Oeh_k(\R).
$$
Moreover, consider  the restriction $\psi$ of $\vphi$ to the space of diagonal matrices, that is $\psi(a_1,\ldots, a_k):=\vphi(\diag(a_1,\ldots, a_k))$. Then we ask that 
\begin{itemize}
\item[-] $\frac{\partial \psi}{\partial a_i}>0$ on $(\R^+)^k$ for every $i=1,\ldots, k$; 
\item[-] for each $i$ and $\bar a_1,\ldots, \bar a_{i-1},\bar a_{i+1},\ldots,\bar a_k>0$, we have 
\[
\psi(\bar a_1,\ldots, \bar a_{i-1},a_i,\bar a_{i+1},\ldots,\bar a_k)\to +\infty \quad \text{ as } a_i\to +\infty.
\]
\end{itemize}
\end{definition}

\begin{remark}
Along this paper, whenever we say that $\vphi:\Seh_k(\R)\to \R$ is a $C^1$ function, we are identifying $\Seh_k(\R)$ with $\R^{k(k+1)/2}$. The partial derivative $\frac{\partial \vphi}{\partial \xi_{ij}}$ with $i\geq j$ denotes the derivative of $\vphi$ along the direction of the symmetric matrix $A$ having all entries zero except for $a_{ij}=a_{ji}=1$.
\end{remark}
\noindent {\bf Examples.} Functions as in Definition \ref{def:symmetric_function} just depend on the set of the eigenvalues of a matrix. To fix our mind, let us think of the examples $\vphi(M)=\left(\text{trace}(M^p)\right)^{1/p}$ for $p\in \N$, or $\vphi(M)=\text{det}(M)$. In these cases, $\psi(a_1,\ldots, a_k)=\left(\sum_{i=1}^k a_i^p\right)^{1/p}$ and $\psi(a_1,\ldots, a_k)=\prod_{i=1}^k a_i$, respectively. We will use the first of these examples to treat \eqref{eq:OPPsecondary}.

\medbreak

The main aim of this section is to prove the following result.

\begin{theorem}\label{thm:main_general} Given $\vphi\in \Feh$, the level
\begin{align}\label{eq:OPPgeneral}
c_\infty:&=\inf_{(\omega_1,\omega_2)\in \Peh_2(\Omega)}  \sum_{i=1}^2 \vphi\left( \diag(\lambda_1(\omega_i),\ldots,\lambda_k(\omega_i))\right)\\
	&=\inf_{(\omega_1,\omega_2)\in \Peh_2(\Omega)}  \sum_{i=1}^2 \psi(\lambda_1(\omega_i),\ldots,\lambda_k(\omega_i))
\end{align}
is achieved.
\end{theorem}

\begin{corollary}\label{coro:corolario_do_main_existence}
In particular, going through the previous examples, we have that the optimal partition problems
$$
\inf_{(\omega_1,\omega_2)\in \Peh_2(\Omega)} \sum_{i=1}^2\left( \sum_{j=1}^k\lambda_j(\omega_i)^p\right)^{1/p}, \qquad p\in \N;
$$
and
$$
\inf_{(\omega_1,\omega_2)\in \Peh_2(\Omega)} \sum_{i=1}^2 \Bigl(\prod_{j=1}^k \lambda_j(\omega_i)\Bigr)
$$
admit a solution.
\end{corollary}

We will solve problem \eqref{eq:OPPgeneral} by approaching it with a penalized problem having an additional  parameter (recall system \eqref{eq:system_system} in the Introduction). More precisely, the proof of Theorem \ref{thm:main_general} in this section will imply Theorem \ref{thm:3rdmain}, (i)-(ii)-(iii), as well as the existence part of (iv). The proof of the remaining statements will be completed in Section \ref{sec:Regularity_of_nodal_set}.

 At this point it is useful to recall that, if $\{v_n\}_{n\in \N}$ represents a sequence of $L^2$--orthogonal eigenfunctions of $(-\Delta,H^1_0(\Omega))$, ordered according to the nondecreasing sequence of eigenvalues, then
\[
\lambda_k(\Omega)=\inf\left\{ \int_\Omega |\nabla v|^2:\ \int_\Omega \nabla v\cdot \nabla v_i\, dx=0,\ i=1,\ldots, k-1\right\}.
\]
Given $u\in H^1_0(\Omega; \R^k)$, define the $k\times k$ symmetric matrix
\begin{equation}\label{eq:M(u)}
M(u)=\Bigl( \int_\Omega \nabla u_i\cdot \nabla u_j\, dx \Bigr)_{i,j=1,\ldots, k}.
\end{equation}

Fix $2<2q<2^\ast$, where $2^\ast=+\infty$ if $N\leq 2$, $2^\ast=2N/(N-2)$ whenever $N\geq 3$. For each $\beta>0$, we define the energy $E_\beta:H^1_0(\Omega; \R^k)\times H^1_0(\Omega;\R^k)\to \R$ as
$$
E_\beta(u,v)=\vphi(M(u))+\vphi(M(v))+\frac{2\beta}{q}\int_\Omega (u_1^2+\ldots +u_k^2)^\frac{q}{2}(v_1^2+\ldots + v_k^2)^\frac{q}{2}\, dx
$$
and consider the energy levels
$$
c_\beta:=\inf \left\{ E_\beta(u,v): u,v\in \Sigma(L^2) \right\},
$$
where
$$
\Sigma(L^2):=\Bigl\{ w=(w_1,\ldots, w_k)\in H^1_0(\Omega; \R^k):\ \int_\Omega w_i w_j\, dx=\delta_{ij} \text{ for every }i,j \Bigr\}
$$
($\delta_{ij}$ denotes the Kronecker symbol). In order to prove that $c_\beta$ is achieved, we need three auxiliary lemmas.
\begin{lemma}\label{lemma_changeofbasis}
Let $\vphi:\Seh_k(\R)\to \R$ be a map satisfying
$$
\vphi(M)=\vphi(P^TMP)\qquad \text{ for all }M\in \Seh_k(\R) \text{ and }P\in \Oeh_k(\R).
$$
Given $u\in \Sigma(L^2)$, let $M(u)$ be defined as in \eqref{eq:M(u)}. Then there exists $\tilde u\in \Sigma(L^2)$ such that 
$$
M(\tilde u)=\diag\left(\| \tilde u_1\|^2,\ldots, \|\tilde u_k\|^2 \right),
$$
and $\vphi(M(u))=\vphi(M(\tilde u))$; moreover, $\sum_{i=1}^k u_i^2=\sum_{i=1}^k \tilde u_i^2$ pointwise. In particular, $E_\beta(\tilde u,\tilde v)=E_\beta(u,v)$.

\end{lemma}
\begin{proof}
 Let $P$ be an orthogonal matrix such $P^T M(u) P$ is diagonal.  Consider the new elements
$$
\left[ \begin{array}{c} \tilde u_1 \\ \vdots \\ \tilde u_k \end{array}\right] = P^T \left[ \begin{array}{c} u_1 \\ \vdots \\ u_k \end{array}\right]\quad \Leftrightarrow \quad \tilde u_i=\sum_{j=1}^k p_{ji}u_j,\; i=1,\ldots, k.
$$
Then, given any bilinear form $b:H^1_0(\Omega)\times H^1_0(\Omega)\to \R$ and its associated matrix $B(u)=(b(u_i,u_j))_{ij}$, we have
$$
b(\tilde u_i,\tilde u_j)=b\Bigl( \sum_{l=1}^k p_{li}u_l,\sum_{m=1}^k p_{mj} u_m  \Bigr)=\sum_{l,m=1}^k p_{li} p_{mj} b(u_l, u_m)=(P^TB(u)P)_{ij}.
$$
If one takes $b(w_1,w_2):=\int_\Omega w_1 w_2\, dx$, then as $u\in \Sigma(L^2)$ we have $B(u)=Id$ and
$$
\int_\Omega \tilde u_i \tilde u_j\, dx=(P^T P)_{ij}=\delta_{ij},
$$ 
 hence $\tilde u\in \Sigma(L^2)$. If one takes $b(w_1,w_2)=\int_\Omega \nabla w_1\cdot \nabla w_2\, dx$ instead, then
$$
(M(\tilde u))_{ij}=(P^T M(u)P)_{ij},
$$  hence $M(\tilde u)$ is a diagonal matrix.

As for the final statement of the lemma, we have
\begin{eqnarray*}
\sum_{i=1}^k \tilde u_i^2   &=&	\sum_{i=1}^k \sum_{l,m=1}^k p_{li} p_{mi} u_l u_m =\sum_{l,m=1}^k u_l u_m \sum_{i=1}^k p_{li} p_{mi}\\
				& =&	\sum_{l,m=1}^k u_l u_m (PP^T)_{lm}=\sum_{l,m=1}^k u_l u_m  \delta_{lm}=\sum_{l=1}^k u_l^2.
\end{eqnarray*}
\end{proof}

This lemma implies that $c_\beta$ has the equivalent characterization shown in the Introduction - see \eqref{eq:c_beta_INTRO} - namely that it is the infimum of $E_\beta(u,v)$ for 
\[
u,v\in \Sigma(L^2)\quad \text{ with } \quad \int_\Omega \nabla u_i\cdot \nabla u_j\, dx=\int_\Omega \nabla v_i\cdot \nabla v_j\, dx=0\ \forall i\neq j.
\]
It also clarifies the special shape of the competition term in the energy. 

For such $(u,v)$, observe that
\[
E_\beta(u,v)\geq 2\psi(0,\ldots,0),\qquad \text{ and so } c_\beta \text{ is finite}.
\]

\begin{lemma}\label{lemma:derivative_matrix}
Let $\vphi:\Seh_k(\R)\to \R$ be a $C^1$ function such that
$$
\vphi(M)=\vphi(P^TMP)\qquad \text{ for all }M\in \Seh_k(\R) \text{ and }P\in \Oeh_k(\R).
$$
Then $\frac{\partial \vphi}{\partial \xi_{ij}}(D)=0$ for every diagonal matrix $D$, and $j> i$.
\end{lemma}
\begin{proof}
Let us check it for instance for $i=1\neq 2=j$. First let  $D=\text{diag}(d_1,\ldots, d_k)$ be  a matrix such that $d_1\neq d_2$. Consider $P(t)\in \Oeh_k(\R)$ given by
$$
P(t)=\left(\begin{array}{ccccc}
\cos t & -\sin t &0 &\ldots & 0\\
\sin t & \cos t & 0 & \ldots & 0\\
0      &   0      & 1  & \ldots & 0\\
\vdots & \vdots & \vdots & \ddots & 0\\ 
0      &   0      & 0  & \ldots & 1
\end{array}\right)
$$
which satisfies $P(0)=Id$ and 
$$
P'(0)=\left(\begin{array}{ccccc}
0  & -1 & 0 &\ldots & 0\\
1 & 0 & 0 & \ldots & 0\\
0      &   0      & 0  & \ldots & 0\\
\vdots & \vdots & \vdots & \ddots & 0\\ 
0      &   0      & 0  & \ldots & 0
\end{array}\right).
$$
Then
\begin{equation*}
\frac{d}{dt} (P(t)^T D P(t))|_{t=0} = P'(0)^T D P(0) + P(0)^T D P'(0)=\left(\begin{array}{ccccc}
0  & d_2-d_1 & 0 &\ldots & 0\\
d_2-d_1 & 0 & 0 & \ldots & 0\\
0      &   0      & 0  & \ldots & 0\\
\vdots & \vdots & \vdots & \ddots & 0\\ 
0      &   0      & 0  & \ldots & 0
\end{array}\right)
\end{equation*}
and since by assumption the map $t\mapsto \vphi(P(t)^T D P(t))$ is constant, 
$$
\frac{\partial \vphi}{\partial \xi_{12}}(D) (d_2-d_1)=0, \quad \text{ and } \quad \frac{\partial \vphi}{\partial \xi_{12}}(D)=0 \text{ since $d_1\neq d_2$}.
$$
As the set of diagonal matrices satisfying $d_1\neq d_2$ is obviously dense among the set of all diagonal matrices (identified with $\R^k$),  since $\vphi$ is $C^1$ we obtain that
$\frac{\partial \vphi}{\partial \xi_{12}}(D)=0$ for every diagonal matrix $D$. The proof in the other cases is analogous, by suitably choosing the family of matrices $P(t)$.
\end{proof}

Let us check that $\Sigma(L^2)$ is a manifold.

\begin{lemma}\label{lemma:Sigma(L^2)manifold}
Let $G:H^1_0(\Omega;\R^k)\to \R^{\frac{k(k+1)}{2}}$ defined by
$$
G(w)=(g_{ij}(w))_{j\geq i}:=\left(\int_\Omega w_i w_j\, dx\right)_{j\geq i},
$$ so that $\Sigma(L^2)=\{w\in H^1_0(\Omega;\R^k):\ g_{ij}(w)=\delta_{ij}\ \forall j\geq i\}$. Then $G'(u)$ is onto for each $u\in \Sigma(L^2)$.  In particular, the set $\Sigma(L^2)$ is a submanifold of $H^1_0(\Omega; \R^k)$ of codimension $k(k+1)/2$.
\end{lemma}

\begin{proof}

Given $u\in \Sigma(L^2)$, we have
$$
G'(u)(\vphi_1,\ldots,\vphi_k)=\left(\int_\Omega (\vphi_m u_n+u_m\vphi_n)\right)_{n\geq m}=(g'_{mn}(u)(\vphi_1,\ldots,\vphi_k))_{n\geq m}.
$$
In order to prove that $G'(u)$ is onto, let us prove that every vector of the canonical basis of $\R^\frac{k(k+1)}{2}$ is the image through $G'(u)$ of an element in $H^1_0(\Omega;\R^k)$. Fix $i,j\in \{1,\ldots, k\}$ with $j\geq i$; then by taking $\bar \vphi:=(\bar \vphi_1,\ldots, \bar \vphi_k)$ with $\bar \vphi_j=u_i$, $\bar \vphi_l=0$ for $l\neq j$, we obtain
$$
g'_{ij}(u)\bar \vphi=\int_\Omega u_i^2\, dx=1\; \text{ (if $i\neq j$)},\qquad \text{ or }\qquad g'_{ij}(u)\bar \vphi=\int_\Omega 2 u_i^2\, dx=2 \; \text{ (if $i= j$)},
$$
and
$$
g'_{mn}(u)\bar \vphi=0 \qquad  \forall (m,n)\neq (i,j), \ n\geq m.
$$
Thus $G'(u)$ is onto and the result follows.
\end{proof}

Let us now fix once and for all some $\vphi\in \Feh$, and denote by $\psi$ its restriction to the space of $k\times k$ diagonal matrices.

\begin{theorem}\label{thm:minimizer_for_c_beta}
Given $\beta>0$, the infimum $c_\beta$ is attained at $u_\beta,v_\beta\in \Sigma(L^2)$ such that 
\begin{equation}\label{eq:orthogonal_in_H^1}
\int_\Omega \nabla u_{i,\beta}\cdot \nabla u_{j,\beta}\, dx=\int_\Omega \nabla v_{i,\beta}\cdot \nabla v_{j,\beta}\, dx=0 \quad \text{ whenever } i\neq j.
\end{equation}  
Moreover, for each $i$ we have
\begin{equation}\label{eq:equation_for_u_beta_v_beta}
\left\{\begin{array}{l}
-a_{i,\beta}\Delta u_{i,\beta}=\sum_{j=1}^k \mu_{ij,\beta} u_{j,\beta}-\beta u_{i,\beta}\Bigl(\sum_{j=1}^ku_{j,\beta}^2\Bigr)^{\frac{q}{2}-1}\Bigl(\sum_{j=1}^k v_{j,\beta}^2\Bigr)^\frac{q}{2} \\[10pt]
-b_{i,\beta}\Delta v_{i,\beta}= \sum_{j=1}^k \nu_{ij,\beta} v_{j,\beta}-\beta v_{i,\beta}
\Bigl(\sum_{j=1}^k v_{j,\beta}^2\Bigr)^{\frac{q}{2}-1}
\Bigl(\sum_{j=1}^k u_{j,\beta}^2\Bigr)^{\frac{q}{2}}  
\end{array}\right.
\end{equation}
with 
\[
a_{i,\beta}=\frac{\partial \vphi}{\partial \xi_{ii}}(M(u_{\beta}))=\frac{\partial \psi}{\partial a_i}(M(u_\beta)),\qquad  b_{i,\beta}=\frac{\partial \vphi}{\partial \xi_{ii}}(M(v_{\beta}))=\frac{\partial \psi}{\partial a_i}(M(v_\beta)),
\] 
and
\begin{align}\label{eq:expression_for_mu_ij}
\mu_{ij,\beta}&=\mu_{ji,\beta}=\delta_{ij}a_{i,\beta} \int_\Omega |\nabla u_{i,\beta}|^2\, dx+\beta\int_\Omega u_{i,\beta} u_{j,\beta}\Bigl(\sum_{l=1}^k u_{l,\beta}^2\Bigr)^{\frac{q}{2}-1}\Bigl(\sum_{l=1}^k v_{l,\beta}^2\Bigr)^\frac{q}{2}\,dx,\\
\nu_{ij,\beta}&=\nu_{ji,\beta}=\delta_{ij}b_{i,\beta} \int_\Omega |\nabla v_{i,\beta}|^2\, dx+\beta\int_\Omega v_{i,\beta} v_{j,\beta}\Bigl(\sum_{l=1}^k v_{l,\beta}^2\Bigr)^{\frac{q}{2}-1}\Bigl(\sum_{l=1}^k u_{l,\beta}^2\Bigr)^\frac{q}{2}\,dx.
\end{align}
\end{theorem}

\begin{remark}
We observe that, although the quantity $q/2-1$ needs not to be positive, there occur no singularities in the  expressions above. In fact, for every $u\in H^1_0(\Omega; \R^k)$ we have that
$$
\int_\Omega |u_i u_j|\Bigl(\sum_{l=1}^k u^2_{l}\Bigr)^{\frac{q}{2}-1}\Bigl(\sum_{l=1}^k v_l^2\Bigr)^\frac{q}{2}\, dx\leq \int_\Omega \Bigl(\sum_{l=1}^k u_{l}^2\Bigr)^{\frac{q}{2}}\Bigl(\sum_{l=1}^k v_l^2\Bigr)^\frac{q}{2}\, dx<\infty,
$$
as $2q<2^\ast$.
\end{remark}

\begin{proof}[Proof of Theorem \ref{thm:minimizer_for_c_beta}]
Fix $\beta>0$ and let us take a minimizing sequence $u_n=(u_{1,n},\ldots, u_{k,n})$, $v_n=(v_{1,n},\ldots, v_{k,n})$. By the Ekeland's variational principle we can suppose that $u_n,v_n\in \Sigma(L^2)$ is such that
$$
E_\beta(u_n,v_n)\to c_{\beta} \qquad \text{ and } \qquad E_\beta|_{\Sigma(L^2)\times \Sigma(L^2)}'(u_n,v_n)\to 0 \qquad \text{ as } n\to \infty.
$$
(taking also in consideration Lemma \ref{lemma:Sigma(L^2)manifold}).

Consider  the functions $\tilde u_n$, $\tilde v_n$ defined in Lemma \ref{lemma_changeofbasis}. As $E_\beta(u,v)=E_\beta(Pu,Qv)$ for every $u,v\in H^1_0(\Omega; \R^k)$, $P,Q\in \Oeh_k(\R)$, and $Pu\in \Sigma(L^2)$ whenever $u\in \Sigma(L^2)$, then 
\begin{equation}\label{eq:constrained_PS_sequence}
E_\beta(\tilde u_n,\tilde v_n)=E_\beta(u_n,v_n)\to c_\beta \qquad \text{ and } \qquad E_\beta|_{\Sigma(L^2)\times \Sigma(L^2)}'(\tilde u_n,\tilde v_n)\to 0.
\end{equation}
The first convergence implies that
\begin{eqnarray*}
E_\beta(\tilde u_n,\tilde v_n)&=&\psi\Bigl( \| \tilde u_{1,n}\|^2 ,\ldots, \| \tilde u_{k,n}\|^2 \Bigr)+\psi\Bigl(\| \tilde v_{1,n}\|^2,\ldots, \int_\Omega \| \tilde v_{k,n}\|^2\Bigr)\\
				&& + \frac{2 \beta}{p}\int_\Omega\Bigl(\sum_{j=1}^k \tilde u_{j,n}^2\Bigr)^\frac{q}{2}\Bigl(\sum_{j=1}^k \tilde v_{j,n}\Bigr)^\frac{q}{2}\, dx\leq c_\beta+1, \qquad \forall n\text{ large}.
\end{eqnarray*}
Since $\int_\Omega \tilde u_{i,n}^2\, dx=1$, we must have $\int_\Omega |\nabla \tilde u_{i,n}|^2\, dx\geq \delta>0$ for all $n$, and thus since $\psi\to +\infty$ as $s_i\to +\infty$ for some $i$, we deduce that $\int_\Omega |\nabla \tilde u_{i,n}|^2\, dx \leq C$ for all $i$ and $n$. Analogously  $\int_\Omega |\nabla \tilde v_{i,n}|^2\, dx\leq C$ and hence there exists $(\tilde u,\tilde v)\in \Sigma(L^2)$ such that, up to a subsequence,
$$
\tilde u_{i,n}\rightharpoonup \tilde u_i,\quad \tilde v_{i,n}\rightharpoonup \tilde v_i \qquad \text{ weakly in $H^1_0(\Omega)$, strongly in $L^r(\Omega)$ ($1\leq r<2^\ast$), a.e $x\in \Omega$.}
$$
The second condition in \eqref{eq:constrained_PS_sequence} together with Lemma \ref{lemma:derivative_matrix} implies the existence of $\tilde \mu_{ij,n}$, $\tilde \nu_{ij,n}$ such that
\begin{align*}
-\frac{\partial \vphi}{\partial \xi_{ii}}(M(\tilde u_n))\Delta \tilde u_{i,n}&=\sum_{j=1}^k \tilde \mu_{ij,n} \tilde u_{j,n}-\beta \tilde u_{i,n} \Bigl(\sum_{j=1}^k \tilde u_{j,n}^2\Bigr)^{\frac{q}{2}-1}    \Bigl(\sum_{j=1}^k \tilde v_{j,n}^2\Bigr)^\frac{q}{2} + {\rm o}_n(1)\\
-\frac{\partial \vphi}{\partial \xi_{ii}}(M(\tilde v_n))\Delta \tilde v_{i,n}&=\sum_{j=1}^k \tilde \nu_{ij,n} \tilde v_{j,n}-\beta \tilde v_{i,n} \Bigl(\sum_{j=1}^k \tilde v_{j,n}^2\Bigr)^{\frac{q}{2}-1}    \Bigl(\sum_{j=1}^k \tilde u_{j,n}^2\Bigr)^\frac{q}{2}+ {\rm o}_n(1)
\end{align*}
in $H^{-1}(\Omega)$. Observe that
$$
\tilde \mu_{ij,n}=\delta_{ij} \frac{\partial \vphi}{\partial \xi_{ii}}(M(\tilde u_n))\int_\Omega |\nabla \tilde u_{i,n}|^2 \, dx+\beta\int_\Omega \tilde u_{i,n}\tilde u_{j,n} \Bigl(\sum_{l=1}^k \tilde u_{l,n}^2\Bigr)^{\frac{q}{2}-1}\Bigl(\sum_{l=1}^k \tilde v_{l,n}^2\Bigr)^\frac{q}{2}\, dx+{\rm o}(\|\tilde u_{j,n}\|)
$$
is bounded independently of $n$. Thus, by  multiplying the $i$--th equation for $\tilde u_n$ by $\tilde u_{i,n}-\tilde u_i$ we obtain
$$
\frac{\partial \vphi}{\partial \xi_{ii}}(M(\tilde u_n))\int_\Omega \nabla \tilde u_{i,n}\cdot \nabla (\tilde u_{i,n}-\tilde u_{i})\, dx={\rm o}_n(1).
$$
As $\Bigl(\int_\Omega |\nabla \tilde u_{1,n}|^2\, dx,\ldots,\int_\Omega |\nabla \tilde u_{k,n}|^2\, dx\Bigr)$ lies in a compact set not containing the origin and $\frac{\partial \psi}{\partial a_{i}}>0$ on $(\R^+)^k$, we have that $\frac{\partial \vphi}{\partial \xi_{ii}}(M(\tilde u_n))\geq \delta >0$ for all $i$ and $n$. Hence $\tilde u_{i,n}\to \tilde u_i$ strongly in $H^1_0(\Omega)$ and, analogously, $\tilde v_{i,n}\to \tilde v_i$. In conclusion, for any $i\neq j$ we have
$$
\int_\Omega \nabla \tilde u_i\cdot \nabla \tilde u_j\, dx=\lim_n \int_\Omega \nabla \tilde u_{i,n}\cdot \nabla \tilde u_{j,n}\, dx=0,\qquad \int_\Omega \nabla \tilde v_i\cdot \nabla \tilde v_j\, dx=\lim_n \int_\Omega \nabla \tilde v_{i,n}\cdot \nabla \tilde v_{j,n}\, dx=0,
$$
and $E_\beta(\tilde u_i,\tilde v_i)=c_\beta$. The last paragraph of the statement is a consequence of the Lagrange multiplier rule, by recalling once again Lemma \ref{lemma:Sigma(L^2)manifold}.
\end{proof}

Observe that \eqref{eq:equation_for_u_beta_v_beta} is a type of competitive system where competition occurs between groups of components. In fact, if one thinks of each $u_i$ and $v_i$ as being a population density, then no competition occurs between different $u_i$'s, or between different $v_i$'s, but each element of $\{u_1,\ldots, u_k\}$ competes with each element of $\{v_1,\ldots,v_k\}$. A Lotka-Volterra (non variational) type competition of this kind was considered in the paper by Caffarelli and F.-H. Lin \cite{CaffarelliLinJAMS}.

\subsection{Uniform bounds in H\"older spaces. Proof of the general existence theorem}\label{subsec:Holderbounds}
Our next goal is to pass to the limit $\beta\to \infty$ in \eqref{eq:equation_for_u_beta_v_beta}, relating the limiting profiles with problem $c_\infty$. Having this in mind, first we establish some uniform bounds (in $\beta)$ in the $L^\infty$--norm, and then in H\"older spaces. Throughout this subsection, we take $(u_\beta,v_\beta)$ as in Theorem \ref{thm:minimizer_for_c_beta}.

\begin{lemma}\label{lemma:uniform_bounds_in_beta}
We have $c_\beta\leq c_\infty$ for every $\beta>0$. In particular this implies that $\|(u_\beta,v_\beta)\|_{H^1_0(\Omega)}$ and $\|(u_\beta,v_\beta)\|_{L^\infty(\Omega)}$ are bounded independently of $\beta$.
\end{lemma}
\begin{proof}
1. Given $(\omega_1,\omega_2)\in \Peh_2(\Omega)$, consider the first $k$ eigenfunctions $\bar u=(\bar u_1,\ldots, \bar u_k)$  of $-\Delta$ in $H^1_0(\omega_1)$, orthonormal in $L^2$, and  the first $k$ $L^2$-orthonormal eigenfunctions $\bar v=(\bar v_1,\ldots, \bar v_k)$ of $-\Delta$ in $H^1_0(\omega_2)$, so that $\bar u,\bar v\in \Sigma(L^2)$. Then, since $\omega_1\cap \omega_2=\emptyset$,
$$
\psi(\lambda_1(\omega_1),\ldots, \lambda_k(\omega_1))+\psi(\lambda_1(\omega_2),\ldots, \lambda_k(\omega_2))=\vphi(M(\bar u))+\vphi(M(\bar v))=E_\beta(\bar u,\bar v)\geq c_\beta
$$
 for all  $\beta>0$, hence $c_\beta\leq c_\infty$.

\medbreak

\noindent 2. As $\displaystyle \int_\Omega \beta (\sum_{j=1}^k u_{j,\beta}^2)^\frac{q}{2}(\sum_{j=1}^k v_{j,\beta}^2)^\frac{q}{2}\, dx\geq 0$ and $\psi|_{(\R^+)^k}\geq \psi(0,\ldots,0)$, we obtain that
\[
\psi\left(\| u_{1,\beta}\|^2,\ldots,\| u_{k,\beta}\|^2\right),\quad \psi\left(\| v_{1,\beta}\|^2,\ldots,\| v_{k,\beta}\|^2\right)
\]
and also
\[
\int_\Omega \beta \Bigl(\sum_{j=1}^k u_{j,\beta}^2\Bigr)^\frac{q}{2}\Bigl(\sum_{j=1}^k v_{j,\beta}^2\Bigr)^\frac{q}{2}\, dx
\]
are bounded from above, independently of $\beta$. Since $u_\beta,v_\beta\in \Sigma(L^2)$, we have $\|u_{i,\beta}\|^2, \|v_{i,\beta}\|^2\geq \delta>0$, and by the assumptions made for $\psi$ we obtain uniform $H^1_0$ bounds for $u_\beta,v_\beta$.

\medbreak

\noindent 3. To prove the $L^\infty$ bounds, observe first of all that the conclusions of the previous point imply the existence of $\bar \mu$ such that $|\mu_{ij,\beta}|\leq \bar \mu$ for every $i,j$ and $\beta$, where $\mu_{ij,\beta}$ are the Lagrange multipliers defined in \eqref{eq:expression_for_mu_ij}. Moreover, there exists $\underline{a}>0$ such that $a_{i,\beta}\geq \underline{a}$. If we combine equation \eqref{eq:equation_for_u_beta_v_beta} with the Kato inequality \footnote{The Kato inequality (see for instance \cite[Th. 21.19]{willem}) states that given $u\in L^1_{\rm loc}$ with $\Delta u\in L^1_{\rm loc}$, then $\text{sign}(u)\Delta u\leq \Delta |u|$.} we deduce that
\[
\begin{split}
-\Delta |u_{i,\beta}| &\leq -\text{sign}(u_{i,\beta}) \Delta u_{i,\beta} \\
		& =\sum_{j=1}^k \frac{\mu_{ij,\beta}}{a_{i,\beta}}\text{sign}(u_{i,\beta}) u_{j,\beta} -\frac{\beta}{a_{i,\beta}} \text{sign}(u_{i,\beta}) u_{i,\beta} \Bigl(\sum_{j=1}^k u_{j,\beta}^2\Bigr)^{\frac{q}{2}-1}\Bigl(\sum_{j=1}^k v_{j,\beta}^2\Bigr)^\frac{q}{2}\\
		& \leq \frac{\bar \mu}{\underline a}\sum_{j=1}^k|u_{j,\beta}|.
\end{split}
\]
Summing up in $i$ yields
$$
-\Delta \Bigl(\sum_{i=1}^k |u_{i,\beta}|\Bigr)\leq \frac{k\bar \mu}{\underline{a}}\sum_{i=1}^k |u_{i,\beta}|.
$$
Since a similar statement holds for $\sum_{i=1}^k |v_{i,\beta}|$, we can conclude by using a standard Brezis-Kato type argument together with the $H^1_0$ bounds.
\end{proof}

\begin{theorem}\label{thm:convergences}
There exists $(u,v)\in H^1_0(\Omega; \R^k)\cap C^{0,\alpha}(\overline{\Omega};\R^k)$ ($\forall 0<\alpha<1$) such that, up to a subsequence, as $\beta\to +\infty$,
\begin{itemize}
\item[(i)] $u_\beta\to u$, $v_\beta\to v$ in $C^{0,\alpha}(\overline \Omega)\cap H^1_0(\Omega)$;
\item[(ii)] $u,v\in \Sigma(L^2)$ and $\int_\Omega \nabla u_i \cdot \nabla u_j\, dx=\int_\Omega \nabla v_i\cdot \nabla v_j\, dx=0\ \forall i\neq j$. Moreover,
\[
\int_\Omega \beta (\sum_{j=1}^k u_{j,\beta}^2)^\frac{q}{2}(\sum_{j=1}^k v_{j,\beta}^2)^\frac{q}{2}\, dx\to 0 \quad \text{ and }\quad  u_i\cdot v_j\equiv 0\text{ in } \Omega\ \forall i,j.
\]
\item[(iii)] Furthermore,
\[
  \begin{split}
    &-a_i \Delta u_i=\mu_i u_i \quad \text{ in  } \omega_u:=\{x\in \Omega: u_1^2+\ldots+u_k^2>0\},\\ 
    &-b_i \Delta v_i=\nu_i v_i \; \,\quad \text{ in  } \omega_v:=\{x\in \Omega: v_1^2+\ldots+v_k^2>0\}
\end{split}
\]
for $a_i=\lim_\beta a_{i,\beta}$, $b_i=\lim_{\beta}b_{i,\beta}$, $\mu_i=\lim_\beta \mu_{ii,\beta}$, $\nu_i=\lim_\beta \nu_{ii,\beta}$. Finally, for every $  i=1,\dots,k, \; x_0\in \Omega,\, r\in (0,\dist(x_0,\partial\Omega))$,
$$a_i\int_{B_r(x_0)}|\nabla u_i|^2\, dx=a_i\int_{\partial B_r(x_0)}u_i\, \partial_nu_i\, d\sigma+\mu_i \int_{B_r(x_0)} u_i^2\, dx$$ 
and 
$$b_i\int_{B_r(x_0)}|\nabla v_i|^2\, dx=b_i\int_{\partial B_r(x_0)}v_i\, \partial_nv_i\, d\sigma+\nu_i \int_{B_r(x_0)} v_i^2\, dx.$$    
\end{itemize}
\end{theorem}

\begin{proof}  

(i) The proof of this point is made by contradiction. Namely, we assume that
$$
L_{\beta}:=\max_{i,j=1,\ldots,k}\left\{  \max_{x,y\in \overline{\Omega}}\frac{|u_{j,\beta}(x)-u_{j,\beta}(y)|}{|x-y|^{\alpha}},
\max_{x,y\in \overline{\Omega}}\frac{|v_{i,\beta}(x)-v_{i,\beta}(y)|}{|x-y|^{\alpha}} \right\} \to +\infty
$$
and will reach a contradiction. The argument is very close to the one in \cite[Section 3]{NTTV1} (there the authors deal with positive solutions and moreover $q=2$, $k=1$ and $a_{i,\beta}=b_{i,\beta}=1$) and therefore we merely focus on the technical differences with respect to \cite{NTTV1}. 

\medbreak

\noindent \emph{Step A.} We may assume that $L_{\beta}$ is achieved by the component  $u_{1,\beta}$, say at the pair $(x_{\beta},y_{\beta})$ with $|x_{\beta}-y_{\beta}|\to 0$. Accordingly, for a given sequence $r_{\beta}\to 0$ to be chosen later on, we define the rescaled functions
$$\bar{u}_{i,\beta}(x)=\frac{1}{L_{\beta}r_{\beta}^{\alpha}} \; u_{i,\beta}(x_{\beta}+r_{\beta}x), 
\qquad \bar{v}_{i,\beta}(x)=\frac{1}{L_{\beta}r_{\beta}^{\alpha}}\;  v_{i,\beta}(x_{\beta}+r_{\beta}x),$$
for $x\in \Omega_{\beta}:=\frac{\Omega-x_{\beta}}{r_{\beta}}$. It follows that
\begin{equation}\label{eq:equation_for_bar_u_beta_bar_v_beta}
\left\{\begin{array}{l}
-a_{i,\beta}\Delta \bar{u}_{i,\beta}=\sum_{j=1}^k r_{\beta}^2\mu_{ij,\beta} \bar{u}_{j,\beta}-\beta M_{\beta} \bar{u}_{i,\beta}\Bigl(\sum_{j=1}^k\bar{u}_{j,\beta}^2\Bigr)^{\frac{q}{2}-1}\Bigl(\sum_{j=1}^k \bar{v}_{j,\beta}^2\Bigr)^\frac{q}{2} \\[10pt]
-b_{i,\beta}\Delta \bar{v}_{i,\beta}= \sum_{j=1}^k r_{\beta}^2\nu_{ij,\beta} \bar{v}_{j,\beta}-\beta M_{\beta} \bar{v}_{i,\beta}
\Bigl(\sum_{j=1}^k \bar{v}_{j,\beta}^2\Bigr)^{\frac{q}{2}-1} \Bigl(\sum_{j=1}^k \bar{u}_{j,\beta}^2\Bigr)^{\frac{q}{2}}  
\end{array}\right.
\end{equation}
with
$$
M_{\beta}=L_{\beta}^{2q-2}r_{\beta}^{2\alpha(q-1)+2}, \qquad \delta \leq a_{i,\beta},\; b_{i,\beta}\leq C, \quad  |\mu_{ij,\beta}|, \; |\nu_{ij,\beta}| \leq C.$$
We observe that
$$
r_{\beta}^2 \| \bar{u}_{j,\beta}\|_{L^{\infty}(\Omega_{\beta})}+r_{\beta}^2\| \bar{v}_{j,\beta}\|_{L^{\infty}(\Omega_{\beta})}\to 0$$
and also that
\begin{multline*}
\max_{i,j=1,\ldots,k}\left\{  \max_{x,y\in \overline{\Omega}_{\beta}}
\frac{|\bar{u}_{j,\beta}(x)-\bar{u}_{j,\beta}(y)|}{|x-y|^{\alpha}},
\max_{x,y\in \overline{\Omega}_{\beta}}\frac{|\bar{v}_{i,\beta}(x)-\bar{v}_{i,\beta}(y)|}{|x-y|^{\alpha}} \right\} \\
	=\frac{|\bar{u}_{1,\beta}(0)-\bar{u}_{1,\beta}(\frac{y_{\beta}-x_{\beta}}{r_{\beta}})|}{| \frac{y_{\beta}-x_{\beta}}{r_{\beta}} |^{\alpha}}=1.
\end{multline*}
We next prove two claims which correspond to Lemmas 3.4 and 3.5 in \cite{NTTV1}.

\medbreak

{\sc Claim 1.} {\it Whenever $r_{\beta}$ is such that
\begin{equation}\label{Claim1}
\liminf_{\beta\to\infty} \beta M_{\beta}>0 \qquad \mbox{ and } \qquad \limsup_{\beta\to\infty}\frac{|x_{\beta}-y_{\beta}|}{r_{\beta}}<\infty
\end{equation}
then both sequences
$$
d_{\beta}:=\sum_{i=1}^{k} \bar{u}^2_{i,\beta}(0)\qquad \mbox{ and } \qquad e_{\beta}:=\sum_{i=1}^{k} \bar{v}^2_{i,\beta}(0)$$
are bounded as $\beta\to \infty$.}

Let $R\geq |x_\beta-y_\beta|/r_\beta$ and assume by contradiction that $d_\beta+e_\beta\to +\infty$. For sufficiently large $\beta$, one has $B_{2R}(0)\subseteq \Omega_\beta$. Let $\eta$ a smooth cutoff function such that $\eta=1$ in $B_R(0)$ and $\eta=0$ in $\R^N\setminus B_{2R}(0)$. By testing the first equation in \eqref{eq:equation_for_bar_u_beta_bar_v_beta} with $\bar{u}_{i,\beta}\eta^2$ and the second one with $\bar{v}_{i,\beta}\eta^2$, we find out that
$$
\beta M_\beta\int_{B_{R}(0)} \left( \sum_{i=1}^k \bar{u}_{i,\beta}^2\right)^{q/2}\left( \sum_{i=1}^k \bar{v}_{i,\beta}^2\right)^{q/2}\leq C \int_{B_{2R}(0)} \sum_{i=1}^k \bar{u}_{i,\beta}^2 \; dx
$$
and
\begin{equation}\label{eq:claim1_aux1}
\beta M_\beta \int_{B_{R}(0)} \left( \sum_{i=1}^k \bar{u}_{i,\beta}^2\right)^{q/2}\left( \sum_{i=1}^k \bar{v}_{i,\beta}^2\right)^{q/2}\leq C \int_{B_{2R}(0)} \sum_{i=1}^k \bar{v}_{i,\beta}^2 \; dx.
\end{equation}
Since $\bar{u}_{i,\beta}$ and $\bar{v}_{i,\beta}$  are H\"older continuous and \eqref{Claim1} holds, there exist $C_1,C_2>0$ (depending only on $R>0$) such that
\[
\left( \sum_{i=1}^k \bar{u}_{i,\beta}^2(x)\right)^{q/2}\left( \sum_{i=1}^k \bar{v}_{i,\beta}^2(x)\right)^{q/2}dx\leq C_1 \left( \sum_{i=1}^k \bar{u}_{i,\beta}^2(x)\right)+C_2
\]
and
\begin{equation}\label{eq:claim1_aux2}
\left( \sum_{i=1}^k \bar{u}_{i,\beta}^2(x)\right)^{q/2}\left( \sum_{i=1}^k \bar{v}_{i,\beta}^2(x)\right)^{q/2}dx\leq C_1 \left( \sum_{i=1}^k \bar{v}_{i,\beta}^2(x)\right)+C_2.
\end{equation}
for every $x\in B_R(0)$. In particular,
\[
(d_{\beta}\, e_{\beta})^{q}\leq (C_1 d_\beta+C_2)(C_1e_\beta+C_3).
\]
Since $q>1$, we cannot have that \emph{both} sequences are unbounded.

Assume first of all that $d_\beta\to +\infty$ (up to a subsequence), and $e_\beta$ is bounded, so that in particular
\[
\inf_{B_{2R}(0)} \sum_{i=1}^k \bar u_{i,\beta}^2 \to +\infty,\quad \text{ and } \quad \sup_{B_{2R}(0)}\sum_{i=1}^k \bar v_{i,\beta}^2 \text{ is bounded,}
\]
and actually
\begin{equation}\label{eq:sup_goes_to_0}
\sup_{B_{2R}(0)}\sum_{i=1}^k \bar v_{i,\beta}^2\to 0 \quad \text{ as }\beta\to \infty.
\end{equation}
Define
\[
I_\beta= \beta M_b \inf_{B_{2R}} \left(\sum_{i=1}^k \bar u_{i,\beta}^2\right)^\frac{q}{2}\to +\infty.
\]
If $q\geq 2$, we have
\[
- a_{i,\beta} \Delta |\bar v_{i,\beta}|\leq {\rm o}_{\beta}(1)- |\bar v_{i,\beta}|^{q-1} I_\beta.
\]
Thus the decay estimate \cite[Lemma 2.2]{SoaveZilio} gives $I_\beta \sup_{B_R} |\bar v_{i,\beta}|^{q-1}={\rm o}_\beta(1)$ as $\beta\to \infty$, for every $i=1,\ldots, k$. This implies that
 \begin{equation}\label{eq:Laplacian_to_0}
 \|\Delta \bar{u}_{i,\beta}\|_{L^\infty(B_R(0))}\to 0 \qquad \text{ for every large $R$}.
 \end{equation}
 In particular, the sequence $\tilde{u}_{\beta}(x):=\bar{u}_{1,\beta}(x)-\bar{u}_{1,\beta}(0)$ converges uniformly in compact sets to  a non constant, harmonic,  H\"older continuous function in $\Omega_{\infty}$, where $\Omega_{\infty}$  is either $\R^N $ or a half-space of $\R^N$. Arguing as in \cite[p. 282]{NTTV1}, this contradicts the Liouville type theorem stated in \cite{NTTV1}. If on the other hand $q<2$, from \eqref{eq:sup_goes_to_0},
\[
-\Delta \left( \sum_{j=1}^k |\bar v_{j,\beta}|\right)\leq C \sum_{j=1}^k |\bar v_{j,\beta}|-C I_\beta \sum_{j=1}^k |\bar v_{j,\beta}|^{q-1}\leq -\widetilde C I_\beta \sum_{j=1}^k |\bar v_{j,\beta}|
\]
and so \cite[Lemma 4.4]{ContiTerraciniVerziniAsymptotic} gives $\sum_{j=1}^k |\bar v_{j,\beta}|\leq C_1 e^{-C_2 \sqrt{I_\beta}}$ in $B_R(0)$, and so \eqref{eq:Laplacian_to_0}, again a contradiction. In conclusion, we have shown that $(d_{\beta})_{\beta} $ is bounded. The case when $(e_\beta)_\beta$ can be handled in an analogous way, and thus the claim is established. 

\medbreak

{\sc Claim 2.} {\it We have that}
\[
\beta L_{\beta}^{2q-2}|x_{\beta}-y_{\beta}|^{2\alpha(q-1)+2}\to \infty.
\]

Indeed, assuming the contrary, by letting
\[
r_{\beta}^{-1}=\left( \beta L_{\beta}^{2q-2}\right)^{2\alpha(q-1)+2}\to \infty
\]
we find ourselves in the situation displayed in \eqref{Claim1}, with $\beta M_{\beta}=1$. According to Claim 1 and up to subsequences, we conclude that $\bar{u}_{i,\beta}\to u_{i,\infty}$,  $\bar{v}_{i,\beta}\to v_{i,\infty}$ uniformly in compact sets, and
$$
\left\{\begin{array}{l}
-a_{i}\Delta {u}_{i,\infty}=- {u}_{i,\infty}\Bigl(\sum_{j=1}^k{u}_{j,\infty}^2\Bigr)^{\frac{q}{2}-1}\Bigl(\sum_{j=1}^k {v}_{j,\infty}^2\Bigr)^\frac{q}{2} \\[10pt]
-b_{i}\Delta {v}_{i,\infty}=- {v}_{i,\infty}
\Bigl(\sum_{j=1}^k {v}_{j,\infty}^2\Bigr)^{\frac{q}{2}-1} \Bigl(\sum_{j=1}^k {u}_{j,\infty}^2\Bigr)^{\frac{q}{2}}.
\end{array}\right.
$$
Moreover, all these functions are H\"older continuous and $u_{1,\infty}$ is non constant. As explained in \cite[p. 283]{NTTV1}, the case where the system holds in a half-space is easier to handle with, and so we assume henceforth that this limit problem holds in the whole space $\R^N$. Let us denote $u=\sum_{i=1}^{\infty}|u_{i,\infty}|$, $v=\sum_{i=1}^{\infty}|v_{i,\infty}|.$ Since $a_{i}, b_{i}\geq \delta >0$ $\forall i $ and since
$$
|u|\left(\sum _{j=1}^k u_{j,\infty}^2\right)^{\frac{q}{2}-1}\geq  |u|^{q-1}, \qquad 
\frac{1}{\delta}\left(\sum _{j=1}^k u_{j,\infty}^2\right)^{\frac{q}{2}}\geq c  |u|^{q}, 
$$for some $c>0$ (and similarly for $v$), we find that
\begin{equation}\label{-Delta|u|leq-c|u|^p-1|v|^p}
-\Delta |u|\leq -c|u|^{q-1}|v|^q, \qquad -\Delta |v|\leq -c |v|^{q-1}|u|^q.
\end{equation}
We may assume that $c=1$ by possibly replacing $u$ and $v$ with $c^{1/(2q-2)}u$ and $c^{1/(2q-2)}v$ respectively. Suppose first that $|u|\; |v|\equiv 0$, so that $-\Delta |u|\leq 0$ and $-\Delta |v|\leq 0$. It follows then from \cite[Proposition 2.2]{NTTV1}  that $v=0$ and so $u_{1,\infty}$ is a non constant, harmonic,  H\"older continuous function in $\R^N$, in contradiction with \cite[Corollary 2.3]{NTTV1}. Assume now that $|u(x_0)|\, |v(x_0)| \neq 0$ for some $x_0\in \R^N$; since the problem is invariant by translation, we may assume that $x_0=0$. Then we argue precisely as in the proof of \cite[Proposition 2.6]{NTTV1}   to conclude that again $v=0$, thus finishing the proof of Claim 2.

For the sake of clarity, we explicitly mention the only part of the argument where an extra care with respect to \cite{NTTV1}  is  in order, due to the fact that here possibly $q\neq 2$. The main step in the proof of \cite[Proposition 2.6]{NTTV1} consists in showing (cf. \cite[Lemma 2.5]{NTTV1}) that, given $\varepsilon>0$, the function
\[
J(r):=\frac{1}{r^{4-\varepsilon}}\int_{B_{r}(0)} f(|x|) \left( |\nabla u|^2+|u|^q|v|^q \right)dx\int_{B_{r}(0)} f(|x|) \left( |\nabla v|^2+|u|^q|v|^q\right)dx
\]
associated to \eqref{-Delta|u|leq-c|u|^p-1|v|^p} is increasing for $r$ sufficiently large, where $f$ is a suitable cut-off function. This property, in turn, relies on the study of the quantities
\begin{align*}
\Lambda_1(r)&:=\frac{\int_{\partial B_1(0)}\left( |\nabla_{\theta}u_{(r)}|^2+r^2 |u_{(r)}|^q|v_{(r)}|^q\right) d\sigma
}{\int_{\partial B_1(0)}u^2_{(r)}d\sigma},\\
\Lambda_2(r)&:=\frac{\int_{\partial B_1(0)}\left( |\nabla_{\theta}v_{(r)}|^2+r^2 |u_{(r)}|^q|v_{(r)}|^q\right) d\sigma
}{\int_{\partial B_1(0)}v^2_{(r)}d\sigma},
\end{align*}
where $u_{(r)}(\theta):=u(r\theta)$, $v_{(r)}(\theta):=v(r\theta)$. Arguing as in \cite[p. 275]{NTTV1}, assume that $\Lambda_1(r_n)$ and $\Lambda_2(r_n)  $ are bounded sequences along a sequence $r_n\to \infty$. In particular,
\[
r_n^2\int_{\partial B_1(0)}  |u_{(r_n)}|^q|v_{(r_n)}|^q\, d\sigma \text{ is less than or equal to } C\int_{\partial B_1(0)} u^2_{(r_n)}\, d\sigma\text{ and } C\int_{\partial B_1(0)} v^2_{(r_n)}\, d\sigma.
\]
By multiplying these inequalities we find that
\[
r_n^2\int_{\partial B_1(0)}  |u_{(r_n)}|^q|v_{(r_n)}|^q\, d\sigma \leq C \|u_{(r_n)}\|_{L^2(\partial B_1(0))}
\|v_{(r_n)}\|_{L^2(\partial B_1(0))}.
\]
In particular, since $q>1$ and since both sequences  $ ||u_{(r_n)}||_{L^2(\partial B_1(0))}$ and $ ||v_{(r_n)}||_{L^2(\partial B_1(0))}$ are bounded away from  zero  (this a consequence of the fact that both $|u|$ and $|v| $ are subharmonic functions  such that $ u(0)\neq 0$, $v(0)\neq 0$), we deduce from the previous inequality that
$$
r_n^2\int_{\partial B_1(0)}  |u_{(r_n)}|^q|v_{(r_n)}|^q\, d\sigma \leq C' \|u_{(r_n)}\|^q_{L^2(\partial B_1(0))}
\|v_{(r_n)}\|^q_{L^2(\partial B_1(0))}
$$
for some $C'>0$. By using this estimate we can apply the remaining argument in \cite[p. 275]{NTTV1}  to conclude the proof of Claim 2.

\medbreak

\noindent \emph{Step B.} Once Claims 1 and 2 are established, one proceeds as follows. Let
$$
r_{\beta}=|x_{\beta}-y_{\beta}|.$$
With this choice of $r_{\beta}$ we have that $\beta M_{\beta}\to \infty $ (cf. Claim 2) and so, according to Claim 1 and up to subsequences, we have that $\bar{u}_{i,\beta}\to u_{i,\infty} $ and $\bar{v}_{i,\beta}\to v_{i,\infty} $ uniformly in compact subsets of  $\R^N$, and $u_{1,\infty} $ is non constant. Moreover, by proceeding exactly as in Lemmas 3.6 and 3.7 in \cite{NTTV1}, we deduce that these convergences hold in $H^1_{{\rm loc}}(\R^N)$, $\beta M_{\beta}\left( \sum_{i=1}^k\bar{u}^2_{i,\beta}\right)^{q/2}\left( \sum_{i=1}^k\bar{v}^2_{i,\beta}\right)^{q/2}\to 0$ in $L^1_{{\rm loc}}(\R^N)$,  $\left(\sum_{i=1}^ku^2_{i,\infty}\right)\left(\sum_{i=1}^kv^2_{i,\infty}\right)\equiv 0$, $-a_i\Delta u_{i,\infty}=0$ in $\{x: \sum_{i=1}^k u^2_{i,\infty}(x)>0\}$, and $-b_i\Delta v_{i,\infty}=0$ in $\{x: \sum_{i=1}^k v^2_{i,\infty}(x)>0\}$. 

A contradiction is then reached as in \cite[p. 293]{NTTV1}. We outline the argument. The main tool in reaching such a contradiction is a monotonicity formula applied to the Almgren quotients
$$
N(x_0,r)=\frac{E(x_0,r)}{H(x_0,r)}\qquad (x_0\in \R^N, \; r>0),$$
where, in the present setting,
$$
E(x_0,r):=\frac{1}{r^{N-2}} \int_{B_r(x_0)} \left( \sum_{i=1}^k a_i|\nabla u_{i,\infty}|^2+ \sum_{i=1}^kb_i|\nabla v_{i,\infty}|^2\right)dx$$
and
$$
H(x_0,r):=\frac{1}{r^{N-1}} \int_{\partial B_r(x_0)} \left( \sum_{i=1}^k a_i u^2_{i,\infty}+ \sum_{i=1}^k b_i v^2_{i,\infty}\right)d\sigma.$$
Namely, one can prove (cf. \cite[Proposition 3.9]{NTTV1}) that $N(x_0,r)$ is a non decreasing function and 
$$
\frac{d}{dr}\log(H(x_0,r))=\frac{2N(x_0,r)}{r}\qquad \forall  x_0\in \R^N, \; r>0.$$
Due to the presence of the terms $a_{i}=\frac{\partial \vphi}{\partial \xi_{ii}}(M(u_{\infty}))$ and  $b_{i}=\frac{\partial \vphi}{\partial \xi_{ii}}(M(v_{\infty}))$ (which are absent in \cite{NTTV1}),
this monotonicity formula will be proved in full detail in the next subsection, in a slightly different context which will also be needed later on. Then we  conclude as follows. We denote $U=\sum_{i=1}^ku_{i,\infty}^2$, $V=\sum_{i=1}^kv_{i,\infty}^2$,  so that $U\, V\equiv 0$, $-\Delta U\leq 0$ and $-\Delta V\leq 0$ in $\R^N$. Since $U\neq 0$, $V$ must vanish somewhere in $\R^N$. Also, $U$ must vanish at some point, otherwise $u_{1,\infty}$ would a non constant, harmonic,  H\"older continuous function in $\R^N$, in contradiction in \cite[Corollary 2.3]{NTTV1}. Thus, by continuity, $U$ and $V$ must have a common zero point. Since moreover they are both non negative, sub-harmonic and  H\"older continuous  in $\R^N$, we deduce from \cite[Proposition 2.2]{NTTV1} that
$$
V=0 \qquad \mbox{ and } \qquad \{x: U(x)>0\} \; \mbox{ is a proper connected subset of } \R^N.$$
Now, let $\Gamma:=\{x: U(x)=0\}$. Since $U$ is $\alpha$-H\"older continuous, by using the above mentioned Almgren monotonicity formula one finds out that
$$
N(x_0,r)=\alpha \qquad \forall x_0\in \Gamma, \, r>0,$$
which turns out to be equivalent to the statement that $u_{i,\infty}$ is given in polar coordinates around $x_0$ by  $u_{i,\infty}(x)=r^{\alpha}g_i(\theta)$ for some function $g_i $ ($i=1,\ldots,k$, $r=|x-x_0|$, $\theta=(x-x_0)/r$). Since this procedure can be applied to any given $x_0\in \Gamma$, we conclude that $\Gamma$ is a cone with respect to each of its points, and so $\Gamma$ is a (proper) subspace of $\R^N$. Moreover, $\R^N\setminus \Gamma$ is connected and so $\Gamma$ has dimension at most $N-2$; in particular, it has zero capacity. As a consequence,  $u_{1,\infty}$ is a non constant, harmonic,  H\"older continuous function in the whole of $\R^N$, in contradiction in \cite[Corollary 2.3]{NTTV1}.

This final contradiction establishes the property (i) in Theorem \ref{thm:convergences}. 

\medbreak

\noindent (ii) The proof of this fact is much simpler and follows precisely as in \cite[p. 294]{NTTV1}. 
 
 \medbreak
 
\noindent (iii) We multiply the $i$--th equation for $u$ in \eqref{eq:equation_for_u_beta_v_beta} by $u_i$, integrate by parts and pass to the limit, by observing that, thanks to \eqref{eq:expression_for_mu_ij},   $\mu_{ij,\beta}\to 0$ for $i\neq j$ as $\beta \to \infty$. A full proof of a stronger result will be presented ahead in Proposition \ref{eq:propositionmeasures}.

\end{proof}

The following corollary implies in particular Theorem \ref{thm:main_general}, as well as Theorem \ref{thm:3rdmain} (i)-(ii)-(iii) except the statement of Lipschitz continuity and the equality between sets in (ii).

\begin{corollary}\label{coro:characterization_of_c_infty}
We have the following alternative characterization for $c_\infty$:
$$
c_\infty= \inf \left\{\vphi(M(w))+\vphi(M(z)):\ w,z\in \Sigma(L^2),\ w_i\cdot z_j=0 \ \forall i,j\right\}.
$$
Moreover, for 
$$\omega_u:= \{u_1^2+\ldots+u_k^2>0\}\qquad \mbox{ and } \qquad\omega_v:=\{v_1^2+\ldots+v_k^2>0\},$$
 where $u$ and $v$ are given by Theorem \ref{thm:convergences}, we have that $(\omega_u,\omega_v)\in \Peh_2(\Omega)$ achieves $c_\infty$. In particular, Theorem \ref{thm:main_general} is proved.
\end{corollary}
\begin{proof}
As a result of the strong convergences proved in the previous theorem, we have that $u_i v_j=0\ \forall i,j$, $\int_\Omega \nabla u_i\cdot \nabla u_j=\int_\Omega \nabla v_i \cdot \nabla v_j=0\ \forall i\neq j$ and $u,v\in \Sigma(L^2)$. Thus, by the monotonicity property of $\psi$, and its invariance with respect to exchange of variables,
\begin{align*}
c_\infty&\geq \lim_{\beta\to \infty } c_\beta\\
	&=\lim_{\beta \to \infty}\vphi(M(u_\beta))+\vphi(M(v_\beta))+\frac{2\beta}{p}\int_\Omega (u_{1,\beta}^2+\ldots +u_{k,\beta}^2)^\frac{p}{2}(v_{1,\beta}^2+\ldots + v_{k,\beta}^2)^\frac{p}{2}\, dx\\
	&=\vphi(M(u))+\vphi(M(v))\\
	&= \psi\left(\| u_1\|^2,\ldots, \|u_k\|^2\right)+\psi\left( \|v_1\|^2,\ldots, \|v_k\|^2\right)\\
	&\geq \psi(\lambda_1(\omega_u),\ldots, \lambda_k(\omega_u))+\psi(\lambda_1(\omega_v),\ldots, \lambda_k(\omega_v))\\
	&\geq c_\infty.
\end{align*}
\end{proof}

We end this subsection by showing that, as a consequence of the minimization, the equations in Theorem \ref{thm:convergences}-(iii) for $u_i$ and $v_i$ hold  in the (apparently) larger sets $\widetilde \omega_u:=\Omega\setminus\overline{\omega_v}$, $\widetilde \omega_v:=\Omega\setminus\overline{\omega_u}$. It is easy to check that $\omega_u \subseteq \widetilde \omega_u$, $\omega_v\subseteq \widetilde \omega_v$; later on we will check that these sets coincide (cf. Corollary \ref{coro:One_component_positive}).

\begin{proposition}\label{prop:lemadoaviao}
The following holds for every $i$:
\[
  \begin{split}
    &-a_i \Delta u_i=\mu_i u_i \text{ in the open set }  \widetilde \omega_u,\\ 
    &-b_i \Delta v_i=\nu_i v_i \text{ in the open set } \widetilde \omega_v.
\end{split}
\]
\end{proposition}
\begin{proof}
1. As we will see, this is an immediate consequence of the Lagrange multiplier rule, and it is related to the fact that $(u,v)$ solves a minimization problem. To fix ideas, let us just prove that 
$$
-a_1 \Delta u_1=\mu_1 u_1\qquad \text{ in } \Omega\setminus\overline{\omega_v}.
$$
Consider the functional $J:H^1_0(\tilde w_u)\to \R$ defined by $J(w)=\vphi(M(w,u_2,\ldots, u_k))+\vphi(M(v))$ and the set
$$
\Meh=\left\{w\in H^1_0(\tilde \omega_u): G(w):=\left(\int_{ \tilde \omega_u} w^2\, dx, \int_{ \tilde \omega_u} w u_2\, dx,\ldots, \int_{ \tilde \omega_u} w u_k\, dx\right)=(1,0,\ldots,0)\right\}.
$$
We observe that $\Meh$ is a manifold in a neighborhood of $u$, since $G'(u)$ is onto. Indeed, $G'(u)u_i$ is equal to $2e_i$ if $i=1$, and $G'(u)u_i=e_i$ if $i\geq 2$, where $e_i$ denotes the $i$--th vector of the canonical basis of $\R^k$.

2. The infimum
$$
\inf \{J(w):\ w\in \Meh\}
$$
is attained at $w=u_1$. In fact, suppose in view of a contradiction that there exists $\bar w\in \Meh$ satisfying $J(\bar w)<J(u_1)=\vphi(M(u))+\vphi(M(v))$. Denote by $\tilde w$ the extension of $\bar w$ by 0 to $\Omega$. Then $(\tilde w,u_2,\ldots, u_k)\in \Sigma(L^2)$ and $\tilde w\cdot v_j\equiv 0$ for every $j$, as $\tilde w=0$ on $\{v_1^2+\ldots+v_k^2>0\}$. Thus we obtain a contradiction with the first statement of Corollary \ref{coro:characterization_of_c_infty}.

3. By the Lagrange multipliers rule, there exists $\tilde \mu_1$ such that
$$
-a_1 \Delta u_1=\tilde \mu_1 u_1 \qquad \text{ in }\tilde \omega_u.
$$
Let us now check that $\tilde \mu_1=\mu_1$. Since by passing to the limit as $\beta\to +\infty$ in \eqref{eq:expression_for_mu_ij} yields $a_1\int_\Omega |\nabla u_1|^2\, dx=\mu_1$, the result follows.
\end{proof}

\subsection{Almgren's Monotonicity Formula}\label{subsec:Almgren}

Let $(u,v)$ be the limiting profile provided by Theorem \ref{thm:convergences}, achieving $c_\infty$. Our aim is to continue to characterize this limit. We will prove an Almgren Monotonicity Formula (Theorem \ref{thm:Almgren}), which will imply the Lipschitz continuity of $u$ and $v$, and will have a central role in the proof of the regularity of the free boundary in the next section.
As a first step towards the monotonicity formula, we need to establish some general integral identities, which are the subject of the following lemmas.

\begin{lemma}
Let $H:\R^{2k}\to \R$ be a $C^1$ function such that $|H(s,t)|\leq C(1+|s|+|t|)^{2^*}$ for every $s,t\in \R^k$, and take $l_i,m_i\in \R$, $i=1,\ldots, k$. Let $w=(w_1,\ldots, w_k)$, $z=(z_1,\ldots, z_k)$ be a smooth solution of
\begin{equation}\label{eq:auxiliary_functional}
-l_i\Delta u_i=H_{w_i}(w,z),\qquad -m_i\Delta v_i=H_{z_i}(w,z).
\end{equation}
Then for every vector field $Y\in C^1_c(\Omega; \R^N)$ we have that
\begin{align*}
\sum_{i=1}^k l_i\int_\Omega \left( \langle dY \nabla w_i,\nabla w_i\rangle-\frac{1}{2}\div Y |\nabla w_i|^2\right)\, dx\\
+\sum_{i=1}^k m_i\int_\Omega \left( \langle dY \nabla z_i,\nabla z_i\rangle-\frac{1}{2}\div Y |\nabla z_i|^2\right)\, dx\\ 
+\int_\Omega H(w,z)\, \div Y\, dx=0.
\end{align*}
\end{lemma}
\begin{proof}
Consider a field $Y\in C^1_c(\Omega;\R^N)$, that is such that $Y=0$ ouside of a compact subset of $\Omega$. Multiply the equation for $w_i$ in \eqref{eq:auxiliary_functional} by $ \nabla w_i\cdot Y$ and the one for $z_i$ by $\nabla z_i\cdot Y$. Summing up, one obtains
\begin{align*}
\sum_{i=1}^k l_i \int_\Omega \langle \nabla w_i,\nabla(\nabla w_i\cdot Y)\rangle\, dx+\sum_{i=1}^k m_i\int_\Omega \langle \nabla z_i,\nabla(\nabla z_i\cdot Y)\rangle\, dx\\
-\sum_{i=1}^k \int_\Omega (H_{w_i}(w,z)(\nabla w_i\cdot Y)+H_{z_i}(w,z)(\nabla z_i\cdot Y))\, dx=0.
\end{align*}
Now
\begin{align*}
\int_\Omega \langle \nabla w_i,\nabla(\nabla w_i\cdot Y)\rangle\, dx=&=\int_\Omega \langle \nabla w_i,(Hw_i Y+dY^T\nabla w_i)\rangle\, dx\\
&=\int_\Omega (\langle Hw_i \nabla w_i,Y\rangle+\langle dY\nabla w_i,\nabla w_i \rangle)\, dx\\
&=\int_\Omega \left(\frac{1}{2}\nabla (|\nabla w_i|^2)\cdot Y+\langle dY \nabla w_i,\nabla w_i \rangle\right)\, dx\\
&=\int_\Omega \left(\langle dY \nabla w_i,\nabla w_i\rangle-\frac{1}{2}|\nabla w_i|^2\div Y\right)\, dx,
\end{align*}
where $Hw_i$ denotes the Hessian of $w_i$. Analogously,
\[
\int_\Omega \langle \nabla v_i,\nabla(\nabla v_i\cdot Y)\rangle\, dx=\int_\Omega \left(\langle dY \nabla v_i,\nabla v_i\rangle-\frac{1}{2}|\nabla v_i|^2\div Y\right)\, dx.
\]
Finally,
\begin{align*}
 \sum_{i=1}^k \int_\Omega (H_{w_i}(w,z)(\nabla w_i\cdot Y)+H_{z_i}(w,z)(\nabla z_i\cdot Y))\, dx \\
 = \int_\Omega  \nabla_x (H(w(x),z(x)))\cdot Y(x)\, dx = -\int_\Omega H(w,z) \div Y\, dx.
\end{align*}

\end{proof}

Now let us apply the previous result to $(u,v)$, the limit functions obtained in Theorem \ref{thm:convergences}. Recall the definitions 
\[
a_i=\frac{\partial \vphi}{\partial \xi_{ii}}(M(u_\beta))=\frac{\partial \psi}{\partial a_i}(M(u_\beta)),\qquad  b_i=\frac{\partial \vphi}{\partial \xi_{ii}}(M(v_\beta))=\frac{\partial \psi}{\partial a_i}(M(v_\beta)),
\]
and
\[
\mu_i=\lim_{\beta\to\infty} \mu_{ii,\beta},\qquad \nu_i=\lim_{\beta\to\infty} \nu_{ii,\beta}.
\]

\begin{corollary}\label{coro:Field_first}
Let $(u,v)$ be the limit functions given by Theorem \ref{thm:convergences}. Then
\begin{align*}
\sum_{i=1}^k a_i \int_\Omega \left( \langle dY \nabla u_i,\nabla u_i\rangle-\frac{1}{2}\div Y |\nabla u_i|^2\right)\, dx\\
+\sum_{i=1}^k b_i\int_\Omega \left( \langle dY \nabla v_i,\nabla v_i\rangle-\frac{1}{2}\div Y |\nabla v_i|^2\right)\, dx\\ 
+\sum_{i=1}^k \int_\Omega \frac{1}{2}(\mu_i u_i^2+\nu_i v_i^2)\, \div Y\, dx=0.
\end{align*}
\end{corollary}
\begin{proof}
Since the functions $(u_\beta,v_\beta)$ solve the system \eqref{eq:equation_for_u_beta_v_beta}, we can apply the previous result with:
\begin{align*}
H_\beta(u,v)=&\frac{1}{2}\sum_{i,j=1}^k (\mu_{ij,\beta}u_{i}u_{j}+\nu_{ij,\beta}v_{i}v_{j}) +\frac{\beta}{q}\left(\sum_{j=1}^k u_{j}^2\right)^{\frac{q}{2}}\left(\sum_{j=1}^k v^2_{j}\right)^\frac{q}{2}.
\end{align*}
By Lemma \ref{lemma:derivative_matrix},  for each $\beta>0$ we have
\begin{equation}\label{eq:auxiliary1}
\begin{split}
\sum_{i=1}^k a_{i,\beta} \int_\Omega (\langle dY\nabla u_{i,\beta},\nabla u_{i,\beta}\rangle -\frac{1}{2}\div Y |\nabla u_{i,\beta}|^2)\, dx&\\
 +\sum_{i=1}^k b_{i,\beta} \int_\Omega (\langle dY\nabla v_{i,\beta},\nabla v_{i,\beta}\rangle -\frac{1}{2}\div Y |\nabla v_{i,\beta}|^2)\, dx&\\ 
+\frac{1}{2}\sum_{i,j=1}^k \int_\Omega (\mu_{ij,\beta} u_{i,\beta} u_{j,\beta}+\nu_{ij,\beta}v_{i,\beta}v_{j,\beta})\div Y\, dx& \\
-\frac{\beta}{q}\int_\Omega \left(\sum_{j=1}^k u^2_{j,\beta}\right)^\frac{q}{2}\left(\sum_{j=1}^k v_{j,\beta}^2\right)^\frac{q}{2}\div Y\, dx=0.
\end{split}
\end{equation}
By using the convergences stated in Theorem \ref{thm:convergences}  we can conclude by passing to the limit in \eqref{eq:auxiliary1}.  
\end{proof}

\begin{corollary}[Local Pohozaev--type identities]\label{coro:Field_second}
Given $x_0\in \Omega$ and $r\in (0,\dist(x_0,\partial \Omega))$, we have that
\begin{align*}
(2-N) \sum_{i=1}^k \int_{B_r(x_0)} (a_i|\nabla u_i|^2+b_i |\nabla v_i|^2)\, dx\\
=\sum_{i=1}^k \int_{\partial B_r(x_0)} a_i r(2(\partial_n u_i)^2-|\nabla u_i|^2)\, d\sigma +\sum_{i=1}^k \int_{\partial B_r(x_0)} b_i r (2(\partial_n v_i)^2-|\nabla v_i|^2)\, d\sigma\\
+\sum_{i=1}^k \int_{\partial B_r(x_0)} r(\mu_i u_i^2+\nu_i v_i^2)\, d\sigma-\sum_{i=1}^k \int_{B_r(x_0)} N(\mu_i u_i^2+\nu_i v_i^2)\, dx.
\end{align*}
\end{corollary}

\begin{proof}
Given $x_0\in \Omega$ and $r>0$ we choose $Y=Y_\delta(x)=(x-x_0)\eta_\delta(x)$ where $\eta_\delta$ is a cut-off function such that $\eta_\delta=1$ on $B_r(x_0)$, $\eta_\delta(x)=0$ in $\R^N\setminus B_{r+\delta}(x_0)$. Observe that
$$ \langle dY_\delta \nabla u_i, \nabla u_i\rangle = \langle d(\eta_\delta(x)(x-x_0))\nabla u_i, \nabla u_i\rangle=\eta_\delta |\nabla u_i|^2+(\nabla u_i\cdot \nabla \eta_\delta)\langle \nabla u_i, x-x_0\rangle,
$$
$$
 \div (Y_\delta)=N\eta_\delta + \langle \nabla \eta_\delta, x-x_0\rangle.
 $$
By using Corollary \ref{coro:Field_first}, we have that
\begin{align}\label{eq:auxiliary2}
\sum_{i=1}^k  \int_\Omega \frac{2-N}{2} \eta_\delta (a_i |\nabla u_i|^2+b_i |\nabla v_i|^2)\, dx+\sum_{i=1}^k\int_\Omega \frac{N}{2}(\mu_i u_i^2+\nu_i v_i^2)\, \eta_\delta \, dx\nonumber\\
= \sum_{i=1}^k a_i \int_\Omega \left( \frac{1}{2}|\nabla u_i|^2 \langle \nabla \eta_\delta,x-x_0\rangle-(\nabla u_i\cdot \nabla \eta_\delta)\langle \nabla u_i,x-x_0\rangle \right)\, dx\nonumber\\
+ \sum_{i=1}^k b_i \int_\Omega \left( \frac{1}{2}|\nabla v_i|^2 \langle \nabla \eta_\delta,x-x_0\rangle-(\nabla v_i\cdot \nabla \eta_\delta)\langle \nabla v_i,x-x_0\rangle \right)\, dx \nonumber\\
-\sum_{i=1}^k \int_\Omega \frac{1}{2}(\mu_i u_i^2+\nu_i v_i^2)\langle \nabla \eta_\delta,x-x_0\rangle\, dx.
\end{align}
We can rewrite the right-hand-side as
\begin{align*}
  \sum_{i=1}^k a_i \int_\Omega |\nabla \eta_\delta|\left( \frac{1}{2}|\nabla u_i|^2\left\langle\frac{\nabla \eta_\delta}{|\nabla \eta_\delta|},x-x_0\right\rangle - \left(\nabla u_i\cdot   \frac{\nabla \eta_\delta}{|\nabla \eta_\delta|}\right)\left\langle \nabla u_i,x-x_0\right \rangle\right)\, dx\\
  +\sum_{i=1}^k b_i \int_\Omega |\nabla \eta_\delta|\left( \frac{1}{2}|\nabla v_i|^2\left\langle\frac{\nabla \eta_\delta}{|\nabla \eta_\delta|},x-x_0\right\rangle - \left(\nabla v_i\cdot   \frac{\nabla \eta_\delta}{|\nabla \eta_\delta|}\right)\left\langle \nabla v_i,x-x_0\right \rangle\right)\, dx\\
  -\sum_{i=1}^k \int_\Omega \frac{1}{2}|\nabla \eta_\delta|(\mu_i^2 u_i^2+\nu_i v_i^2)\left\langle \frac{\nabla \eta_\delta}{|\nabla \eta_\delta|},x-x_0\right\rangle\, dx
\end{align*}
which, by the co-area formula, is equal to
\begin{align*}
\sum_{i=1}^k a_i \int_0^1 \int_{\{\eta_\delta=t\}}\left(  \frac{1}{2}|\nabla u_i|^2\left\langle\frac{\nabla \eta_\delta}{|\nabla \eta_\delta|},x-x_0\right\rangle - \left(\nabla u_i\cdot   \frac{\nabla \eta_\delta}{|\nabla \eta_\delta|}\right)\left\langle \nabla u_i,x-x_0\right \rangle\right)\, d\sigma dt\\
+\sum_{i=1}^k b_i \int_0^1 \int_{\{\eta_\delta=t\}}\left(  \frac{1}{2}|\nabla v_i|^2\left\langle\frac{\nabla \eta_\delta}{|\nabla \eta_\delta|},x-x_0\right\rangle - \left(\nabla v_i\cdot   \frac{\nabla \eta_\delta}{|\nabla \eta_\delta|}\right)\left\langle \nabla v_i,x-x_0\right \rangle\right)\, d\sigma dt\\
-\sum_{i=1}^k \int_0^1 \int_{\{\eta_\delta=t\}} \frac{1}{2}(\mu_i u_i^2+\nu_i v_i^2)\left\langle \frac{\nabla \eta_\delta}{|\nabla \eta_\delta|},x-x_0\right\rangle \, d\sigma dt.
\end{align*}
Thus, as $\delta \to 0$, \eqref{eq:auxiliary2} becomes
\begin{align*}
\sum_{i=1}^k \int_{B_r(x_0)} \frac{2-N}{2}(a_i |\nabla u_i|^2+b_i |\nabla v_i|^2)\, dx+\sum_{i=1}^k \int_{B_r(x_0)} \frac{N}{2}(\mu_i u_i^2+\nu_i v_i^2)\, dx\\
=\sum_{i=1}^k a_i \int_{\partial B_r(x_0)} \left( (\nabla u_i\cdot \nu) \langle \nabla u_i,x-x_0\rangle -\frac{1}{2}|\nabla u_i|^2\langle \nu,x-x_0\rangle \right)\, d\sigma\\
+\sum_{i=1}^k b_i \int_{\partial B_r(x_0)} \left( (\nabla v_i\cdot \nu) \langle \nabla v_i,x-x_0\rangle -\frac{1}{2}|\nabla v_i|^2\langle \nu,x-x_0\rangle \right)\, d\sigma\\
+\sum_{i=1}^k \int_{\partial B_r(x_0)} \frac{1}{2} (\mu_i u_i^2+\nu_i v_i^2)\langle \nu,x-x_0\rangle\, d\sigma.
\end{align*}
As $\nu=\frac{x-x_0}{r}$ on $\partial B_r(x_0)$, the right-hand-side of the previous inequality is equal to
\begin{align*}
\sum_{i=1}^k a_i \int_{\partial B_r(x_0)} \left( r (\partial_n u_i)^2-\frac{r}{2}|\nabla u_i|^2\right)\, d\sigma+\sum_{i=1}^k b_i \int_{\partial B_r(x_0)} \left( r (\partial_n v_i)^2-\frac{r}{2}|\nabla v_i|^2\right)\, d\sigma\\
+\sum_{i=1}^k \int_{\partial B_r(x_0)} \frac{r}{2}(\mu_i u_i^2+\nu_i v_i^2)\, d\sigma
\end{align*}
and the conclusion follows.
\end{proof}

Ahead we will only need the local Pohozaev identities of Corollary \ref{coro:Field_second}. We have stated as well the more general ones of Corollary \ref{coro:Field_first}, as we think they are of independent interest (see for instance \cite{DancerWangZhang} for an application).

Under the previous notations, now define the quantities
\[
E(x_0,(u,v),r)=\frac{1}{r^{N-2}} \sum_{i=1}^k \int_{B_r(x_0)}  \left( a_i |\nabla u_i|^2+b_i |\nabla v_i|^2-\mu_i u_i^2-\nu_i v_i^2\right)\, dx,
\]
\[
H(x_0,(u,v),r)=\frac{1}{r^{N-1}} \sum_{i=1}^k \int_{\partial B_r(x_0)}  (a_i u_i^2+b_i v_i^2)\, d\sigma
\]
and the Almgren's quotient by
\[
N(x_0,(u,v),r)=\frac{E(x_0,(u,v),r)}{H(x_0,(u,v),r)},
\]
whenever $H(x_0,(u,v),r)\neq 0$.
\begin{lemma}\label{fracddrE(x_0,(u,v),r)=}
Given $x_0\in \Omega$ and $r\in (0,\dist(x_0,\partial \Omega))$ we have that
\[
\begin{split}
\frac{d}{dr} E(x_0,(u,v),r)=\frac{2}{r^{N-2}} \sum_{i=1}^k \int_{\partial B_r(x_0)} \left( a_i (\partial_n u_i)^2+b_i (\partial_n v_i)^2\right)\, d\sigma\\
					-\frac{2}{r^{N-1}} \sum_{i=1}^k \int_{B_r(x_0)} (\mu_i u_i^2+\nu_i v_i^2)\, dx.
\end{split}
\]
\end{lemma}
\begin{proof}
The derivative of $r\mapsto E(x_0,(u,v),r)$ is given by
\begin{align*}
\frac{2-N}{r^{N-1}}\sum_{i=1}^k \int_{B_r(x_0)} (a_i |\nabla u_i|^2+b_i |\nabla v_i|^2)\, dx-\frac{2-N}{r^{N-1}}\sum_{i=1}^k \int_{B_r(x_0)} (\mu_i u_i^2+\nu_i v_i^2)\, dx\\
+\frac{1}{r^{N-2}}\sum_{i=1}^k\int_{\partial B_r(x_0)} (a_i|\nabla u_i|^2+b_i |\nabla v_i|^2)\, d\sigma-\frac{1}{r^{N-2}}\sum_{i=1}^k \int_{\partial B_r(x_0)} (\mu_i u_i^2+\nu_i v_i^2)\, d\sigma.
\end{align*}
By using Corollary \ref{coro:Field_second}, this quantity is equal to
\begin{align*}
\frac{1}{r^{N-2}}\sum_{i=1}^k \int_{\partial B_r(x_0)} a_i (2(\partial_n u_i)^2-|\nabla u_i|^2)\, d\sigma+\frac{1}{r^{N-2}}\sum_{i=1}^k \int_{\partial B_r(x_0)} b_i (2(\partial_n v_i)^2-|\nabla v_i|^2)\, d\sigma\\
+\frac{1}{r^{N-2}}\sum_{i=1}^k \int_{\partial B_r(x_0)} (\mu_i u_i^2+\nu_i v_i^2)\, d\sigma -\frac{N}{r^{N-1}}\sum_{i=1}^k \int_{B_r(x_0)} (\mu_i u_i^2 +\nu_i v_i^2)\, dx\\
-\frac{2-N}{r^{N-1}}\sum_{i=1}^k \int_{B_r(x_0)} (\mu_i u_i^2 +\nu_i v_i^2)\, dx + \frac{1}{r^{N-2}}\sum_{i=1}^k \int_{\partial B_r(x_0)} (a_i |\nabla u_i|^2+b_i |\nabla v_i|^2)\, d\sigma\\
-\frac{1}{r^{N-2}}\sum_{i=1}^k \int_{\partial B_r(x_0)} (\mu_i v_i^2+\nu_i v_i^2)\, d\sigma
\end{align*}
or, equivalently, to
$$
\frac{2}{r^{N-2}}\sum_{i=1}^k \int_{\partial B_r(x_0)} (a_i (\partial_n u_i)^2+b_i (\partial_n v_i)^2)\, d\sigma-\frac{2}{r^{N-1}}\sum_{i=1}^k \int_{B_r(x_0)}(\mu_i u_i^2+\nu_i v_i^2)\, dx.
$$
\end{proof}

\begin{lemma}\label{fracddrH(x_0,(u,v),r)=}
Given $x_0\in \Omega$ and $r\in (0,\dist(x_0,\partial \Omega))$ we have
$$
\frac{d}{dr}H(x_0,(u,v),r)=\frac{2}{r^{N-1}}  \sum_{i=1}^k \int_{\partial B_r(x_0)} \left(a_i u_i (\partial_n u_i)+ b_i v_i (\partial_n v_i)\right)\, d\sigma.
$$
\end{lemma}
\begin{proof}
After a change of variables,
$$
H(x_0,(u,v),r)=\sum_{i=1}^k \int_{\partial B_1(0)} (a_i u_i^2(x_0+rx)+b_i v_i^2(x_0+rx))\, d\sigma,
$$
whence
\[
\begin{split}
\frac{d}{dr}H(x_0,(u,v),r)=& 2\sum_{i=1}^k\int_{\partial B_1(0)}a_i u_i(x_0+rx)\langle \nabla u_i(x_0+rx),x\rangle\, d\sigma\\
					&+2\sum_{i=1}^k\int_{\partial B_1(0)} b_i v_i(x_0+rx)\langle \nabla v_i(x_0+rx),x\rangle\, d\sigma\\
					=& \frac{2}{r^{N-1}}\sum_{i=1}^k \int_{\partial B_r(x_0)} (a_i u_i (\partial_n u_i)+b_i v_i (\partial_n v_i))\, d\sigma.
\end{split}
\]
\end{proof}

Now we prove an auxiliary technical result.

\begin{lemma}\label{lemma:estimate1}
There exist $C>0$ and $\bar r>0$, both depending only on $N$ and $\|u_i\|, \|v_i\|$, $i=1,\ldots, k$, such that
$$
\frac{1}{r^{N-1}}  \sum_{i=1}^k \int_{B_r(x_0)}   (\mu_i u_i^2+\nu_i v_i^2)\, dx\leq C r (E(x_0,(u,v),r)+H(x_0,(u,v),r))
$$
for every $x_0\in \Omega$, $r\in (0,\bar r)$.
\end{lemma}
\begin{proof}
By making $\beta\to +\infty$ in \eqref{eq:expression_for_mu_ij} we obtain
$$
\mu_i=a_i\int_\Omega |\nabla u_i|^2\, dx\leq \kappa_1 a_i \quad \text{ and } \quad \nu_i=b_i\int_\Omega |\nabla v_i|^2\, dx\leq \kappa_1b_i,
$$
with $\kappa_1$ depending only on $\|u_i\|, \|v_i\|$. Thus we have
\begin{align*}
\frac{1}{r^{N-1}}\sum_{i=1}^k \int_{B_r(x_0)}(\mu_i u_i^2+\nu_i v_i^2)\, dx \leq \frac{\kappa_1}{r^{N-1}}\sum_{i=1}^k \int_{B_r(x_0)} (a_i u_i^2+b_i v_i^2)\, dx\\
	\leq \frac{\kappa_1}{N-1} r \left( \frac{1}{r^{N-2}}\sum_{i=1}^k \int_{B_r(x_0)} (a_i |\nabla u_i|^2+b_i |\nabla v_i|^2)\, dx+\frac{1}{r^{N-1}}\sum_{i=1}^k\int_{\partial B_r(x_0)}(a_i u_i^2 +b_i v_i^2)\, d\sigma\right)
\end{align*}
by the Poincar\'e inequality. Next we  observe that
\begin{equation*}
\begin{split}
\frac{1}{r^{N-2}} \sum_{i=1}^k \int_{B_r(x_0)}  \left( a_i |\nabla u_i|^2 + b_i |\nabla v_i|^2 \right)=E(x_0,(u,v),r)+\frac{1}{r^{N-2}}\sum_{i=1}^k \int_{B_r(x_0)}(\mu_i u_i^2+\nu_i v_i^2)\, dx.
\end{split}
\end{equation*}
Thus for $r$ small enough such that $\frac{\kappa_1}{N-1}  r^2\leq 1/2$ the result follows, with $C=\frac{2\kappa_1}{N-1}$.
\end{proof}

\begin{remark}For later reference, we mention that in the proof of Lemma \ref{lemma:estimate1} we have established the following result:
There exists $\bar r>0$ depending only on $N$ and $\|u_i\|, \|v_i\|$, $i=1,\ldots, k$, such that
\begin{align}\label{eq:estimate1_bis}
\frac{1}{r^{N-2}}\sum_{i=1}^k \int_{B_r(x_0)} (a_i |\nabla u_i|^2+b_i |\nabla v_i|^2)\, dx+\frac{1}{r^{N-1}}\sum_{i=1}^k\int_{\partial B_r(x_0)}(a_i u_i^2 +b_i v_i^2)\, d\sigma \nonumber \\
 \leq 2 (E(x_0,(u,v),r)+H(x_0,(u,v),r)).
\end{align}
for every $x_0\in \Omega$, $r\in (0,\bar r)$.
\end{remark}

We now have all the ingredients to state and prove the monotonicity result.

\begin{theorem}[Almgren's Monotonicity Formula]\label{thm:Almgren}
Let $C=C(N, ||u_i||,||v_i||)$ be given by Lemma \ref{lemma:estimate1}. Then, given $\tilde \Omega\Subset \Omega$, there exists $\tilde r>0$ such that for every $x_0\in \tilde \Omega$ and $r\in (0,\tilde r]$ we have $H(x_0,(u,v),r)\neq 0$, $N(x_0,(u,v),\cdot)$ is absolutely continuous function, and
$$
\frac{d}{dr}N(x_0,(u,v),r)\geq -2Cr(N(x_0,(u,v),r)+1).
$$
In particular, $e^{C r^2}(N(x_0,(u,v),r)+1)$ is a non decreasing function for $r\in (0,\tilde r]$ and the limit $N(x_0,(u,v),0^+):=\lim_{r\to 0^+} N(x_0,(u,v),r)$ exists and is finite. Also,
$$
\frac{d}{dr}\log(H(x_0,(u,v),r))=\frac{2}{r}N(x_0,(u,v),r)\qquad \forall r\in (0,\tilde r).
$$
Moreover, the common nodal set
\[
\Gamma_{(u,v)}:=\{x\in \Omega:\ u_i(x)=v_i(x)=0 \text{ for every }i=1,\ldots,k\}
\]
has no interior points.
\end{theorem}
\begin{proof} We use the same argument as in \cite[Proposition 4.3]{NTTV1}. As a starting point, we observe that Theorem \ref{thm:convergences} (iii) implies that
\begin{equation}\label{second_expression_for_E(r)}
E(x_0,(u,v),r)=\frac{1}{r^{N-2}}\sum_{i=1}^k\int_{\partial B_r(x_0)}\left( a_i u_i(\partial_n u_i)+ b_i v_i (\partial_n v_i)\right)d\sigma.
\end{equation}
Suppose first that $H(x_0,(u,v),r)>0$ in some interval. For simplicity of notations, we write $H(r):=H(x_0,(u,v),r)$ and similarly for $E(r)$ and $N(r)$.
By combining  Lemma \ref{fracddrH(x_0,(u,v),r)=} and \eqref{second_expression_for_E(r)} we find that
\begin{equation}\label{fracddrlogH(r)=frac2rN(r)}
\frac{d}{dr}\log H(r)=\frac{2}{r}N(r).
\end{equation}
Also, by using Lemmas \ref{fracddrE(x_0,(u,v),r)=} and \ref{fracddrH(x_0,(u,v),r)=} together with \eqref{second_expression_for_E(r)}, we deduce that
{\small
\begin{align*}
\frac{d}{dr}N(r)=\frac{2}{r^{2N-3}H^2(r)}\left[ \int_{\partial B_r(x_0)} \sum_{i=1}^k (a_i (\partial_n u_i)^2 + b_i (\partial_n v_i)^2) d\sigma\int_{\partial B_r(x_0)} \sum_{i=1}^k ( a_i u_i^2+ b_i v_i^2 )d\sigma\right.\\
\left.- \left(\int_{\partial B_r(x_0)}\sum_{i=1}^k (a_i u_i(\partial_n u_i)+b_i v_i (\partial_n v_i))\, d\sigma\right)^2\right]
-\frac{2}{H(r)r^{N-1}}\int_{B_r(x_0)}\sum_{i=1}^k (\mu_i u_i^2+\nu_i v_i^2)\, dx. \end{align*}}
As a consequence, thanks to Lemma \ref{lemma:estimate1}, 
 $$
\frac{d}{dr}N(r) \geq -\frac{2}{H(r)r^{N-1}}\int_{B_r(x_0)}\sum_{i=1}^k (\mu_i u_i^2+\nu_i v_i^2)\, dx \geq -2Cr\, \frac{E(r)+H(r)}{H(r)}=-2Cr\, (N(r)+1).$$
Thus, in order to conclude the proof of Theorem \ref{thm:Almgren} we must show that     $H(x_0,(u,v),r)$ is indeed   positive  for small values of $r$ and $x_0\in \tilde \Omega$.

We first prove that $\Gamma_{(u,v)}$ has no interior points. Assume the contrary, and let $x_1\in \Gamma_{(u,v)}$ be such that $d_1:=\dist(x_1,\partial \Gamma_{(u,v)})>0$. Thus $H(x_1,(u,v),r)>0$ for $r\in (d_1,d_1+\varepsilon)$ for some small $\ep$. By what we have done so far, the function $H(r):=H(x_1,(u,v),r) $ satisfies the initial value problem $H'(r)=a(r)H(r)$ for $r\in (d_1,d_1+\ep)$ and $H(d_1)=0$, with $a(r)=2N(r)/r$, which is continuous also at the point $d_1$ by the monotonicity of the map $e^{Cr^2}(N(r)+1)$. Then $H(r)\equiv 0 $ for $r\in (d_1,d_1+\ep)$, a contradiction with the definition of $d_1$. This proves our claim that 
$\Gamma_{(u,v)}$ has empty interior.

Now, suppose $H(x_0,(u,v),r)=0$ for some   $r>0$ and $x_0\in \tilde \Omega$. This means that $u_i=v_j=0$ on $\partial B_{r}(x_0)$ for every $i,j$. By extending a non zero component of $(u,v)$, say $u_1|_{B_r(x_0)}$ by zero to the whole   $\Omega$ (such a non zero component exists, as $\Gamma_{(u,v)}$ has empty interior), call $\tilde{u}_1$ such an extension, since $-\Delta |\tilde{u}_1|\leq \frac{\mu_1}{a_1}|\tilde{u}_1| $ in $\Omega$, we find out that
$$
\frac{\mu_1}{a_1}\int_{B_r(x_0)} u_1^2\, dx\geq \int_{B_r(x_0)}|\nabla u_1|^2dx\geq \lambda_1(B_r(x_0))\int_{B_r(x_0)} u_1^2\, dx.$$
This cannot hold if $r=r(\mu_1,a_1)$ is small enough, uniformly in $x_0\in \tilde \Omega$, and we are done.
\end{proof}

As a consequence of the Almgren's Monotonicity Formula, we have the following.

\begin{corollary}\label{Corollaries 2.6_2.7_2.8_TaTe}
Let $(u,v)$ be as before. Then
\begin{itemize}
\item[1.] given $\tilde \Omega\Subset \Omega$ there exist $\tilde C=\tilde C(\tilde \Omega, N, \|u_i\|,\|v_i\|)$ and $\tilde r>0$ such that, for every $x\in \tilde \Omega$ and $0<r_1<r_2\leq \tilde r$,
$$
H(x,(u,v),r_2)\leq H(x,(u,v),r_1)\left(\frac{r_2}{r_1}\right)^{2\tilde C};
$$
\item[2.] the map $\Omega \mapsto \R$, $x\mapsto N(x,(u,v),0^+)$ is upper semi-continuous;
\item[3.] for every $x_0\in \Gamma_{(u,v)}$ we have $N(x_0,(u,v),0^+)\geq 1$.
\end{itemize}
\end{corollary}
\begin{proof}
The proof of this result follows word by word the ones of Corollaries 2.6, 2.7 and 2.8 in \cite{TavaresTerracini1}.
\end{proof}
The Almgren's monotonicity formula allows us to improve the regularity result for $(u,v)$ and to give and alternative characterization of the nodal set.

\begin{corollary}\label{coro:Lipschiz}
The functions $u$ and $v$ are Lipschitz continuous.
\end{corollary}
\begin{proof}
Once Theorem \ref{thm:Almgren} and Corollary \ref{Corollaries 2.6_2.7_2.8_TaTe} are settled, we can reason exactly as in \cite[Proposition 4.1]{NTTV1} (see also \cite[Subsection 2.4]{TavaresThesis}).
\end{proof}

\begin{corollary}\label{coro:One_component_positive}
If $A$ is a connected component of $\widetilde \omega_u$ (respectively $\widetilde \omega_v$), then there exists a component $u_i$ such that $A\subseteq \{u_i^2>0\}$ (respectively $A\subseteq \{v_i^2> 0\}$).

As a consequence, we have $\omega_u=\widetilde \omega_u$, $\omega_v=\widetilde \omega_v$, and
\[
\begin{split}
\Gamma_{(u,v)}&:=\{x\in \Omega:\ u_i(x)=v_i(x)=0 \ \text{for every }i=1,\ldots, k\}\\
			&=\partial \omega_u\cap \Omega=\partial \omega_v\cap \Omega
\end{split}
\]
\end{corollary}
\begin{proof}
Let $A$ be a connected component of $\widetilde \omega_u$, so that in particular $v_j\equiv 0$ in $A$ for every $j$. Since all $u_i$'s
 are eigenfunctions on the set $A$, and by the minimality provided by Corollary \ref{coro:characterization_of_c_infty},  then either $u_j\equiv 0$ in $A$ for every $j$, or $A\subseteq\{u_i^2>0\}$ for some $i$. The first case cannot occur since $\Gamma_{(u,v)}$ has an empty interior, and thus the first statement of the corollary is proved. All the other statements now follow in a straightforward way.
 \end{proof}
 
We end this subsection with a more detailed description of the equation satisfied by the functions $u_i,v_i$ in the whole $\Omega$.
\begin{proposition}\label{eq:propositionmeasures}
Let $(u,v)$ be the limiting configuration obtained before. Then for each $i\in \{1,\ldots, k\}$ there exist measures $\Mr_i$ and $\Nr_i$, both supported on $ \Gamma_{(u,v)}\footnote{That is, $\Mr_i(A)=\Mr_i(A\cap  \Gamma_{(u,v)})$ and $\Nr_i(A)=\Nr_i(A\cap \Gamma_{(u,v)})$ for each measurable set $A$.}$, such that
\begin{equation}\label{eq:limiting_equations}
-a_i \Delta u_i=\mu_i u_i-\Mr_i,\quad -b_i \Delta v_i=\nu_i v_i-\Nr_i \qquad \text{ in }\mathscr{D}'(\Omega)=(C^\infty_\textrm{c}(\Omega))',\ i=1,\ldots,k.
\end{equation}
\end{proposition}
\begin{proof}By using Kato's inequality in \eqref{eq:equation_for_u_beta_v_beta}, we deduce that
$$
-a_{ij,\beta}\Delta |u_{i,\beta}|\leq \sum_{j=1}^k |\mu_{ij,\beta}||u_{j,\beta}|-\beta |u_{i,\beta}|\Bigl( \sum_{j=1}^k u_{j,\beta}^2 \Bigr)^{\frac{p}{2}-1}\Bigl(\sum_{j=1}^k v_{j,\beta}^2\Bigr)^\frac{p}{2}.
$$
Integrating over $\Omega$, we deduce that
$$
\int_\Omega \beta |u_{i,\beta}|\Bigl( \sum_{j=1}^k u_{j,\beta}^2 \Bigr)^{\frac{p}{2}-1}\Bigl(\sum_{j=1}^k v_{j,\beta}^2\Bigr)^\frac{p}{2}\, dx\leq \int_{\partial \Omega} a_{i,\beta}\partial_n |u_{i,\beta}|\, d\sigma + \sum_{j=1}^k \int_\Omega |\mu_{ij,\beta}||u_{j,\beta}|\, dx\leq C,
$$
where we have used the fact that $\partial_n |u_{i,\beta}|\leq 0$, and $C$ depends only on the uniform bounds of $\|u_{i,\beta}\|,\|v_{i,\beta}\|$, $i=1,\ldots,k$. Thus $\beta u_{i,\beta}\Bigl( \sum_{j=1}^k u_{j,\beta}^2 \Bigr)^{\frac{p}{2}-1}\Bigl(\sum_{j=1}^k v_{j,\beta}^2\Bigr)^\frac{p}{2}$ is bounded in $L^1(\Omega)$ independently of $\beta$, whence there exists $\Mr_i\in \Meh(\Omega)=(C_0(\Omega))'$, a Radon measure, such that
$$
\beta u_{i,\beta}\Bigl( \sum_{j=1}^k u_{j,\beta}^2 \Bigr)^{\frac{p}{2}-1}\Bigl(\sum_{j=1}^k v_{j,\beta}^2\Bigr)^\frac{p}{2}\rightharpoonup \Mr_i \qquad \text{ weak --}\ast \text{ in }\Meh(\Omega).
$$
Reasoning in the same manner with the equation for $v_{i,\beta}$, we obtain the existence of $\Nr_i$ such that
$$
\beta v_{i,\beta}\Bigl( \sum_{j=1}^k v_{j,\beta}^2 \Bigr)^{\frac{p}{2}-1}\Bigl(\sum_{j=1}^k u_{j,\beta}^2\Bigr)^\frac{p}{2}\rightharpoonup \Nr_i \qquad \text{ weak --}\ast \text{ in }\Meh(\Omega).
$$
Thus by Theorem \ref{thm:convergences} we conclude that \eqref{eq:limiting_equations} holds; we recall that, thanks to \eqref{eq:expression_for_mu_ij}, we have that $\mu_{ij,\beta}\to 0$, $\nu_{ij,\beta}\to 0$ for $i\neq j$.

Let $x_0\in \Omega\setminus \Gamma_{(u,v)}$. Then there exists $\delta>0$ such that either $B_\delta(x_0)\subseteq \omega_u$ or $B_\delta(x_0)\subseteq\omega_v$ and the hence the measures are concentrated on $\Gamma_{(u,v)}$.

\end{proof}

\section{Regularity of the nodal set of the approximated solution}\label{sec:Regularity_of_nodal_set}

Along this section, $(u,v)$ will always denote the limiting function obtained in Theorem \ref{thm:convergences}. Our aim is to prove a regularity result for the common nodal set $\Gamma_{(u,v)}$, completing the proof of Theorem \ref{thm:3rdmain}. Although with different proofs, initially we follow the structure of \cite{CaffarelliLinJAMS,TavaresTerracini1}: first we make a detailed study of blowup sequences, and afterwards we use the Almgren's quotient to split the nodal set $\Gamma_{(u,v)}$ and give a first characterization of its regular part. On this subset, locally we prove that locally $\Omega$ is a Reifenberg flat and NTA domain, and that $\Gamma$ separates exactly two different connected components of the partition. Finally through an iterative procedure we construct the normal vector to $\Gamma_{(u,v)}$, proving an extremality condition and a non degeneracy property.

\subsection{Compactness of blowup sequences}\label{subset:Blowupsequences}

In order to understand the local behavior of the functions $(u,v)$ around any given point, let us study the behavior of blowup sequences. To that purpose, let $\tilde \Omega\Subset \Omega$ and let $x_n\in \tilde \Omega$, $t_n\to 0^+$ be a given sequence. We define a blowup sequence by
$$
u_{i,n}(x):=\frac{u_i(x_n+t_n x)}{\rho_n},\quad v_{i,n}(x)=\frac{v_i(x_n+t_n x)}{\rho_n}\qquad \forall\ x\in \Omega_n:=\frac{\Omega-x_n}{t_n},
$$
where we have normalized using the quantity
\[
\begin{split}
\rho_n^2&=H(x_n,(u,v),t_n)=\frac{1}{t_n^{N-1}}\sum_{i=1}^k \int_{\partial B_{t_n}(x_n)} (a_i u_i^2+b_i v_i^2)\, d\sigma\\
	      &= \sum_{i=1}^k \left(a_i \|u_i(x_n+t_n \cdot)\|^2_{L^2(\partial B_1(0))}+b_i \|v_i(x_n+t_n \cdot)\|^2_{L^2(\partial B_1(0))}\right).
\end{split}
\]
We observe that
\begin{equation}\label{H(x_n,(u,v),t_n)=rho_n^2H(0,(u_n,v_n),t_n)}
 H(0,(u_n,v_n),R)=H(x_n, (u,v), t_n R)\rho_n^{-2}\quad \forall R>0,
\end{equation}
\begin{equation}\label{eq:L2_norm_equal_to_1}
\sum_{i=1}^k \left(a_i \|u_{i,n}\|^2_{L^2(\partial B_1(0))}+b_i \|v_{i,n}\|^2_{L^2(\partial B_1(0))}\right)=1
\end{equation}
and, for every $x_0\in \R^N$, $r>0$ (cf. Theorem \ref{thm:convergences} (iii))
\begin{equation}\label{multiply_by_u_i_n}
a_i\int_{B_r(x_0)}|\nabla u_{i,n}|^2\, dx=a_i\int_{\partial B_r(x_0)}u_{i,n}\, \partial_nu_{i,n}\, d\sigma+t_n^2\mu_i \int_{B_r(x_0)} u_{i,n}^2\, dx,
\end{equation}
\begin{equation}\label{multiply_by_v_i_n}
b_i\int_{B_r(x_0)}|\nabla v_{i,n}|^2\, dx=b_i\int_{\partial B_r(x_0)}v_{i,n}\, \partial_nv_{i,n}\, d\sigma+t_n^2\nu_i \int_{B_r(x_0)} v_{i,n}^2\, dx.
\end{equation}
It is clear that
\begin{equation}\label{-a_iDelta_u_,n_leq}
-a_i\Delta u_{i,n}=t_n^2\mu_i u_{i,n}-\Mr_{i,n}, \qquad -b_i\Delta v_{i,n}=t_n^2\nu_i v_{i,n}-\Nr_{i,n} \qquad \text{ in }\mathscr{D}'(\Omega_n),
\end{equation}
where $\Mr_{i,n}$ and $\Nr_{i,n}$ are measures  defined by
\[
\Mr_{i,n}(A)=\frac{1}{\rho_n t_n^{N-2}}\Mr_{i}(x_n+t_n A), \qquad \Nr_{i,n}(A)=\frac{1}{\rho_n t_n^{N-2}}\Nr_{i}(x_n+t_n A)
\]
for every measurable set $A\subseteq \Omega_n$.

\begin{theorem}\label{thm:convergence_of_blowupsequence}
Under the previous notations there exists $(\bar u,\bar v)\not\equiv (0,0)$ such that $\bar u_i\cdot \bar v_j\equiv 0 \text{ in $\R^N$ } \forall i, j$ and, up to a subsequence,
$$
(u_n,v_n)\to (\bar u,\bar v) \quad \text{ in }C^{0,\alpha}_\textrm{loc}(\R^N)\quad \text{ for all }0<\alpha<1,\ \text{strongly in }H^1_\textrm{loc}(\R^N).
$$
Moreover, there exist $\bar \Mr_i,\bar \Nr_i\in \Meh_\textrm{loc}(\R^N)$, both concentrated on  the set
$$
\Gamma_{(\bar u,\bar v)}:=\{x\in \R^N:\ \bar u_i(x)=\bar v_i(x)=0,\ i=1,\ldots, k\},
$$
such that
$$
\Mr_{i,n}\rightharpoonup \bar \Mr_i,\quad \Nr_{i,n}\rightharpoonup \bar \Nr_i \quad \text{ weakly in }\Meh_\textrm{loc}(\R^N), 
$$
and
$$
-a_i \Delta \bar u_i=-\bar \Mr_i,\quad -b_i \Delta \bar v_i=-\bar \Nr_i \quad \text{ in }\mathscr{D}'(\R^N).
$$
In particular,
\begin{equation}\label{multiply_by_baru_i_and_barv_i}
\int_{B_r(x_0)}|\nabla \bar u_{i}|^2\, dx=\int_{\partial B_r(x_0)} \bar u_{i}\, \partial_{n} \bar u_{i}\, d\sigma \quad \mbox{ and } 
\int_{B_r(x_0)}|\nabla \bar v_{i}|^2\, dx=\int_{\partial B_r(x_0)}\bar v_{i}\, \partial_{n}\bar v_{i}\, d\sigma.
\end{equation}
Finally, for each $x_0\in \R^N$ and $r>0$,
\begin{equation}\label{eq:pohozaev_for_bar_uv2}
\begin{split}
(2-N) \sum_{i=1}^k \int_{B_r(x_0)} (a_i|\nabla\bar u_i|^2+b_i |\nabla\bar v_i|^2)\, dx=&\sum_{i=1}^k \int_{\partial B_r(x_0)} a_i r(2(\partial_{n} \bar u_i)^2-|\nabla \bar u_i|^2)\, d\sigma \\
&+\sum_{i=1}^k \int_{\partial B_r(x_0)} b_i r (2(\partial_n \bar v_i)^2-|\nabla \bar v_i|^2)\, d\sigma.
\end{split}
\end{equation}

\end{theorem}

In order to prove Theorem \ref{thm:convergence_of_blowupsequence} we will need to prove first a  technical result. Define
\begin{equation*}
\Gamma_{( u_n, v_n)}:=\{x\in \Omega_n: u_{i,n}(x)=0=v_{i,n}(x) \quad \forall i=1,\ldots,k\}.
\end{equation*}

\begin{lemma} \label{H(x,(u_n,v_n),r)leqCr^2}  Let $\tilde r$ be given by Theorem {\rm \ref{thm:Almgren}}.  Given $R>0$ then, for $n$ sufficiently large,
$$
H(x,(u_n,v_n),r)\leq Cr^2 \qquad \forall 0<r<\tilde r/2, \; x\in B_{2R}(0)\cap \Gamma_{(u_n,v_n)},$$
with $C=C(\tilde \Omega,N,\|u_i\|, \|v_i \|)$.
\end{lemma}

\begin{proof} As long as we have $r<\tilde r$, $t_n<1$ and $x_n+t_nx\in \tilde \Omega$, we have that
$$
\rho_n^2H(x,(u_n,v_n),r)=H(x_n+t_nx, (u,v), rt_n)$$
and so (cf. Theorem \ref{thm:Almgren})
\begin{align}\label{fracddrlogfracrho_n^2}
\frac{d}{dr} \log \frac{\rho_n^2H(x,(u_n,v_n),r)}{r^2}&=t_n \frac{d }{dr}\left[\log  H(x_n+t_nx,(u,v),\cdot)\right]|_{r=t_n x}-\frac{2}{r} \\
&=\frac{2}{r} ( N(x_n+t_n x,(u,v),rt_n)-1),
\end{align}
implying that
$$
\frac{r}{2}\, \frac{d}{dr} \log \frac{\rho_n^2H(x,(u_n,v_n),r)}{r^2}=
( N(x_n+t_n x,(u,v),rt_n)+1)e^{C(rt_n)^2}e^{-C(rt_n)^2}-2.
$$
Now, applying Theorem \ref{thm:Almgren},
$$
( N(x_n+t_n x,(u,v),rt_n)+1)e^{C(t_nr)^2}\geq N(x_n+t_n x,(u,v),0^+)+1.
$$
But since, by assumption, $x_n+t_n x\in \Gamma_{(u,v)}$, it follows from Corollary \ref{Corollaries 2.6_2.7_2.8_TaTe} that
$$
( N(x_n+t_n x,(u,v),rt_n)+1)e^{C(rt_n)^2}\geq 2.
$$
Going back to \eqref{fracddrlogfracrho_n^2}, we conclude that
$$
\frac{d}{dr} \log \frac{\rho_n^2H(x,(u_n,v_n),r)}{r^2}\geq \frac{4}{r} (e^{-C(rt_n)^2}-1)\geq \frac{4}{r}(e^{-C' r^2}-1)\qquad \forall 0<r<\tilde r.
$$
This implies, after an integration, that, for every $0<r<\bar r<\tilde r$,
\[
\frac{H(x,(u_n,v_n),r)}{r^2}\leq \frac{H(x,(u_n,v_n),\bar r)}{\bar r^2}\;  {\rm exp} (\int_r^{\bar r} \frac{4}{s}(1-e^{-C's^2})\, ds )\leq C(\bar r)<\infty,
\]
thus proving the lemma.
\end{proof}

\medbreak

\begin{proof}[Proof of Theorem \ref{thm:convergence_of_blowupsequence}] We follow the argument in the proof of \cite[Theorem 3.3]{TavaresTerracini1}.
\medbreak

\noindent \emph{Step 1.} We first prove that given $R>1$ then, for $n$ sufficiently large,
\[
\|(u_n,v_n)\|_{H^1_0(B_R(0))}\leq C.
\]
Indeed, if $n$ is  so large that $t_nR\leq \tilde r$ (cf. Theorem {\rm \ref{thm:Almgren}}) we have
\[
\begin{split}
\sum_{i=1}^k\int_{\partial B_R(0)}(a_iu_{i,n}^2+b_iv_{i,n}^2)\, d\sigma &=\frac{1}{\rho_n^2t_n^{N-1}} \sum_{i=1}^k\int_{\partial B_{t_nR}(x_n)}(a_iu_{i}^2+b_iv_{i}^2)\, d\sigma\\
					&= R^{N-1}\frac{H(x_n,(u,v),t_nR)}{H(x_n,(u,v),t_n)} \leq R^{N-1} R^{2\tilde C}
\end{split}
\]
(cf. Corollary \ref{Corollaries 2.6_2.7_2.8_TaTe}), yielding that 
\begin{equation}\label{H(0,(u_n,v_n),R)leqC(R)}
H(0,(u_n,v_n),R)\leq C(R).
\end{equation}
On the other hand,
\[
\begin{split}
\frac{1}{R^{N-2}} &\sum_{i=1}^k \int_{B_R(0)}(a_i|\nabla u_{i,n}|^2+b_i|\nabla v_{i,n}|^2)\, dx  \\
 \leq&    \frac{C(R)}{H(0,(u_n,v_n),R)} \frac{1}{R^{N-2}}\sum_{i=1}^k \int_{B_R(0)}(a_i|\nabla u_{i,n}|^2+b_i|\nabla v_{i,n}|^2)\, dx \\
=&\frac{C(R)}{H(0,(u_n,v_n),R)} \left( \frac{1}{R^{N-2}} \sum_{i=1}^k \int_{B_R(0)}(a_i|\nabla u_{i,n}|^2+b_i|\nabla v_{i,n}|^2)\, dx\right. \\
&+  \left. \frac{1}{R^{N-1}}\sum_{i=1}^k \int_{\partial B_R(0)}(a_i u_{i,n}^2+b_i v_{i,n}^2)\, d\sigma \right)-C(R)\\
=&\frac{C(R)\rho_n^{-2}}{H(0,(u_n,v_n),R)} \left( \frac{1}{(t_nR)^{N-2}}\sum_{i=1}^k \int_{B_{t_nR}(x_n)}(a_i|\nabla u_{i}|^2+b_i|\nabla v_{i}|^2)\, dx\right.\\
&+ \left. \frac{1}{(t_nR)^{N-1}}\sum_{i=1}^k \int_{\partial B_{t_nR}(x_n)}(a_i u_{i}^2+b_i v_{i}^2)\, d\sigma \right)-C(R)\\
 \leq & \frac{2C(R)\rho_n^{-2}}{H(0,(u_n,v_n),R)} \Big(  E(x_n,(u,v),t_nR)+H(x_n,(u,v),t_nR) \Big)-C(R)
\end{split}
\]
where we have used \eqref{eq:estimate1_bis} in the last inequality. By taking \eqref{H(x_n,(u,v),t_n)=rho_n^2H(0,(u_n,v_n),t_n)} into account, we conclude that
\[
\begin{split}
\frac{1}{R^{N-2}}&\sum_{i=1}^k \int_{B_R(0)}(a_i|\nabla u_{i,n}|^2+b_i|\nabla v_{i,n}|^2)\, dx\\
& \leq \frac{2C(R)}{H(x_n,(u,v),t_nR)}\Big(  E(x_n,(u,v),t_nR)+H(x_n,(u,v),t_nR) \Big)-C(R)\\
&= 2C(R)(N(x_n,(u,v),t_nR)+1)-C(R)\\
&\leq 2C(R)(N(x_n,(u,v),t_nR)+1)e^{\tilde C (t_n R)^2}-C(R)\\
&\leq 2C(R)(N(x_n,(u,v),\tilde r)+1)e^{\tilde C \tilde r^2}-C(R)\leq C'(R),
\end{split}
\]
where we have used Theorem \ref{thm:Almgren} in the last inequality. The new constant $C'(R)$ depends on $\sup_{y\in \tilde \Omega} |N(y,(u,v),\tilde r)|<\infty$.
\medbreak

\noindent \emph{Step 2.} We next show that given $R>0$ then, for $n$ sufficiently large,
$$
\|(u_n,v_n)\|_{L^{\infty}(B_R(0))}\leq C'.$$
Indeed, it follows from \eqref{-a_iDelta_u_,n_leq} that, by letting $U_n=\sum_{i=1}^k|u_{i,n}|$, then $-\Delta U_n\leq C U_n$. This, together with the estimate in Step 1 and a standard Brezis-Kato  bootstrap argument, proves our claim for  $\|u_n\|_{L^{\infty}(B_R(0))}$ and for $\|v_n \|_{L^{\infty}(B_R(0))}$ as well.

\medbreak

\noindent \emph{Step 3.} We next prove that given $R>0$ the measures $\Mr_{i,n}$ and $\Nr_{i,n}$  are bounded in  $\Meh(B_R(0))$. Let  $\widetilde \Mr_{i,n}$ and $\widetilde \Nr_{i,n}$ be positive measures satisfying
\begin{equation}\label{-a_iDelta|u_,n|leq}
-a_i\Delta |u_{i,n}|\leq t_n^2|\mu_i|\,  |u_{i,n}|-\widetilde \Mr_{i,n}, \qquad -b_i\Delta |v_{i,n}|\leq t_n^2 |\nu_i|\,  |v_{i,n}|- \widetilde \Nr_{i,n} \qquad \text{ in }\mathscr{D}'(\Omega_n)
\end{equation}
and
\begin{equation}\label{comparing_measures_i_n}
|\Mr_{i,n}|\leq \widetilde \Mr_{i,n}, \qquad  |\Nr_{i,n}|\leq  \widetilde\Nr_{i,n}. 
\end{equation}
Multiply the first equation in \eqref{-a_iDelta|u_,n|leq} by a smooth cut-off function $\varphi$ such that $0\leq \varphi\leq 1$, $\varphi=1$  in $B_R(0)$ and $\varphi=0$ in $\R^N\setminus B_{2R}(0)$. It holds
\begin{eqnarray*}
\| \widetilde\Mr_{i,n}\|_{\Meh(B_R(0))}&=&\widetilde \Mr_{i,n}(B_R(0)) \leq  \int_{B_{2R}(0)}\varphi\, d \widetilde \Mr_{i,n} \\
&\leq& -a_i\int_{B_{2R}(0)} \langle \nabla |u_{i,n}|,\nabla \varphi\rangle\, dx+t_n^2|\mu_i|\int_{B_{2R}(0)}   |u_{i,n}| \, \varphi  dx.
\end{eqnarray*}
Thanks to the conclusion in Step 1, we deduce that $(\| \widetilde \Mr_{i,n}\|_{\Meh(B_R(0))})_n$ is a bounded sequence. It follows then from \eqref{comparing_measures_i_n} that also 
$(\|\Mr_{i,n}\|_{\Meh(B_R(0))})_n$ is  bounded. A similar conclusion holds for $(\| \widetilde \Nr_{i,n}\|_{\Meh(B_R(0))})_n$ and  $(\|\Nr_{i,n}\|_{\Meh(B_R(0))})_n$.

\medbreak

\noindent \emph{Step 4.} We now turn to the proof of a Lipschitz estimate. Namely, we prove that,   given $R>0$ then, for $n$ sufficiently large,
\[
\|(u_n,v_n)\|_{C^{0,1}(B_R(0))}\leq C''.
\]
To this aim, suppose without loss of generality, that
\begin{eqnarray*}
[(u_n,v_n)]_{C^{0,1}(B_R(0))}&=&\max_{i=1,\ldots,k}\left\{  \mathop{\sup_{x,y\in B_R(0)}}_{x\neq y}\frac{|u_{i,n}(x)-u_{i,n}(y)|}{|x-y|}, \mathop{\sup_{x,y\in B_R(0)}}_{x\neq y}\frac{|v_{i,n}(x)-v_{i,n}(y)|}{|x-y|} \right\}\\
&=&\frac{|u_{1,n}(y_n)-u_{1,n}(z_n)|}{|y_n-z_n|},
\end{eqnarray*}
with $|y_n|\leq R$, $|z_n|\leq R$ and $y_n\neq z_n$. Let
\[
r_n:=|y_n-z_n|>0, \qquad 2R_n:=\max\{\dist(y_n,\Gamma_{u_n}), \dist(z_n,\Gamma_{u_n}) \}=\dist(z_n,\Gamma_{u_n}),
\]
where we have denoted
\[
\Gamma_{u_n}:=\{x\in \R^N: u_{i,n}(x)=0 \quad \forall i=1,\ldots,k\}.
\]
In view of proving our claim, we may already assume that $R_n>0$ and, thanks to Step 2, that $r_n\to 0$. Moreover, since $-a_1\Delta u_{1,n}=t_n^2\mu_1u_{1,n}$ in $\Omega_n\setminus \Gamma_{u_n}$, the desired conclusion follows easily from standard elliptic estimates in case $\liminf_n R_n>0$. In conclusion, we may already assume that
\[
r_n\to 0 \qquad \mbox{ and } \qquad 0<R_n\to 0.
\]
A crucial remark in our subsequent argument is the observation that
\begin{equation}\label{d(z_n,Gamma_n)=d(z_n,Gamma_(u_n,v_n))}
d(z_n,\Gamma_{u_n})=d(z_n,\Gamma_{(u_n,v_n)}),
\end{equation}
where we recall the definition
\begin{equation*}
\Gamma_{( u_n, v_n)}:=\{x\in \R^N: u_{i,n}(x)=0=v_{i,n}(x) \quad \forall i=1,\ldots,k\}.
\end{equation*}
This, of course, is a consequence of the fact that $\Big( \sum_{i=1}^k u_{i,n}^2\Big)\Big( \sum_{i=1}^k v_{i,n}^2\Big)\equiv 0$ and  $R_n>0$.

\noindent{\em Case 4a. Suppose first that $(R_n/r_n)_n$ is a bounded sequence.} By letting $U_n:=\sum_{i=1}^k u_{i,n}^2 $ we have $-\Delta U_n\leq 2Ct_n^2U_n$ in $\Omega_n$. By applying \cite[Lemma 3.9]{TavaresTerracini1} to this inequality we obtain that
\begin{align}\label{U_n^2(z_n)leq}
U_n(z_n)&\leq \frac{1}{|B_{R_n}(z_n)|} \int_{B_{R_n}(z_n)} U_n(x)\, dx + \frac{Ct_n^2R_n^2}{2(N+2)}\|U_n\|_{L^\infty(B_{R_n}(z_n))} \nonumber \\ 
		&\leq \frac{1}{|B_{R_n}(z_n)|} \int_{B_{R_n}(z_n)} U_n(x)\, dx + C'R_n^2.
\end{align}
Now, thanks to \eqref{d(z_n,Gamma_n)=d(z_n,Gamma_(u_n,v_n))} we can choose $w_n\in \Gamma_{(u_n,v_n)} \cap B_{2R}(0)$ in such a way that
\[
d(z_n,\Gamma_{(u_n,v_n)})=|z_n-w_n|=d(z_n,\Gamma_n)=2R_n.
\]
It follows then from Lemma \ref{H(x,(u_n,v_n),r)leqCr^2} that, for large $n$ and small $r$,
\[
H(w_n,(u_n,v_n),r)\leq Cr^2.
\]
In particular, we have
\[
\frac{1}{r^{N-1}} \int_{\partial B_r(w_n)}U_n(\sigma)\, d\sigma \leq C'r^2,
\]
thus also
\[
\frac{1}{r^{N}} \int_{ B_r(w_n)}U_n(x)\, dx \leq C''r^2.
\]
Since $R_n\to 0$, we obtain from \eqref{U_n^2(z_n)leq} that $U_n(z_n)\leq C''' R_n^2$. In particular, since by assumption the sequence $(R_n/r_n)_n$ is bounded, 
\[
u_{1,n}^2(z_n)\leq Cr_n^2,
\]
with $C$ independent of $n$.
This proves our desired estimate
\begin{equation}\label{|u_1,n(z_n)-u_1,n(y_n)|^2leqC |y_n-z_n|^2=Cr_n^2}
|u_{1,n}(z_n)-u_{1,n}(y_n)|^2\leq C |y_n-z_n|^2=Cr_n^2
\end{equation}
at least in the case where $u_{1,n}(y_n)=0$. But in case $u_{1,n}(y_n)\neq 0$ we have
\[
0<d(y_n,\Gamma_n)\leq d(z_n,\Gamma_n)=2R_n\to 0,
\]
and we obtain as above that $u_{1,n}^2(y_n)\leq Cd(y_n,\Gamma_n)^2\leq C'r_n^2$, leading still  to \eqref{|u_1,n(z_n)-u_1,n(y_n)|^2leqC |y_n-z_n|^2=Cr_n^2}.

\medbreak

\noindent{\em Case 4b.} Suppose secondly that $r_n/R_n \to 0$. By applying elliptic estimates to the equation $-a_1\Delta u_{1,n}=t_n^2\mu_1u_{1,n}$ together with the bound in Step 2 we obtain that
\begin{equation}\label{case2}
\|u_{1,n}\|_{C^{0,1}(B_{R_n/2}(z_n))} \leq C(R_n^{-1}\|u_{1,n}\|_{L^{\infty}(B_{R_n}(z_n))}+1).
\end{equation}
Arguing as in Case 1 above we prove the existence of $C>0$ such that for large $n$ and every $x\in B_{R_n}(z_n)$ it holds $u_{1,n}^2(x)\leq C d(x,\Gamma_n)^2\leq C'R_n^2$. By \eqref{case2} we deduce that the sequence $\|u_{1,n}\|_{C^{0,1}(B_{R_n/2}(z_n))} $ is bounded. Since by assumption $y_n \in B_{R_n/2}(z_n) $ for large $n$, this completes the argument if Case 2  holds and finishes the  proof of the estimates in Step 4.

\medbreak

\noindent \emph{Step 5.} By what we have proved so far, by taking subsequences we can find
\[
(u_n,v_n)\to (\bar u,\bar v) \quad \text{ in }C^{0,\alpha}_\textrm{loc}(\R^N)\quad \text{ for all }0<\alpha<1,\ \text{weakly in }H^1_\textrm{loc}(\R^N),
\]
as well as measures $\bar \Mr_i$, $\bar \Nr_i$, with the properties stated in Theorem \ref{thm:convergence_of_blowupsequence}. The normalization property \eqref{eq:L2_norm_equal_to_1} implies that $(\bar u,\bar v)\not\equiv (0,0)$, while the fact that  $\bar \Mr_i$ and $\bar \Nr_i$ are concentrated on   
 $\tilde \Gamma_{(\bar u,\bar v)}$ follows from the convergence $(u_n,v_n)\to (\bar u,\bar v)$ in $L^{\infty}_\textrm{loc}(\R^N)$. Next we show that
 \[
(u_n,v_n)\to (\bar u,\bar v) \quad \text{ in } \text{strongly in }H^1_\textrm{loc}(\R^N).
\]
To show this, we multiply the equation $-a_i\Delta (u_{i,n}-\bar u_i)=t_n^2\mu_i u_{i,n}+\bar \Mr_i- \Mr_{i,n}$ by $(u_{i,n}-\bar u_i)\varphi$, where a smooth cut-off function $\varphi$ such that $0\leq \varphi\leq 1$, $\varphi=1$  in $B_R(0)$ and $\varphi=0$ in $\R^N\setminus B_{2R}(0)$. Since
\[
\left|\int_{\R^N}(u_{i,n}-\bar u_i)\varphi\, d \Mr_{i,n}\right|\leq \int_{\R^N}|u_{i,n}-\bar u_i| \varphi\, d |\Mr_{i,n}|\leq \|u_{i,n}-\bar u_i \|_{L^{\infty}(B_{2R}(0))}|\Mr_{i,n}|(B_{2R}(0))
\]
and similarly for $|\int_{\R^N}(u_{i,n}-\bar u_i)\varphi\, d \Mr_{i}|$, we obtain that $a_i\int_{B_{R}(0)}|\nabla (u_{i,n}-\bar u_i)|^2\, dx\to 0$, and the conclusion follows. In particular, we obtain \eqref{multiply_by_baru_i_and_barv_i} by passing to the limit in \eqref{multiply_by_u_i_n} and in \eqref{multiply_by_v_i_n}.

\medbreak

\noindent {\em Step 6.} Let us now check \eqref{eq:pohozaev_for_bar_uv2}. From Corollary \ref{coro:Field_second} we deduce, by taking suitable rescailings, that 
\begin{align*}
(2-N) \sum_{i=1}^k \int_{B_r(x_0)} (a_i|\nabla u_{i,n}|^2+b_i |\nabla v_{i,n}|^2)\, dx\\
=\sum_{i=1}^k \int_{\partial B_r(x_0)} a_i r(2(\partial_n u_{i,n})^2-|\nabla u_{i,n}|^2)\, d\sigma +\sum_{i=1}^k \int_{\partial B_r(x_0)} b_i r (2(\partial_n v_{i,n})^2-|\nabla v_{i,n}|^2)\, d\sigma\\
+\sum_{i=1}^k \int_{\partial B_r(x_0)} r t_n^2(\mu_i u_{i,n}^2+\nu_i v_{i,n}^2)\, d\sigma-\sum_{i=1}^k \int_{B_r(x_0)} Nt_n^2 (\mu_i u_{i,n}^2+\nu_i v_{i,n}^2)\, dx.
\end{align*}
Then  \eqref{eq:pohozaev_for_bar_uv2} follows by letting $n\to \infty$ in this identity. 
\end{proof}

\medbreak

Having established a strong convergence result for any blowup sequence, based on the properties of $(u,v)$ we can characterize in a better way the possible blowup limits. First we observe that the Almgren's monotonicity formula valid from $(u,v)$ is transported to any limit $(\bar u,\bar v)$.

\begin{corollary}[Almgren's formula for $(\bar u,\bar v)$]\label{coro:Almgren_for_baru_barv}
For any $x_0\in \R^N$, $r>0$, let
\[
E(x_0,(\bar u,\bar v),r)=\frac{1}{r^{N-2}}\sum_{i=1}^k \int_{B_r(x_0)}(a_i |\nabla \bar u_i|^2+b_i|\nabla \bar v_i|^2)\, dx,
\]
\[
H(x_0,(\bar u,\bar v),r)=\frac{1}{r^{N-1}}\sum_{i=1}^k \int_{\partial B_r(x_0)} (a_i \bar u_i^2+b_i \bar v_i^2)\, d\sigma,
\]
and define the Almgren's quotient by
\[
N(x_0,(\bar u,\bar v),r)=\frac{E(x_0,(\bar u,\bar v), r)}{H(x_0,(\bar u, \bar v), r)}.
\]
Then $H(x_0,(\bar u,\bar v),r)\neq 0$ for every $r>0$, $r\mapsto N(x_0,(\bar u,\bar v), r)$ is a nondecreasing map, and
\begin{equation}\label{eq:derivative of H(bar u, bar v, r)}
\frac{d}{dr}\log (H(x_0,(\bar u,\bar v),r))=\frac{2}{r}N(x_0,(\bar u,\bar v), r).
\end{equation}
In particular, the set
\[
\Gamma_{(\bar u,\bar v)}=\{x\in \R^N:\ \bar u_i(x)=\bar v_i(x)=0 \text{ for every }i=1,\ldots,k\}
\]
has no interior points.
\end{corollary}
\begin{proof}
The proof is similar to the one in Theorem \ref{thm:Almgren}. We provide the full details. We simplify the notations by writing $E(r)=E(x_0,(\bar u,\bar v),r)$ and similarly for $  H(r)$ and $N(r)$. At first, thanks to \eqref{multiply_by_baru_i_and_barv_i} we see that
\begin{equation}\label{E(x_0,(baru,barv),r)=}
E(r)=\frac{1}{r^{N-2}}\sum_{i=1}^k \int_{\partial B_r(x_0)}(a_i \bar u_i\, \partial_n \bar u_i+b_i \bar v_i\, \partial_n \bar v_i)\, d\sigma
\end{equation}
On the other hand, by a simple computation similar to the the one in Lemma \ref{fracddrH(x_0,(u,v),r)=},
\begin{equation}\label{H'(x_0,(baru,barv),r)=}
H'(r)=\frac{2}{r^{N-1}}  \sum_{i=1}^k \int_{\partial B_r(x_0)}(a_i \bar u_i\, \partial_n \bar u_i+b_i \bar v_i\, \partial_n \bar v_i)\, d\sigma
\end{equation}
and, as in Lemma \ref{fracddrE(x_0,(u,v),r)=}, the derivative $E'(r)=\frac{d}{dr} E(x_0,(\bar u,\bar v),r)$ is given by
\begin{equation*}
\frac{2-N}{r^{N-2} }\sum_{i=1}^k \int_{B_r(x_0)} (a_i|\nabla\bar u_i|^2+b_i |\nabla\bar v_i|^2)\, dx+
\frac{1}{r^{N-2}}
\sum_{i=1}^k \int_{\partial B_r(x_0)} (a_i|\nabla\bar u_i|^2+b_i |\nabla\bar v_i|^2)\, d\sigma 
\end{equation*}
that is, thanks to \eqref{eq:pohozaev_for_bar_uv2}
\begin{equation}\label{E'(x_0,(baru,barv),r)=}
E'(r)=\frac{2}{r^{N-2}} \sum_{i=1}^k \int_{\partial B_r(x_0)} \left( a_i (\partial_n \bar u_i)^2+b_i (\partial_n \bar v_i)^2\right)\, d\sigma.
\end{equation}
By putting together \eqref{E(x_0,(baru,barv),r)=} to \eqref{E'(x_0,(baru,barv),r)=} we obtain that, as long as $H(r)>0$,
\begin{equation}\label{fracddrlogH(x_0,(baru,bar v),r)=}
(\log H)'(r)=\frac{2}{r} N(r)
\end{equation}
and
\begin{multline}\label{eq:Cauchy-Schwarz}
\frac{1}{2}H^2(r)r^{2N-3}  N'(r)= \sum_{i=1}^k \int_{\partial B_r(x_0)} (a_i(\partial_n \bar u_i)^2+b_i (\partial_n \bar v_i)^2)\, d\sigma \\ \times  \sum_{i=1}^k  \int_{\partial B_r(x_0)} (a_i \bar u_i^2+b_i\bar v_i^2)\, d\sigma
 - \left(\sum_{i=1}^k \int_{\partial B_r(x_0)} (a_i \bar u_i\partial_n \bar u_i+b_i \bar v_i \partial_n \bar v_i)\, d\sigma  \right)^2\geq 0,
\end{multline}
by the Cauchy-Schwarz inequality.

Now, arguing precisely as in Theorem \ref{thm:Almgren}, \eqref{fracddrlogH(x_0,(baru,bar v),r)=} together with the monotonicity of the map $N(r)$ implies that $\Gamma_{(\bar u,\bar v)}$ has empty interior. In particular, thanks also to  \eqref{multiply_by_baru_i_and_barv_i}, we must have that $H(r)>0$ for every $r>0$.
\end{proof}

\medbreak

We now introduce a very important definition, which will help us in the future to lighten the notation.

\begin{definition}\label{def_of_BU_y}
Given $y\in \Omega$, we define the set of all possible blowup limits at $y$ by
\[
\Bcal\Ucal_{y}=\left\{(\bar u,\bar v):  \begin{array}{l}\exists\; x_n\to y,\ t_n\to 0 \text{ such that, for every $i$,}\\[4pt]
					 \displaystyle u_{i,n}:=\frac{u_i(x_n+t_n\cdot)}{\sqrt{H(x_n,(u,v), t_n)}}\to \bar u\ \text{ and }\ v_{i,n}:=\frac{v_i(x_n+t_n\cdot)}{\sqrt{H(x_n,(u,v), t_n)}}\to \bar v \\[10pt]
					 \text{ strongly in }  H^1_\textrm{loc}(\R^N)\cap C^{0,\alpha}_\textrm{loc}(\R^N),\ \forall 0<\alpha<1
					 		 \end{array}\right\}
\]
\end{definition}

\medbreak 

From now on we will talk about the associated sequences $x_n,t_n$ defined for each $(\bar u,\bar v)\in \Bcal\Ucal_{y}$, with obvious meaning. Observe that at this point we cannot say that the blowup limit is unique, and it could a priori depende on the chosen sequences. Particular choices of the sequence $(x_n)_n$ provide additional information on the limit $(\bar u,\bar v)$, as a consequence of the Almgren's Formula.
\begin{corollary}\label{coro:blowup_special_cases}
Let $x_0\in \Omega$ and $(\bar u,\bar v)\in \Bcal\Ucal_{x_0}$ with associated sequences $x_n,t_n$. Assume that one of the following conditions hold:
\begin{itemize}
\item[1.] $(x_n)_n$ is a constant sequence, i.e. $x_n=x_0$ for every $n$;
\item[2.] $(x_n)_n\subseteq \Gamma_{(u,v)}$, and $x_0\in \Gamma_{( u, v)}$ is such that $N(x_0,(u,v),0^+)=1$.
\end{itemize}
Then $N(0,(\bar u,\bar v),r)\equiv N(x_0,(u,v),0^+)=:\alpha$ for every $r>0$, and there exist $g_i,h_i:S^{N-1}\to \R$ such that
$$
\bar u_i=r^\alpha g_i(\theta),\qquad \bar v_i=r^\alpha h_i(\theta),
$$
where $r=|x|$, $\theta=x/|x|$.
\end{corollary}
\begin{proof}
The proof that $N(r):=N(0,(\bar u,\bar v),r)$ has constant value equal to $\alpha=N(x_0,(u,v),0^+)$ follows exactly as in \cite[Corollary 3.12]{TavaresTerracini1}, on the ground of Theorem \ref{thm:Almgren}, Corollary \ref{Corollaries 2.6_2.7_2.8_TaTe}, and the fact that $N(0,(u_n,v_n),r)=N(x_n,(u,v),t_nr)$. In this situation we have  $N'(r)\equiv 0$ and from \eqref{eq:Cauchy-Schwarz} we know that, for a.e. $r>0$ there exists $C(r)>0$ such that
$$
(\partial_n \bar u_1,\ldots, \partial_n \bar u_k,\partial_n \bar v_1,\ldots, \partial_n \bar v_k)=C(r)(\bar u_1,\ldots, \bar u_k,\bar v_1,\ldots, \bar v_k)\qquad \text{ on }\partial B_r(0).
$$
By inserting this information in \eqref{eq:derivative of H(bar u, bar v, r)} we obtain that
$$
\frac{2}{r}\alpha=\frac{H'(r)}{H(r)}=\frac{2\sum_{i=1}^k \int_{\partial B_r(0)}(a_i \bar u_i (\partial_n \bar u_i)+b_i \bar v_i (\partial_n \bar v_i))\, d\sigma}{\sum_{i=1}^k \int_{\partial B_r(0)} (a_i \bar u_i^2+b_i \bar v_i^2)\, d\sigma}=2C(r)
$$
and $C(r)=\alpha/r$. As on $\partial B_r(0)$ the normal derivative corresponds to the derivative with respect to the radius, the conclusion follows.
\end{proof}


\subsection{Hausdorff dimension estimates for the nodal and singular sets. Definition of $\Reh_{(u,v)}$, the regular part of the nodal set}\label{subset:HausdorffDim}

In this subsection we provide upper estimates for the Hausdorff dimension of the nodal set as well as for some of its subsets. Recall from Corollary \ref{Corollaries 2.6_2.7_2.8_TaTe} that $N(x,(u,v),0^+)\geq 1$ whenever $x\in \Gamma_{(u,v)}$.

\begin{definition}
Let $(u,v)$ be the limit given by Theorem \ref{thm:convergences}. We split the nodal set $\Gamma_{(u,v)}$ into the following two sets:
\[
\Reh_{(u,v)}=\{x\in \Gamma_{(u,v)}:\ N(x,(u,v),0^+)=1\}
\] 
and
\[
\Seh_{(u,v)}=\Gamma_{(u,v)}\setminus \Reh_{(u,v)}=\{x\in \Gamma_{(u,v)}:\ N(x,(u,v),0^+)>1\}.
\]
\end{definition}
We will prove in the subsequent subsections that $\Reh_{(u,v)}$ turns out to be the regular part of the nodal set, being of class $C^{1,\alpha}$ for some $0<\alpha<1$, whereas the singular part $\Seh_{(u,v)}$ has small Hausdorff measure: more precisely at most $N-2$. In the spirit of \cite[Lemma 4.1]{CaffarelliLinJAMS} or \cite[Lemma 6.1]{TavaresTerracini1}, we prove a jump condition for the value of the map $x\in \Gamma_{(u,v)}\mapsto N(x,(u,v),0^+)$, which will yield in particular that $\Seh_{(u,v)}$ is relatively close in $\Gamma_{(u,v)}$. Here, the fact that we have groups of components, which eventually change sign, makes the proof more delicate.

\begin{proposition}\label{prop:jump_condition_on_N}
Let $(u,v)$ be the limit given by Theorem \ref{thm:convergences}. Then there exists $\delta_N>0$ (depending only on the dimension) such that, for every $x_0\in \Gamma_{(u,v)}$, either
$$
N(x_0,(u,v),0^+)=1 \quad \text{ or } \quad N(x_0,(u,v),0^+)\geq 1+\delta _N
$$
\end{proposition}
\begin{proof}
Consider a blowup sequence at a fixed center $x_0 \in \Gamma_{(u,v)}$:
$$
u_{i,n}(x)=\frac{u_i(x_0+t_n x)}{\rho_n},\qquad v_{i,n}(x)=\frac{v_i(x_0+t_n x)}{\rho_n},
$$
where $\rho_n^2=H(x_0,(u,v),t_n)$ and $t_n\downarrow 0$. By Theorem \ref{thm:convergence_of_blowupsequence} and Corollary \ref{coro:blowup_special_cases} (case 1) we have convergence (up to a subsequence) to a limiting configuration $(\bar u,\bar v)$, where
$$
\bar u_i=r^\alpha g_i(\theta), \quad \bar v_i=r^\alpha h_i(\theta),\qquad \text{ with }\alpha=N(x_0,(u,v),0^+).
$$
Moreover, these functions are harmonic in $\R^N\setminus  \Gamma_{(\bar u,\bar v)}$, whence
$$
-\Delta_{S^{N-1}} g_i=\lambda g_i \quad \text{ and } \quad -\Delta_{S^{N-1}} h_i=\lambda h_i \qquad \text{ in } \partial B_1(0)\setminus \Gamma_{(\bar u,\bar v)},
$$
$\lambda=\alpha(\alpha+N-2)$. In particular, if $A$ is a connected component of either $\partial B_1(0)\cap \{\bar u_i\neq 0\}$ or $\partial B_1(0)\cap \{\bar v_i\neq 0\}$ for some $i$, then we have $\lambda=\lambda_1(A)$, the first eigenvalue of $-\Delta_{S^{N-1}}$ in $H^1_0(A)$. We now split the proof in two cases:

\medbreak

\noindent \emph{Case 1:} We have $\bar u_i\not\equiv 0$ and $\bar v_j\not\equiv 0$ for some $i,j$. 

Without loss of generality, suppose that for $1\leq l,m\leq k$ we have
$$
\bar u_1,\ldots, \bar u_l\not \equiv 0,\qquad \bar v_1,\ldots, \bar v_m\not \equiv 0,
$$
and that all the remaining components are 0.

\medbreak

\emph{Subcase 1.1.} Suppose that for one component, say $\bar u_1$, the set $\{\bar u_1\neq 0\}$ has at least two connected components: $A_1$ and $A_2$. Let $B$ be a connected component of $\{\bar v_1\neq 0\}$. As $\bar u_i\cdot \bar v_j \equiv 0$, the sets $A_1,A_2$ and $B$ are disjoint, and thus there exists $C\in \{A_1\cap \partial B_1(0),A_2\cap \partial B_1(0),B\cap \partial B_1(0)\}$ such that 
$$
\Hh^{N-1}(C)\leq \Hh^{N-1}(\partial B_1(0))/3.
$$
Whence
$$
\lambda=\lambda_1(C)\geq \lambda_1(E(\pi/3))=\lambda_1(\{x\in \partial B_1(0):\ \arccos (\langle x,(0,\ldots, 0,1)\rangle<\pi/3)\}).
$$
Since $E(\pi/2)$ (the half-sphere) has first eigenvalue equal to $N-1$, we obtain the existence of $\gamma>0$ such that $\lambda_1(E(\pi/3))=N-1+\gamma$, and 
$$
\alpha=\sqrt{\left(\frac{N-2}{2}\right)^2+\lambda}-\frac{N-2}{2}\geq 1+\delta_N \qquad \text{ for some }\delta_N>0.
$$

\medbreak
 
\noindent \emph{Subcase 1.2.} Suppose that the sets $\{\bar u_i\neq 0\}$, $\{\bar v_i\neq 0\}$ have at most one connected component. Then each nontrivial component is signed, and to fix ideas we can suppose that the nontrivial components are
$$
\bar u_1,\ldots,\bar u_l \geq 0,\qquad \text{ and } \qquad \bar v_1,\ldots, \bar v_m\geq 0.
$$
If there exists $i,j$ such that $\{\bar u_i> 0\}\cap \{\bar u_j> 0\}=\emptyset$ (or the same for $\bar v$), then we can argue exactly as in Subcase 1.1 and obtain $\alpha\geq 1+\delta_N$. Thus, in the remaining case, $\{\bar u_i>0\}\cap \{\bar u_j>0\}\neq \emptyset$ for every $i,j=1,\ldots, l$. Since each component $\bar u_i$ is harmonic in $\cup_{j}\{\bar u_j>0\}$, by the maximum principle we must have $\{ \bar u_j>0\}=\{ \bar u_1>0\}$. Analogously, $\{ \bar v_i>0\}=\{ \bar v_1> 0\}$ for every $i=1,\ldots, m$. In this case $\bar u_j=l_j \bar u_1$, $\bar v_j=\tilde l_j \bar v_1$ (with $l_1=\tilde l_1=1$) and the functions $\tilde u=\sqrt{\sum_{i=1}^k a_il_i^2}\bar u_1=\sqrt{\sum_{i=1}^l a_i \bar u_i^2}$, $\tilde v=\sqrt{\sum_{i=1}^m b_i \tilde l_i^2}\bar v_1=\sqrt{\sum_{i=1}^m b_i \bar v_i^2}$  belong to the class $\Geh_\textrm{loc}(\R^N)$  introduced in \cite[Definition 3.2]{TavaresTerracini1}. Thus by \cite[Lemma 6.1]{TavaresTerracini1} we have that  $N(0,(\bar u,\bar v),0^+)=N(0,(\tilde u,\tilde v), 0^+)=\alpha$ is equal to 1 and moreover that $\Gamma_{(\bar u,\bar v)}=\Gamma_{(\tilde u,\tilde v)}$ is a hyperplane.

\medbreak

\noindent \emph{Case 2:} Either $\bar u_1\equiv \ldots \equiv \bar u_k\equiv 0$ or $\bar v_1\equiv \ldots \equiv \bar v_k\equiv 0$.

Suppose that $\bar u_1,\ldots, \bar u_l\not \equiv 0$, $l\leq k$, and that all the other functions are zero.

If $\Gamma_{(\bar u,\bar v)}\cap S^{N-1}=\emptyset$, then all nontrivial components are harmonic, and thus $\alpha\in \N$. Moreover, if $\alpha=1$, then $\bar u_i=x\cdot \nu_i$ for some $\nu_i\in S^{N-1}$ and in particular $\Gamma_{\bar u_i}$ is a hyperplane and $\Gamma_{(\bar u,\bar v)}$ is a vector space of dimension at most $N-1$.

Let us now suppose that $\Gamma_{(\bar u,\bar v)}\cap S^{N-1}\neq \emptyset$. We will prove this case via an induction argument which follows very closely \cite[Section 6]{TavaresTerracini1}. Because of that, here we will be rather sketchy, providing however all the precise references for the details. We will show that either $\alpha\geq 1+\delta_N$, or $\alpha=1$ and $\Gamma_{(\bar u,\bar v)}$ is a hyperplane.

Observe that, in dimension 2, each connected component of $S^{N-1}\cap \{\bar u_i\neq 0\}$ is an arc, and that all these arcs must have the same length. Moreover, $\Gamma_{\bar u_i}\cap S^{N-1}$ cannot consist of a single point, as in that case
\[
\lambda=\lambda_1(E(\pi))=\frac{1}{4}<1, \qquad \text{ a contradiction}.
\]
Take $\bar x\in \Gamma_{(\bar u,\bar v)}\cap S^{N-1}$, and let $A$ be a connected component of $\{\bar u_1 \neq 0\}\cap S^{N-1}$ (we take the first component, without loss of generality) so that $\bar x$ is one end of the arc $A$. Let $\bar y$ denote the other end of $A$, so that $\bar u_1(\bar y)=0$. Clearly $\bar x\neq \bar y$, and thus there exists another arc $B$ disjoint from $A$, having $\bar y$ as one end point. Let $\bar z$ be the other end of $B$. If $\bar z\neq \bar x$, we obtain the existence of at least three disjoint arcs and, as in Subcase 1.1, $\alpha\geq 1+\delta_2$ for some $i$. If $\bar z=\bar x$ then we conclude that each arc must be a half circle, and in this case we have $\alpha=1$ and $\Gamma_{(\bar u,\bar v)}$ is a hyperplane.

Suppose now that we can prove that the same situation occurs for $N-1$, Case 2. Take $(\bar u,\bar v)=r^\alpha(G(\theta),H(\theta))$ in dimension $N$, belonging to $\Bcal\Ucal_{x_0}$. First of all, we can suppose once again that $S^{N-1}\setminus \Gamma_{(\bar u,\bar v)}$ has at least one, and at most two connected components. As in \cite[p. 306]{TavaresTerracini1}, for each $y_0\in \Gamma_{(\bar u,\bar v)}\cap S^{N-1}$ we can take a blowup sequence centered at $y_0$:
\[
\tilde u_{i,n}(x)=\frac{\bar u_i(y_0+t_n x)}{\sqrt{H(y_0,(\bar u,\bar v),t_n)}},\qquad \tilde v_{i,n}(x)=\frac{\bar v_i(y_0+t_n x)}{\sqrt{H(y_0,(\bar u,\bar v),t_n)}}
\]
which converges to
\[
\tilde u_i=r^{\gamma_{y_0}}\tilde g_i(\theta),\qquad \tilde v_i=r^{\gamma_{y_0}}\tilde h_i(\theta)
\]
for some $\gamma_{y_0}>0$. By the homogeneity of $\bar u,\bar v$, as in \cite[Lemma 6.3]{TavaresTerracini1}, $(\tilde u,\tilde v)$ only depends on $N-1$ variables, and by the induction hypothesis either $\gamma_{y_0}=1$ (and $\Gamma_{\tilde u,\tilde v}$ is a hyperplane) or $\gamma_{y_0}\geq 1+\delta_{N-1}$. Now either $\gamma_{y_0}\geq 1+\delta_{N-1}$ for some $y_0$ -- which gives $\alpha\geq 1+\delta_{N-1}$ by the upper semicontinuity of $x\mapsto N(x,(\bar u,\bar v),0^+)$ as in \cite[Lemma 6.5]{TavaresTerracini1} -- or $\gamma_{y_0}=1$ for every $y_0$. In the latter case, we can reason exactly as in \cite[p. 307]{TavaresTerracini1} to conclude that $S^{N-1}\cap \Gamma_{(\bar u,\bar v)}$ is a hyperplane, and $\alpha=1$.
\end{proof}

\begin{remark}\label{rem:at_most_hyperplane}
From the previous proof we learn that, given $(\bar u,\bar v)\in \Bcal\Ucal_{x_0}$ with $x_0\in \Reh_{( u, v)}$, we have that $\Gamma_{(\bar u,\bar v)}$ is a vector space having dimension at most $N-1$. Moreover, it is actually a hyperplane except if possibly either $\bar u\equiv 0$ or $\bar v\equiv 0$. Via a ``Clean-Up'' result, in the next subsection we will check that neither $\bar u$ nor $\bar v$ can vanish when we blowup at a point of $\Reh_{(u, v)}$. This key fact will yield a more complete characterization of $(\bar u,\bar v)$, cf. Theorem \ref{coro:hyperplane} ahead.
\end{remark}

\begin{corollary}
The set $S_{(u,v)}$ is closed in $\Omega$ for every $N\geq 2$.
\end{corollary}
\begin{proof}
This is a direct consequence of Proposition \ref{prop:jump_condition_on_N} combined with the upper semicontinuity of the map $x\mapsto N(x,(u,v),0^+)$.
\end{proof}

\begin{theorem}\label{thm:hausdorff_measures_of_nodalsets}
For any $N\geq 2$ we have that:
\begin{itemize}
\item[1.] $\Hh_\textrm{dim}(\Gamma_{(u,v)})\leq N-1$;
\item[2.] $\Hh_\textrm{dim}(S_{(u,v)})\leq N-2$. If $N=2$ then moreover for any given compact set $\tilde \Omega \Subset \Omega$ we have that $S_{(u,v)}\cap \tilde \Omega$ is a finite set.
\end{itemize}
\end{theorem}
\begin{proof}
Once one knows that both the sets $\Gamma_{(u,v)}$ and $\Seh_{(u,v)}$ are relatively close in each open subset of $\Omega$, and taking in considerations the results on the previous and current subsection, the proof follows almost word by word the one presented in \cite[Theorem 4.5 \& Remark 4.7]{TavaresTerracini1} for a similar situation (see also \cite[Theorem 4.7(b)]{CaffarelliLinJAMS}). The proof therein is based on the Federer Reduction Principle.
\end{proof}

\subsection{A Clean-Up result. $\Reh_{(u,v)}$ is locally flat}\label{subsec:Cleanup}

Having established that $S_{(u,v)}$ is negligible, we will now focus on the remaining part of nodal set, namely $\Reh_{(u,v)}$. We prove in this and in the following subsection that $\Reh_{(u,v)}$ is a regular hypersurface. Here we will prove a \emph{Clean-Up} result, which will imply that both $\bar u,\bar v \not \equiv 0$, whenever $(\bar u,\bar v)\in \Bcal\Ucal_{x_0}$ with $x_0\in \Reh_{(u,v)}$. This will allow us to prove that, after a blowup at any point, $\Reh_{(u,v)}$ is a hyperplane, and it is locally flat.

Given $x\in \Omega$, $\nu\in S^{N-1}$ and $\ep>0$, we define the strip
$$
S(x,\nu,\ep):=\{y\in \R^N:\ |(y-x)\cdot \nu|\leq \ep\}.
$$

We saw in the previous subsection (cf. Remark \ref{rem:at_most_hyperplane}) that, whenever we take a blowup sequence centered at a point of $\Reh_{(u,v)}$, the limiting nodal set is a vector space, being at most a hyperplane. This implies that $\Reh_{(u,v)}$ is locally contained in a small strip.

\begin{lemma}\label{lemma:Gamma_locally_in_a_strip}
Given $x_0\in \Reh_{(u,v)}$, let $R_0>0$ be such that 
\[
B_{2R_0}(x_0)\cap \Reh_{(u,v)}=B_{2R_0}(x_0)\cap \Gamma_{(u,v)}
\]  
(recall that $\Reh_{(u,v)}$ is an open set of $\Gamma_{(u,v)}$). Then 
\begin{eqnarray*}\label{eq:ep-flat_degree0}
&\forall \ep>0,\ \exists 0<r_0<R_0:\ \forall 0<r<r_0,\ \forall x\in \Gamma_{(u,v)}\cap B_{R_0}(x_0), \ \exists \nu=\nu(x,r)\in S^{N-1} \text{ such that}&\nonumber\\
&\Gamma_{(u,v)}\cap B_r(x)\subset S(x,\nu(x,r),\ep r).&
\end{eqnarray*}
\end{lemma}
\begin{proof}
Supposing that the conclusion does not hold, then there exists $\ep>0$, $r_n\to 0$, $x_n \in \Reh_{(u,v)}\cap B_{R_0}(x_0)$ such that
\[
\Gamma_{(u,v)}\cap B_{r_n}(x_n)\not \subseteq S(x_n,\nu,\ep r_n),\qquad \forall \nu\in S^{N-1}.
\]
Consider the blowup sequences
$$
u_{i,n}(x)=\frac{u_i(x_n+r_nx)}{\rho_n},\qquad v_{i,n}(x)=\frac{v_i(x_n+r_n x)}{\rho_n}.
$$
Then we have
\begin{equation}\label{eq:strip:contradiction!}
\Gamma_{(u_n,v_n)}\cap B_1(0)\not \subseteq S(0,\nu, \ep) \qquad \forall\nu\in S^{N-1}.
\end{equation}
Suppose that $x_n\to \bar x\in \Reh_{(u,v)}\cap \overline{B_{R_0}(x_0)}$. We have (up to a subsequence)
$$
u_{i,n}\to \bar u_i,\qquad v_{i,n}\to \bar v_i \qquad \text{strongly in }C^{0,\alpha}(\overline B_1(0))\cap H^1(B_1(0))
$$
and, as $N(\bar x,(u,v),0^+)=1$, we know (Remark \ref{rem:at_most_hyperplane}) that $\Gamma_{(\bar u,\bar v)}$ is a vector space of dimension at most $N-1$. Thus there exists $\bar \nu\in S^{N-1}$ and $\gamma>0$ such that
$$
\sum_{i=1}^k (\bar u_i^2+\bar v_i^2)\geq 2\gamma,\qquad \text{ in } B_1(0)\setminus S(0,\bar \nu,\ep)=\{y\in B_1(0):\ |y\cdot \bar \nu|> \ep\}
$$
and hence, for large $n$,
$$
\sum_{i=1}^k (\bar u_{i,n}^2+\bar v_{i,n}^2)\geq \gamma,\qquad \text{ in } B_1(0)\setminus S(0,\bar \nu,\ep),
$$
which contradicts \eqref{eq:strip:contradiction!}.
\end{proof}

\medbreak

Next we state an independent result, based on a similar one for harmonic functions present in \cite[Lemma 12]{Aram}.

\begin{lemma}\label{lemma:saint_venant}
Consider the ball $B_1(0)\subseteq \R^N$. Then for every $a>0$ there exist $C,\bar \ep>0$ such that, for every $0<\ep<\bar \ep$, and every $w\in C(B_1(0))$ nonnegative, satisfying $-\Delta w\leq aw$ in $B_1(0)$ and $w\equiv 0$ in $B_1(0)\setminus S(0,\nu,\ep)$ for some $\nu\in S^{N-1}$, we have
$$
\sup_{B_{1/2}(0)} w \leq e^{-C/\ep} \sup_{B_1(0)} w.
$$
\end{lemma}
\begin{proof}
Take $\kappa>0$ large so that
\[
|B_\kappa(0)\cap S(0,\nu,1)|\leq \frac{1}{4}|B_\kappa(0)|
\]
and take $\bar \ep=\min\{\frac{1}{2\kappa},\sqrt{\frac{N+2}{2\kappa^2 a}}\}$. Given $0<\ep<\bar \ep$, and $x\in B_{1/2}(0)$, we have $B_{\kappa \ep}(x)\subset B_1(0)$, and
\[
\frac{|B_{\kappa\ep}(x)\cap S(0,\nu,\ep))|}{|B_{\kappa\ep}(x)|}\leq \frac{1}{4}.
\]
Then, given $x\in B_{1/2}(0)$, as $-\Delta w\leq aw$ in $B_1(0)$,\footnote{The first inequality is an easy adaptation of the standard proof of the mean value theorem for subharmonic functions, see for instance \cite[Theorem 2.1]{GilbargTrudinger}.}
\[
\begin{split}
w(x) &\leq \frac{1}{|B_{\kappa\ep}(x)|}\int_{B_{\kappa\ep}(x)} w\, dx + \frac{a(\kappa\ep)^2}{2(N+2)}\sup_{B_{\kappa\ep}(x)}w\\
	&\leq \frac{|B_{\kappa\ep}(x)\cap S(0,\nu,\ep)|}{|B_{\kappa\ep}(x)|} \sup_{B_{\kappa\ep}(x)}w + \frac{1}{4}\sup_{B_{\kappa\ep}(x)}w\\
	&\leq \frac{1}{2}\sup_{B_{\kappa\ep}(x)}w,
\end{split}
\]
We can now iterate this procedure $\text{int}\left(\frac{1-|x|}{\kappa \ep}\right)$ times (where $\text{int}(\, \cdot\,)$ denotes the integer part of a number), and hence
\[
w(x)\leq \left(\frac{1}{2}\right)^{\frac{1-|x|}{\kappa\ep}-1} \sup_{B_1(0)} w \leq \left(\frac{1}{2}\right)^\frac{1}{4\kappa\ep}\sup_{B_1(0)} w \qquad \forall x\in B_{1/2}(0),
\]
Thus the conclusion holds for $C:=\frac{\log 2}{4k}>0$.
\end{proof}

\medbreak

Given $x_0\in \Reh_{(u,v)}$ and $0<\ep<\bar \ep$, let $R_0$, $r_0$ be as in Lemma \ref{lemma:Gamma_locally_in_a_strip}.
For $x\in \Gamma_{(u,v)}\cap B_{R_0}(x_0)$ and $0<r<r_0$, we define the set 
\[
\Upsilon_{(x,r)}=\{\nu\in S^{N-1}:\ \Gamma_{(u,v)}\cap B_r(x)\subseteq S(x,\nu,\ep r)\}\neq \emptyset.
\]

\begin{lemma}
If there exists $\bar r\leq r_0$, $\bar x\in \Gamma_{(u,v)}\cap B_{R_0}(x_0)$, and $\bar \nu\in \Upsilon_{(\bar x,\bar r)}$ such that
\[
\sum_{i=1}^k u_i^2=0 \quad \text{ in } B_{\bar r}(\bar x)\setminus S(\bar x,\bar \nu, \ep \bar r),
\]
then
\[
\sum_{i=1}^k u_i^2=0 \quad \text{ in } B_r(x)\setminus S(x,\nu, \ep  r)\ \forall x\in  \Gamma_{(u,v)}\cap B_{R_0}(x_0),\ r\leq r_0,\ \nu\in \Upsilon_{(x,r)}.
\]
\end{lemma}
\begin{proof}
Take $x\in \Gamma_{(u,v)}\cap B_{R_0}$, $r< r_0$. Observe that if $\sum_{i=1}^k u_i^2=0$ in $B_r(x)\setminus S(x,\nu,\ep r)$ for some $\nu\in \Upsilon_{(x,r)}$, then $\sum_{i=1}^k v_i^2>0$ in such set. Moreover,  if $\sum_{i=1}^k u_i^2=0$ in $B_r(x)\setminus S(x,\nu,\ep r)$ for some $\nu\in \Upsilon_{(x,r)}$, then the same is true for all vectors in the latter set. 

Thus the map $\varphi:B_{R_0}(x_0)\times(0,r_0)\to \{0,1\}$ defined as $\varphi(x,r)=0$ if  $\sum_{i=1}^k u_i^2=0$ on $B_r(x)\setminus S(x,\nu,\ep r)$ for some $\nu\in \Upsilon_{(x,r)}$, $\varphi(x,r)=1$ otherwise, 
is continuous, hence constant. 
\end{proof}

Combining this result with Lemma \ref{lemma:saint_venant}, we have the following.

\begin{lemma}
Suppose there exists $\bar r\leq r_0$, $\bar x\in \Gamma_{(u,v)}\cap B_{R_0}(x_0)$, and $\bar \nu\in \Upsilon_{(\bar x,\bar r)}$ such that
\[
\sum_{i=1}^k u_i^2=0 \quad \text{ in } B_{\bar r}(\bar x)\setminus S(\bar x,\bar \nu, \ep \bar r),
\]
Then
$$
\sup_{B_{r/2}(x)}\sum_{i=1}^k u_i^2 \leq e^{-C/\ep} \sup_{B_r(x)} \sum_{i=1}^k u_i^2,\qquad \forall x\in \Gamma_{(u,v)}\cap B_{R_0}(x_0), \ r\leq r_0.
$$

\end{lemma}
\begin{proof}
Take $x\in  \Gamma_{(u,v)}\cap B_{R_0}(x_0)$, $r\leq r_0$ and $\nu\in \Upsilon_{(x,r)}$. From the previous lemma we know that $\tilde u:=\sum_{i=1}^k u_i^2=0$ in $B_r(x)\setminus S(x,\nu, \ep  r)$ . Moreover, $-\Delta \tilde u \leq 2 \tilde \lambda \tilde u$ in $B_r(x)$, with $\tilde \lambda=\max_i \{\mu_i/a_i\}$. Hence, rescaling $\tilde u$: $\tilde u_1(y)=\tilde u(x+r y)$, we have
$$
-\Delta \tilde u_1\leq a \tilde u_1 \text{ in } B_1(0) \text{ for $a=2R_0^2\tilde \lambda$},\quad \text{ and } \qquad \tilde u_1=0 \text{ in } B_1(0)\setminus S(0,\nu,\ep).
$$
Then, by Lemma \ref{lemma:saint_venant}, 
$$
\sup_{B_{1/2}(0)}\tilde u_1\leq e^{-C/\ep} \sup_{B_1(0)}\tilde u_1,
$$
and the lemma follows going back to $\tilde u$.
\end{proof}

\begin{proposition}[Clean-Up]\label{prop:cleanup} Take $\bar x\in \Gamma_{(u,v)}\cap B_{R_0}(x_0)$ and suppose that, for small $r$, 
\[
\sum_{i=1}^k u_i^2 \equiv 0 \qquad  \text{ in } B_r(\bar x)\setminus S(\bar x, \nu(\bar x, r), \ep r)
\]
for some $\nu(\bar x,r)\in \Upsilon_{(\bar x,r)}$. Then 
\[
\sum_{i=1}^k u_i^2 \equiv 0\qquad  \text{ in } B_{R_0}(x_0).
\] 
The same results holds for $\sum_{i=1}^k v_i^2$ in the place of $\sum_{i=1}^k u_i^2$.
\end{proposition}

\begin{proof}
We want to check that the measures $\Delta u_i+\lambda_i u_i$  over $B_{R_0}(x_0)$ is zero for each $i=1,\ldots, k$, doing this by covering this set with balls of radius $1/2^{n+1}$ (here, $\lambda_i=\mu_i/a_i$). Take $\bar n$ such that $1/2^{\bar n}\leq r_0$. Then
$$
\sup_{B_{1/2^{n+1}}(\bar x)} \sum_{i=1}^k u_i^2\leq e^{-C/\ep} \sup_{B_{1/2^n}(\bar x)} \sum_{i=1}^k u_i^2,\qquad \forall \bar x\in \Gamma_{(u,v)}\cap B_{R_0}(x_0),\ n\geq \bar n.
$$
We can then iterate, obtaining
$$
\sup_{B_{1/2^n}(\bar x)} \sum_{i=1}^k u_i^2\leq C_1 e^{-\frac{C}{\ep}(n-\bar n)}, \qquad \forall n\geq \bar n+1,
$$
for $C_1$ depending on $\|u\|_\infty$.  Observe that if $B_{1/2^{n+1}}(x)\cap \Gamma_{(u,v)}= \emptyset$, then $\Delta u_1+\lambda_1 u_1=0$ on such a ball. Thus (in the sense of measures)
\[
 \int_{B_{R_0}(x_0)}(\Delta u_i+\lambda_i u_i)\, dx \displaystyle=\int_A (\Delta u_i+\lambda_i u_i)\, dx=\int_A (\Delta u_i^++\lambda_i u_i^+)\, dx\ -\int_A(\Delta u_i^-+\lambda_i u_i^-)\, dx
\]
for 
\[
A:=\bigcup_{\bar x\in \Gamma_{(u,v)}\cap B_{R_0}(x_0)} B_{1/2^{n+1}}(\bar x).
\]
On the other hand, using a cut-off function argument, it is straightforward to show that
\[
\begin{split}
0\leq \int_{B_{1/2^{n+1}}(\bar x)}(\Delta u_i^\pm +\lambda_i u_i^\pm)\, dx &\leq C_2 2^{2n}\int_{B_{1/2^n}(\bar x)} u_i^{\pm}\, dx \leq C_2 2^{2n}\int_{B_{1/2^n}(\bar x)} \sqrt{\sum_{i=1}^ku_i^2}  \, dx\\
												&\leq C_2 \frac{2^{2n}}{2^{nN}}e^{-\frac{C}{2\ep}(n-\bar n)}	 \leq C_3 e^{-\frac{C}{2\ep}n}, \qquad \forall n\geq \bar n,\ N\geq 2.
\end{split}
\]
Since we can cover $B_{R_0}(x_0)$ with $((\text{int}(R_0)+1)2^{n+2}+2)^N$ balls of radius $1/2^{n+1}$, we deduce that
\[
\begin{split}
0\leq \int_A (\Delta u_i^++\lambda_i u_i^+)\, dx &\leq  \sum_{\bar x\in A} \int_{B_{1/2^{n+1}(\bar x)}} (\Delta u_i^++\lambda_i u_i^+)\, dx \\
&\leq \tilde C_4 2^{nN} e^{-\frac{C}{2\ep}n}=C_4 e^{n(N\log 2-\frac{C}{\ep})}\to 0\quad \text{ as $n\to \infty$},
\end{split}
\]
 if we can choose right from the start $\ep<\frac{C}{N\log 2}$. Since the same result clearly holds also for $u_i^-$, we conclude that $-\Delta u_i=\lambda_i u_i$ in $B_{R_0}(x_0)$, and by the unique continuation property, since $u_i$ vanishes in a set with nonempty interior, then $u_i\equiv 0$.
\end{proof}

This result has many strong and important consequences.

\begin{theorem}\label{coro:hyperplane} Take $x_0\in \Reh_{(u,v)}$ and $(\bar u,\bar v)\in \Bcal\Ucal_{x_0}$. Then $\bar u,\bar v\not \equiv 0$. Moreover, $\Gamma_{(\bar u,\bar v)}$ is a hyperplane, and there exist $\nu\in S^{N-1}$ and $\alpha_1,\ldots, \alpha_k,\beta_1,\ldots,\beta_k\in \R$ such that
\begin{equation}\label{eq:u=x_1^+_v=x_1^-}
\bar u_i=\alpha_i (x\cdot \nu)^+,\qquad \bar v_i=\beta_i (x\cdot \nu)^- \qquad \text{ in } \R^N.
\end{equation}
Furthermore,
\begin{equation}\label{eq:Reflection_Principle}
\sum_{i=1}^k a_i \alpha_i^2=\sum_{i=1}^k b_i \beta_i^2,
\end{equation}
so that the function 

\[
\sqrt{a_1 \bar u_1^2+\ldots + a_k \bar u_k^2}-\sqrt{b_1 \bar v_1^2+\ldots+b_k \bar v_k^2} \qquad \text{ is harmonic in } \R^N.
\]
\end{theorem}

\begin{proof}

\medbreak

1. Let $(\bar u,\bar v)\in \Bcal\Ucal_{x_0}$ with $x_0\in \Reh_{(u,v)}$ and suppose, in view of a contradiction, that $\bar u\equiv 0$. Take $x_n\in \Gamma_{(u,v)}$ with $x_n\to x_0$ and $t_n\to 0$ such that
\[
u_n(x)=\frac{u(x_n+t_n x)}{\sqrt{H(x_n,(u,v),t_n)}}\to \bar u,\qquad v_n(x)=\frac{v(x_n+t_n x)}{\sqrt{H(x_n,(u,v),t_n)}}\to \bar v.
\]
As $\Gamma_{(\bar u,\bar v)}=\Gamma_{\bar v}$ is a vector space of dimension at most $N-1$ (cf. Remark \ref{rem:at_most_hyperplane}), then there exists $\bar \nu\in S^{N-1}$ such that, for every $\ep>0$, $\bar u\equiv 0$ and $\bar v\neq 0$ in $B_1(0)\setminus S(0,\bar \nu,\ep)$. For $\ep>0$ fixed let $\gamma>0$ be such that
\[
\sum_{i=1}^k \bar v_i^2\geq 2\gamma \quad \text{ in } B_1(0)\setminus S(0,\bar \nu,\ep).
\]
Then, for sufficiently large $n$,
\[
\sum_{i=1}^k v_{i,n}^2\geq \gamma \quad \text{ in } B_1(0)\setminus S(0,\bar \nu,\ep),
\]
and so
\[
\sum_{i=1}^k v_i^2 > 0,\ \sum_{i=1}^k u_{i}^2\equiv 0 \quad \text{ in } B_{t_n}(x_n) \setminus  S(x_n,\bar \nu,t_n \ep ).
\]
As $t_n<r_0$, $x_n\in \Gamma_{(u,v)}\cap B_{R_0}(x_0)$ for large $n$, and $\bar \nu\in \Upsilon_{(x_n,t_n)}$, then by Proposition \ref{prop:cleanup} we obtain $u\equiv 0$ in $B_r(x_0)$ for small $r>0$. However, this contradicts the fact that $x_0\in \Gamma_{(u,v)}=\partial \omega_u\cap \Omega=\omega_v\cap \Omega$ (Corollary \ref{coro:One_component_positive}).

\medbreak

2. Next, taking into account the proof of Proposition \ref{prop:jump_condition_on_N} (see also Remark \ref{rem:at_most_hyperplane}), we have that $\Gamma_{(\bar u,\bar v)}$ is a hyperplane, $S^{N-1}$ splits into two half spheres $S^+$, $S^-$, and 
\[
\bar u_i = r g_i(\theta),\qquad \bar v_i= r h_i(\theta),
\]
where $g_i$, $h_i$ are all first eigenfunctions respectively on $S^+$ and $S^-$. Thus there exists $\nu$ so that \eqref{eq:u=x_1^+_v=x_1^-} holds.

\medbreak

3. As 
\[
\Delta \bar u_i=0\ \text{ in } \overline \Gamma^+:=\{x\cdot \nu>0\},\qquad \Delta \bar v_i=0\ \text{ in } \overline \Gamma^-:=\{x\cdot \nu<0\},
\]
and $\overline \Gamma:= \Gamma_{(\bar u,\bar v)}=\{x\cdot \nu=0\}$, given a ball $B_r(x_0)$ we have, by using the Rellich-type formula
\[
\div ((x-x_0)|\nabla \bar u_i|^2-2\langle x-x_0,\nabla \bar u_i \rangle\nabla \bar u_i=(N-2)|\nabla \bar u_i|^2-2\langle x-x_0,\nabla \bar u_i\rangle\Delta \bar u_i
\]
in $B_r(x_0)\cap \overline \Gamma^+$,
\begin{align*}
r\int_{\partial B_r(x_0)\cap \overline \Gamma^+} (|\nabla \bar u_i|^2-2(\partial_n \bar u_i)^2) d\sigma&=(N-2)\int_{B_r(x_0)\cap \overline \Gamma^+} |\nabla \bar u_i|^2\, dx\\
						&- \int_{B_r(x_0)\cap \overline \Gamma} |\nabla \bar u_i|^2 \langle \nu,x-x_0 \rangle\, d\sigma
\end{align*}
Analogously,
\begin{align*}
r\int_{\partial B_r(x_0)\cap \overline \Gamma^-} (|\nabla \bar v_i|^2-2(\partial_n \bar v_i)^2) d\sigma&=(N-2)\int_{B_r(x_0)\cap \overline \Gamma^-} |\nabla \bar v_i|^2\, dx\\
						&+ \int_{B_r(x_0)\cap \overline \Gamma} |\nabla \bar v_i|^2 \langle \nu,x-x_0 \rangle\, d\sigma.
\end{align*}
Therefore, combining this with the local Pohozaev--type identities \eqref{eq:pohozaev_for_bar_uv2},
\[
\sum_{i=1}^k \int_{B_r(x_0)\cap \overline \Gamma} \left( a_i |\nabla \bar u_i|^2-b_i |\nabla \bar v_i|^2\right)\langle \nu,x-x_0 \rangle \, d\sigma=0\qquad \forall x_0\in \R^N,\ r>0,
\]
and hence $\sum_{i=1}^k a_i \alpha_i^2=\sum_{i=1}^k b_i \beta_i^2$.
\end{proof}

\medbreak

We are now ready to improve the result of Lemma \ref{lemma:Gamma_locally_in_a_strip}, showing that $\Gamma_{(u,v)}\cap B_{R_0}(x_0)$ verifies the so called $(N-1)$--dimensional $(\delta, R_0)$-- Reifenberg flatness condition for every small $\delta$ and some $R_0=R_0(\delta)>0$.

\begin{corollary}\label{lemma:flatness_condition}
Within the previous framework, for any given $\delta>0$ there exists $R_0>0$ such that for every $x\in \Gamma_{(u,v)}\cap B_{R_0}(x_0)$ and $0<r<R_0$ there exists a hyperplane $H=H_{x,r}$ containing $x$ such that \footnote{Here $d_\Hh(A,B):=\max\{\sup_{a\in A}\text{dist}(a,B),\sup_{b\in B}\text{dist}(A,b)\}$ denotes the Hausdorff distance. Notice that $d_\Hh(A,B)\leq \delta$ if and only if $A\subseteq N_\delta(B)$ and $B\subseteq N_\delta (A)$, where $N_\delta(\cdot)$ is the closed $\delta$--neighborhood of a set. Observe moreover that when $H$ is a hyperplane, $N_\delta(H)$ is a strip $S(x,\nu,\delta)$ for $x\in H$ and a vector $\nu\in S^{N-1}$, orthogonal to $H$.}
\begin{equation}\label{eq:flatness_condition}
d_\Hh (\Gamma_{(u,v)}\cap B_r(x),H\cap B_r(x))\leq \delta r.
\end{equation}
\end{corollary}
\begin{proof}
In the ground of Theorem \ref{coro:hyperplane}, the proof follows exactly as the one of \cite[Lemma 5.3]{TavaresTerracini1}.
\end{proof}

\medbreak

The Reifenberg flat condition yields a local separation result.

\begin{corollary}\label{coro:positive}
Given $x_0\in \Reh_{(u,v)}$, there exists $R_0$ such that $\Gamma_{(u,v)}\cap B_{R_0}(x_0)=\Reh_{(u,v)}\cap B_{R_0}(x_0)$, and
\[
B_{R_0}(x_0)\setminus \Gamma_{(u,v)} \qquad \text{has exactly two connected components $\Omega_1,\Omega_2$.} 
\]
Moreover,
\[
\sum_{i=1}^k u_i^2>0,\ \sum_{i=1}^k v_i^2\equiv 0 \text{ on } \Omega_1\qquad \text{ and }\qquad \sum_{i=1}^k u_i^2=0,\ \sum_{i=1}^k v_i^2>0 \text{ on } \Omega_2.
\]
Finally, given $y\in \Gamma_{(u,v)}\cap B_{R_0}(x_0)$ and $0<r<R_0$ there exists a hyperplane $H=H_{y,r}$ containing $y$ and a unitary vector $\nu=\nu_{y,r}$ orthogonal to $H$ such that
\begin{equation}\label{eq:One_of_reifenberg_properties}
\{x+t\nu \in B_r(y):\ x\in H,\ t\geq \delta r\}\subset \Omega_1,\qquad \{x-t\nu \in B_r(y):\ x\in H,\ t\geq \delta r\}\subset \Omega_2.
\end{equation}
\end{corollary}
\begin{proof}
See \cite[Proposition 5.4]{TavaresTerracini1}, or \cite[Proposition 3.38]{TavaresThesis} for a detailed proof. For a result in the same direction, check \cite[Theorem 4.1]{HongWang}
\end{proof}

Combining \eqref{eq:One_of_reifenberg_properties} with Corollary \ref{lemma:flatness_condition}, we see that both $\Omega_1$ and $\Omega_2$ as defined above are $(\delta,R_0)$-- Reifenberg flat domains. We suppose that $\delta>0$ is small enough so that $\Omega_1$ and $\Omega_2$ are also NTA domains (check Definitions \ref{def:NTA} and \ref{def:Reifenberg} ahead, and \cite[Theorem 3.1]{KenigToro} for the proof of this result). 

\subsection{The set $\Reh_{(u,v)}$ is locally a regular hypersurface }\label{subsec:regularity}

Define
\[
U(x):=(\sqrt{a_1}u_1(x),\ldots, \sqrt{a_k}u_k(x)),\qquad V(x):=(\sqrt{b_1}v_1(x),\ldots, \sqrt{b_k}v_k(x)),
\]
and observe that locally $\Reh_{(u,v)}$ is characterized as being the set of zeros of $|U(x)|-|V(x)|$. We will prove the regularity of such set by proving by definition that $|U|-|V|$ is differentiable at each point of $\Reh_{(u,v)}$, having nonzero gradient (Theorem \ref{thm:Regularity!}).

We fix $x_0\in \Reh_{(u,v)}$ and work in $B_{R_0}(x_0)$ as in the previous subsection. We also recall that
\[
\Omega_1=\{x\in B_{R_0}(x_0):\ \sum_{i=1}^k u_i^2>0\}, \quad \text{ and } \quad \Omega_2=\{x\in B_{R_0}(x_0):\ \sum_{i=1}^k v_i^2>0\}.
\]
To simplify the notation, in this subsection we denote $\Gamma:=\Gamma_{(u,v)}\cap B_{R_0}(x_0)$.
Given $x\in \Omega_1$, define
\[
\Ucal(x)=\frac{U(x)}{|U(x)|}=\frac{(\sqrt{a_1}u_1(x),\ldots, \sqrt{a_k} u_k(x))}{\sqrt{a_1u_1^2(x)+\ldots+a_k u_k^2(x)}}
\]
and, for $x\in \Omega_2$,
\[
\Vcal(x)=\frac{V(x)}{|V(x)|}=\frac{(\sqrt{b_1}v_1(x),\ldots, \sqrt{b_k}v_k(x))}{\sqrt{b_1 v_1^2(x)+\ldots+b_kv_k^2(x)}}.
\]

Our main interest will be to evaluate $\Ucal,\Vcal$ at $\Gamma$. By  Corollary \ref{coro:One_component_positive} (or Proposition \ref{lemma:survives_blowup_then_positive} combined with Theorem \ref{coro:hyperplane}) we have that at least one component of $u$ and of $v$ is signed. Suppose from now on, without loss of generality, that $u_1>0$ in $\Omega_1$ and $v_1>0$ in $\Omega_2$. This crucial fact allows us to apply the regularity results shown in Appendix \ref{appendix:NTA}. More precisely, Lemma \ref{prop:u_i/u_1_alpha_Holder_generalcase} implies that $\Ucal$ and $\Vcal$ can be uniquely extended up to $\Gamma$, as soon as one observes that for instance $\Ucal$ can be rewritten as
\[
\Ucal=\frac{(\sqrt{a_1}, \sqrt{a_2}\frac{u_2}{u_1}\ldots, \sqrt{a_k} \frac{u_k}{u_1})}{\sqrt{a_1+a_2\left(\frac{u_2}{u_1}\right)^2\ldots+a_k \left(\frac{u_k}{u_1}\right)^2}}
\]
Moreover, from this proposition we also deduce that
\begin{equation}\label{eq:Ucal_alphaHolder}
|\Ucal (x)-\Ucal(x_0)|\leq C r^\alpha,\quad |\Vcal(x)-\Vcal(x_0)|\leq C r^\alpha,\qquad \forall x\in B_r(x_0),\ r<R_0.
\end{equation}

Observe that, formally, 
\[
\Ucal(y)=\lim_{x\to y} \frac{U(x)/\dist(x,\partial \Omega)}{|U(x)/\dist(x,\partial \Omega)|}=\frac{-\partial_n U(y)}{|\partial_n U(y)|},\quad  \Vcal(y)=\frac{-\partial_n V(y)}{|\partial_n V(y)|}\qquad \text{for $y\in \Gamma$}.
\] 
Rigorously, however, we can only for the moment observe that $U(y)$, for $y\in \Gamma$, is related to the blowup limits centered at $y$. 

\begin{lemma}\label{lemma:interpretation_of_Ucal_and_Vcal_at_boundary}
Take $y\in \Gamma$ and let  
\[
(\bar u,\bar v)=(\alpha_1(x\cdot \nu)^+,\ldots, \alpha_k(x\cdot \nu)^+, \beta_1(x\cdot \nu)^-,\ldots, \beta_k(x\cdot \nu)^-)\in \Bcal\Ucal_{y}.
\]  
Then
\[
\Ucal(y)=\frac{(\sqrt{a_1}\bar u_1,\ldots,\sqrt{a_k}\bar u_k)}{\sqrt{a_1 \bar u_1^2+\ldots+a_k \bar u_k^2}}=\frac{(\sqrt{a_1}\alpha_1,\ldots, \sqrt{a_k}\alpha_k)}{\sqrt{a_1\alpha_1^2+\ldots+a_k \alpha_k^2}}
\]
and
\[
\Vcal(y)=\frac{(\sqrt{b_1}\bar v_1,\ldots,\sqrt{b_k}\bar v_k)}{\sqrt{b_1 \bar v_1^2+\ldots+b_k \bar v_k^2}}=\frac{(\sqrt{b_1}\beta_1,\ldots, \sqrt{b_k}\beta_k)}{\sqrt{b_1\beta_1^2+\ldots+b_k \beta_k^2}}.
\]
\end{lemma}
\begin{proof}
Let $(\bar u,\bar v)\in \Bcal\Ucal_{y}$ be as in the statement, limit of a blowup sequence centered at $x_n\to y$ and with parameter $t_n\downarrow 0$. Take $\bar z\in \R^N$ such that $|(\bar u_1(\bar z),\ldots, \bar u_k(\bar z))|\neq 0$ and (for large $n$) $x_n+t_n \bar z \in \Omega_1$. Then
\begin{equation}\label{eq:interpretation_of_Ucal(y)}
\Ucal (y)=\lim_{n\to \infty}\frac{U(x_n+t_n \bar z)}{|U(x_n+t_n \bar z)|} = \lim_{n\to \infty}\frac{U(x_n+t_n \bar z)/\rho_n}{|U(x_n+t_n \bar z)/\rho_n|}=\frac{(\sqrt{a_1} \bar u_1,\ldots, \sqrt{a_k}\bar u_k)}{\sqrt{a_1 \bar u_1^2 + \ldots a_k \bar u_k^2}},
\end{equation}
and a similar property holds for $\Vcal$.
\end{proof}

\medbreak 

We now define two important auxiliary functions.
\begin{definition}
Given $x_0\in \Gamma$ we define
\[
u_{x_0}(x)=\Ucal(x_0)\cdot U(x),\qquad v_{x_0}(x)=\Vcal(x_0)\cdot V(x).
\]
\end{definition}

The ideia for considering these functions is the following: given a blowup limit $(\bar u,\bar v)$, the Almgren's monotonicity formula tells us that a certain linear combination of $\bar u_i$ extends in an harmonic way with a linear combination of $\bar v_i$ (Theorem \ref{coro:hyperplane}). However, those blowup limits may vary in general from point. Fixing $\Ucal,\Vcal$ at $x_0$ together with estimate \eqref{eq:Ucal_alphaHolder} will allow us to control in some sense this variation.

Let us see that $u_{x_0}, v_{x_0}$ are positive close to $x_0$, and check which problem they solve.

\begin{lemma}
There exists $r_0$ such that $u_{x_0}>0$ in $B_{r_0}(x_0)\cap \Omega_1$, $v_{x_0}>0$ in $B_{r_0}(x_0)\cap \Omega_2$. In particular, there exist \emph{positive} Radon measures $\Mcal_{x_0}$, $\Ncal_{x_0}$, both concentrated on $\Gamma$, such that
\[
-\Delta u_{x_0}=\sum_{i=1}^k \frac{\mu_i}{a_i} \Ucal_i(x_0) u_i- \Mcal_{x_0}, \quad -\Delta v_{x_0}=\sum_{i=1}^k \frac{\nu_i}{b_i} \Vcal_i(x_0) v_i- \Ncal_{x_0}\quad \text{ in } B_{r_0}(x_0).
\]
\end{lemma}
\begin{proof}
By using \eqref{eq:Ucal_alphaHolder} we see that
\[
u_{x_0}(x)=\Ucal(x)\cdot U(x)+(\Ucal(x_0)-\Ucal(x))\cdot U(x)\geq (1-C r_0^\alpha)|U(x)|>0 \quad \forall x\in B_{r_0}(x_0)\cap \Omega_1,
\]
whenever $r_0>0$ is small enough. The rest of the proof follows easily.
\end{proof}

\medbreak

Our aim now it to analyze the difference $\varphi_{x_0}:=u_{x_0}-v_{x_0}$ close to $x_0$, as a first step before analyzing $|U|-|V|$ at $x_0$. For that, let us take a blowup sequence centered at $y\in \Gamma\cap B_{r_0}(x_0)$. Define, given $t_n\to 0^+$,
\[
u_{x_0}^{y,t_n} (x):=\frac{u_{x_0}(y+t_n x)}{\sqrt{H(y,(u,v),t_n)}}, \qquad v_{x_0}^{y,t_n} (x):=\frac{v_{x_0}(y+t_n x)}{\sqrt{H(y,(u,v),t_n)}}.
\]
Up to a subsequence, there exists $ (\bar u,\bar v)\in \Bcal\Ucal_{y}$ such that
\begin{align*}
& u_{x_0}^{y,t_n} \to \Ucal(x_0)\cdot (\sqrt{a_1} \bar u_1,\ldots, \sqrt{a_k} \bar u_k)=\Ucal(x_0)\cdot \bar U=: \bar u_{x_0}^{y}\\
& v_{x_0}^{y,t_n} \to \Vcal(x_0)\cdot (\sqrt{b_1} \bar v_1,\ldots, \sqrt{b_k} \bar v_k)=\Vcal(x_0)\cdot \bar V=:\bar v_{x_0}^{y}.
\end{align*}
Observe that, if $y=x_0$, then by  Lemma \ref{lemma:interpretation_of_Ucal_and_Vcal_at_boundary} one has
\[
\bar u_y^{y}= \sqrt{a_1 \bar u_1^2+\ldots+a_k \bar u_k^2}, \quad \bar v_y^{y}= \sqrt{b_1 \bar v_1^2+\ldots+ b_k \bar v_k^2} \quad \text{ for some } (\bar u,\bar v)\in \Bcal\Ucal_{y}. 
\]
In such a case, by Theorem \ref{coro:hyperplane} we have
\[
\bar u_i= \alpha_i (x\cdot \nu)^+,\qquad \bar v_i=\beta_i (x\cdot \nu)^-
\]

and $\bar u_{y}^y-\bar v_y^y$ is harmonic.
In particular, on $\{x\cdot \nu=0\}$,
\begin{equation}\label{eq:equality of the gradients}
|\nabla \bar u_{y}^y|\equiv |\nabla \bar u_{y}^y(0)|=|\nabla \bar U(0)|=\left(\sum_{i=1}^k a_i\alpha_i ^2\right)^\frac{1}{2}=\left(\sum_{i=1}^k b_i \beta_i^2\right)^\frac{1}{2}= |\nabla \bar V(0)|= |\nabla \bar v_{y}^y(0)|\equiv |\nabla \bar v_y^y|
\end{equation}
\begin{lemma}
Under the previous notations, let 
\[
\Mcal_{x_0}^{y,t_n}(\cdot):=\frac{1}{\rho_n t_n^{N-2}}\Mcal_{x_0}(y+t_n \cdot),\quad \Ncal_{x_0}^{y,t_n}(\cdot):=\frac{1}{\rho_n t_n^{N-2}}\Ncal_{x_0}(y+t_n \cdot)
\]
with $\rho_n^2=H(y,(u,v),t_n)$, so that in $B_\frac{r_0}{t_n}(\frac{x_0-y}{t_n})$
\[
-\Delta u_{x_0}^{y,t_n}=t_n^2 \sum_{i=1}^k \frac{\mu_i}{a_i} \Ucal_i(x_0) u_{x_0}^{y,t_n} - \Mcal_{x_0}^{y,t_n}, \quad -\Delta v_{x_0}^{y,n}=t_n^2 \sum_{i=1}^k \frac{\nu_i}{b_i} \Vcal_i(x_0) v_{x_0}^{y,t_n} - \Ncal_{x_0}^{y,t_n} 
\]
Then there exists $\overline \Mcal_{x_0}^y$, $\overline \Ncal_{x_0}^y$ such that
\[
\Mcal_{x_0}^{y,t_n}\rightharpoonup \overline \Mcal_{x_0}^y,\quad \Ncal_{x_0}^{y,t_n}\rightharpoonup \overline \Ncal_{x_0}^y
\]
and moreover
\[
\overline \Mcal_{x_0}^y(E)=\int_{E\cap \overline \Gamma} |\nabla \bar u_{x_0}^y|\, d\sigma=|\nabla \bar u_{x_0}^y(0)| |E\cap \overline \Gamma|, \ \ \overline \Ncal_{x_0}^y(E)=\int_{E\cap \overline \Gamma} |\nabla \bar v_{x_0}^y|\, d\sigma=|\nabla \bar v_{x_0}^y(0)| |E\cap \overline \Gamma|,
\]
with $\overline \Gamma=\{u_{x_0}^y = \bar v_{x_0}^y=0\}$. In particular, when $y=x_0$, one has that the limiting measures coincide, that is
\begin{equation}\label{eq:the_limiting_measures_coincide}
\overline \Mcal_{y}^y \equiv \overline \Ncal_y^y.
\end{equation}
\end{lemma}
\begin{proof}
The proof is similar to the one of Theorem \ref{thm:convergence_of_blowupsequence} and of \eqref{eq:Reflection_Principle} in Theorem \ref{coro:hyperplane}, and so we omit it here.
\end{proof}

This result, together with \eqref{eq:Ucal_alphaHolder} and \eqref{eq:equality of the gradients}, allows us to compare $\Mcal_{x_0}$ and $\Ncal_{x_0}$ close to $x_0$.

\begin{lemma}
We have
\[
\Mcal_{x_0}\leq (1+C r^\alpha)\Ncal_{x_0} \quad \text{ and } \quad \Ncal_{x_0}\leq (1+C r^\alpha)\Mcal_{x_0} \qquad \text{ in } B_r(x_0)
\]
for some $C>0$ independent of $x_0$ and $r>0$. In particular, the new functions
\[
\overline \varphi_{x_0,r}:= (1+Cr^\alpha)u_{x_0}-v_{x_0},\quad \text{ and } \quad \underline \varphi_{x_0,r}:= u_{x_0}-(1+Cr^\alpha) v_{x_0}
\]
satisfy
\[
-\Delta \overline \varphi_{x_0,r} \leq (1+C r^\alpha)\sum_{i=1}^k \frac{\mu_i}{a_i} \Ucal_i(x_0) u_i-\sum_{i=1}^k \frac{\nu_i}{b_i} \Vcal_i(x_0) v_i
\]
and
\[
-\Delta \underline \varphi_{x_0,r} \geq \sum_{i=1}^k \frac{\mu_i}{a_i}\Ucal_i(x_0) u_i- (1+C r^\alpha)\sum_{i=1}^k \frac{\nu_i}{b_i} \Vcal_i(x_0) v_i
\]
in $B_{r}(x_0)$.
\end{lemma}

\begin{proof}
1. We claim that
\begin{equation}\label{eq:auxiliary_for_comparison between measures}
\lim_{t\rightarrow 0}\frac{\Mcal_{x_0}(\overline B_t(y))}{\Ncal_{x_0}(\overline B_t(y))}\geq 1-C r^\alpha \qquad \text{ for every }y\in \Gamma_{(u,v)}\cap B_{r}(x_0).
\end{equation}
Fix $y\in \Gamma_{(u,v)}\cap B_{r}(x_0)$ and consider any arbitrary sequence $t_n\downarrow 0$. Using the previous notations, take $u_{x_0}^{y,t_n}(x)$, $v_{x_0}^{y,t_n}$ and the rescaled measures $\Mcal_{x_0}^{y,t_n}$, $\Ncal_{x_0}^{y,t_n}$, and some corresponding blowup limits $\bar u_{x_0}^{y}=\Ucal(x_0)\cdot \bar U$, $\bar v_{x_0}^y=\Vcal(x_0)\cdot \bar V$, $\overline \Mcal_{x_0}^y$, $\overline \Ncal_{x_0}^y$.
One has (by \cite[Section 1.6, Theorem 1]{EvansGariepy})
\begin{equation}\label{eq:fraction_between_measures}
\lim_n \frac{\Mcal_{x_0}(\overline B_{t_n}(y))}{\Ncal_{x_0}(\overline B_{t_n}(y))}=\lim_n\frac{\Mcal_{x_0}^{y,t_n}(\overline B_1(0))}{\Ncal_{x_0}^{y,t_n}(\overline B_1(0))}=\frac{\overline \Mcal_{x_0}^y(\overline B_1(0))}{\overline \Ncal_{x_0}^y(\overline B_1(0))}=\frac{|\nabla \bar u_{x_0}^y(0)|}{|\nabla \bar v_{x_0}^y(0)|}.
\end{equation}
On the other hand, 
\[
\left| |\Ucal(x_0)\cdot \nabla \bar U(0)|-|\Ucal(y)\cdot \nabla \bar U(0)|\right| \leq   |(\Ucal(x_0)-\Ucal(y))\cdot \nabla \bar U(0)|\leq C r^\alpha |\nabla \bar U(0)|=Cr^\alpha |\nabla \bar u_y^y(0)|,
\]
so
\[
\left||\nabla \bar u_{x_0}^y(0)|-|\nabla \bar u_y^y(0)| \right| \leq Cr^\alpha |\nabla \bar u_y^y(0)|
\]
where we have used \eqref{eq:Ucal_alphaHolder} and \eqref{eq:equality of the gradients}. Analogously, one has
\[
\left||\nabla \bar v_{x_0}^y(0)|-|\nabla \bar v_y^y(0)| \right| \leq Cr^\alpha |\nabla \bar v_y^y(0)|
\]
Combining these two inequalities with $|\nabla \bar u_y^y(0)|=|\nabla \bar v_y^y(0)|$ (cf. \eqref{eq:equality of the gradients}) and \eqref{eq:fraction_between_measures} yields the validity of our initial claim \eqref{eq:auxiliary_for_comparison between measures}. 

\medbreak

\noindent 2. By the previous point, we have $D_{\Ncal_{x_0}}\Mcal_{x_0}(y)\leq (1+C r^\alpha)$ for $\Ncal_{x_0}$--a.e. $y\in B_{r_0}(x_0)$  and hence the Radon-Nikodym Decomposition Theorem (check for instance \cite[Sect. 1.6, Th. 3]{EvansGariepy}) yields that for every Borel set $E\subseteq B_{r_0}(x_0)$
\[
\Mcal_{x_0}(E)=(\Mcal_{x_0})_s(E)+ \int_E D_{\Ncal_{x_0}}\Mcal_{x_0}\, d\Ncal_{x_0} \leq (1 + C r^\alpha) \Ncal_{x_0}(E).
\]
(where $(\Mcal_{x_0})_s\geq 0$ represents the singular part of $\Mcal_{x_0}$ with respect to $\Ncal_{x_0}$), and this proves the first inequality in the proposition. The proof of the other inequality follows in the exact same way.
\end{proof}

\medbreak

Recall that $\varphi_{x_0}:=u_{x_0}-v_{x_0}$ and let $\psi_{x_0,r}$, for each small $r>0$, be the solution of
\[
\begin{cases}
-\Delta \psi_{x_0,r}=\sum_{i=1}^k \frac{\mu_i}{a_i} \Ucal_i(x_0)u_i-\sum_{i=1}^k \frac{\nu_i}{b_i}\Vcal_i(x_0)v_i  &\text{ in } B_r(x_0)\\
		\psi_{x_0,r}=\varphi_{x_0}=u_{x_0}-v_{x_0} \qquad 			& \text{ on } \partial B_r(x_0).
\end{cases}
\]
As we will see ahead, $\nabla \psi_{x_0,r}(x_0)$ will be a good approximation, for $r$ small, to a normal vector of $\Reh_{(u,v)}$ at $x_0$. Having that in mind, our next aim is to estimate the difference $\varphi_{x_0}-\psi_{x_0,r}$  in terms of $r$, in each ball $B_r(x_0)$. Introduce the auxiliary functions $\overline \psi_{x_0,r}$, $\underline \psi_{x_0,r}$ defined in $B_r(x_0)$ by the conditions:
\[
\begin{cases}
-\Delta \overline \psi_{x_0,r}=(1+Cr^\alpha)\sum_{i=1}^k \frac{\mu_i}{a_i} \Ucal_i(x_0)u_i-\sum_{i=1}^k \frac{\nu_i}{b_i}\Vcal_i(x_0)v_i  &\text{ in } B_r(x_0)\\
	    \overline \psi_{x_0,r}=\overline \varphi_{x_0,r}=(1+C r^\alpha)u_{x_0}-v_{x_0} \qquad 			& \text{ on } \partial B_r(x_0).
\end{cases}
\]
and
\[
\begin{cases}
-\Delta \underline \psi_{x_0,r}=\sum_{i=1}^k \frac{\mu_i}{a_i} \Ucal_i(x_0)u_i-(1+Cr^\alpha)\sum_{i=1}^k \frac{\nu_i}{b_i}\Vcal_i(x_0)v_i  &\text{ in } B_r(x_0)\\
		\underline \psi_{x_0,r}=\underline \varphi_{x_0,r}=u_{x_0}-(1+Cr^\alpha)v_{x_0} \qquad 			& \text{ on } \partial B_r(x_0).
\end{cases}
\]
By the comparison principle and since $u_{x_0},v_{x_0}>0$, we have that
\[
\underline \psi_{x_0,r}\leq \underline \varphi_{x_0,r}\leq \varphi_{x_0}\leq \overline \varphi_{x_0,r}\leq \overline \psi_{x_0,r} \qquad \text{ in } B_r(x_0),
\]
and moreover (since also $\sum_{i=1}^k \frac{\mu_i}{a_i} \Ucal_i(x_0)u_i, \sum_{i=1}^k \frac{\nu_i}{b_i}\Vcal_i(x_0)v_i>0$)
\[
\underline \psi_{x_0,r}\leq \psi_{x_0,r}\leq \overline \psi_{x_0,r} \qquad \text{ in } B_r(x_0).
\]
 Now, since $\overline \psi_{x_0,r}-\underline \psi_{x_0,r}$ satisfies the problem
\[
\begin{cases}
-\Delta (\overline \psi_{x_0,r}-\underline \psi_{x_0,r})= C r^\alpha \sum_{i=1}^k\left( \frac{\mu_i}{a_i}\Ucal_i(x_0)u_i+\frac{\nu_i}{b_i}\Vcal_i(x_0)v_i\right) & \text{ in } B_r(x_0)\\
\overline \psi_{x_0,r}-\underline \psi_{x_0,r}=Cr^\alpha (u_{x_0}+v_{x_0}) 					& \text{ on } \partial B_r(x_0).
\end{cases}
\]
we have, by using also the maximum principle \cite[Theorem 3.7]{GilbargTrudinger} and the fact that $u_i$, $v_i$ are Lipschitz continuous and zero at $x_0$,
\begin{align}
\|\varphi_{x_0}-\psi_{x_0,r}\|_{L^\infty(B_r(x_0))}&\leq \|\overline \psi_{x_0,r}-\underline \psi_{x_0,r}\|_{L^\infty(B_r(x_0))}\\
									&\leq C' r^\alpha \|u_{x_0}+v_{x_0}\|_{L^\infty(\partial B_r(x_0))}+C'r^\alpha \sum_{i=1}^k \|u_i+v_i\|_{L^\infty(B_r(x_0))}\\
									&\leq C'' r^{1+\alpha} \qquad \forall x\in B_r(x_0). \label{eq:|phi-psi|_leq_diff_of_psis}
\end{align}

\begin{proposition}
There exists
\[
\nu(x_0):=\lim_{r\to 0} \nabla \psi_{x_0,r}(x_0).
\]
Moreover, the map $\Gamma\to \R^N$, $x_0\mapsto \nu(x_0)$ is H\"older continuous of order $\alpha$.
\end{proposition}

\begin{proof}
1. Fix $r>0$ and let us prove that $\{\nabla \psi_{x_0,r/2^k}(x_0)\}_k$ is a Cauchy sequence. One has 
\[
\|\varphi_{x_0}-\psi_{x_0,r}\|_{L^\infty(B_r(x_0))}\leq C r^{1+\alpha},\quad \|\varphi_{x_0}-\psi_{x_0,r/2}\|_{L^\infty(B_{r/2}(x_0))}\leq \frac{C}{2^{1+\alpha}}r^{1+\alpha},
\]
which implies that
\begin{align*}
\|\psi_{x_0,r}-\psi_{x_0,r/2}\|_{L^\infty(B_{r/2}(x_0))}&\leq \| \psi_{x_0,r}-\varphi_{x_0}\|_{L^\infty(B_{r/2}(x_0))}+\| \psi_{x_0,r/2}-\varphi_{x_0}\|_{L^\infty(B_{r/2}(x_0))}\\
									&\leq C r^{1+\alpha}+\frac{C}{2^{1+\alpha}}r^{1+\alpha}=C\left(1+\frac{1}{2^{1+\alpha}}\right)r^{1+\alpha}
\end{align*}
As $\psi_{x_0,r}-\psi_{x_0,r/2}$ is harmonic in $B_{r/2}(x_0)$, then by using classical estimates we have
\[
|\nabla \psi_{x_0,r}(x_0)-\nabla \psi_{x_0,r/2}(x_0)|\leq C'(1+\frac{1}{2^{1+\alpha}})r^\alpha.
\]
Iterating this procedure, we obtain
\begin{equation}\label{eq:auxiliary_for_nu(x_0)}
|\nabla \psi_{x_0,r/2^{n+m}}(x_0)-\nabla \psi_{x_0,r/2^m}(x_0)|\leq C' (1+\frac{1}{2^{1+\alpha}}) \left( \sum_{i=m}^{m+n-1}\frac{1}{2^{\alpha i}}\right) r^\alpha.
\end{equation}
Thus, for each $r>0$, $\{\nabla \psi_{x_0,r/2^k}(x_0)\}_k$ is a Cauchy sequence, and thus $\nu(x_0):=\lim_{k\to \infty} \nabla \psi_{x_0,r/2^k}(x_0)$ exists.

\medbreak

\noindent 2. Now let $r_n\to 0$ be an arbitrary sequence. Suppose, without loss of generality, that $\frac{r}{2^{n+1}}<r_n\leq \frac{r}{2^n}$ (if not, one can pass to a subsequence of $\{r/2^n\}_n$. Then
\[
\|\varphi_{x_0}-\psi_{x_0,r_n}\|_{L^\infty(B_{r/2^{n+1}}(x_0))} \leq C r_n^{1+\alpha},\quad \|\varphi_{x_0}-\psi_{r/2^{n+1}}\|_{L^\infty(B_{r/2^{n+1}}(x_0))}\leq \frac{C}{2^{(n+1)(1+\alpha)}}r^{1+\alpha}.
\]
and so
\begin{align*}
\|\psi_{x_0,r_n}-\psi_{x_0,r/2^{n+1}}\|_{L^\infty(B_{r/2^{n+1}}(x_0))}&\leq C r_n^{1+\alpha}+\frac{C}{2^{(n+1)(1+\alpha)}}r^{1+\alpha}\\
													&\leq \frac{C}{2^{n(1+\alpha)}}r^{1+\alpha}+\frac{C}{2^{(n+1)(1+\alpha)}}r^{1+\alpha}.
\end{align*}
Once again by classical estimates,
\[
|\nabla \psi_{x_0,r_n}(x_0)-\nabla \psi_{x_0,r/2^{n+1}}(x_0)| \leq C\left(\frac{1}{2^{n(1+\alpha)}}+\frac{1}{2^{(n+1)(1+\alpha)}}\right) r^\alpha 
\]
and thus
\[
|\nabla \psi_{x_0,r_n}(x_0)-\nu(x_0)|\leq | \nabla \psi_{x_0,r_n}(x_0)-\nabla \psi_{x_0,r/2^{n+1}}(x_0)|+|\nabla \psi_{x_0,r/2^{n+1}}(x_0)-\nu(x_0)| 
\] and now the conclusion follows by taking the limit as $n\to \infty$.

\medbreak

\noindent 3. Finally, let us prove that $x_0\mapsto \nu(x_0)$ is H\"older continuous of order $\alpha$. Take $x_0,x_1\in \Gamma$ and take $r:=|x_0-x_1|$. From \eqref{eq:auxiliary_for_nu(x_0)} we have show that there exists $C>0$ (independent of the point and of the radius) such that
\[
|\nabla \psi_{x_i,r}-\nu(x_i)|\leq C r^\alpha.
\]
On the other hand, the H\"older continuity of $\Ucal$, $\Vcal$ one can easily show that $\|\varphi_{x_1}-\varphi_{x_0}\|_{L^\infty(B_r(x_0)\cap B_r(x_1))}\leq C r^{1+\alpha}$.
Moreover, $\| \varphi_{x_i}-\psi_{r,x_i}\|_{L^\infty(B_r(x_0))}\leq C r^{1+\alpha}$. By combining all this we obtain 
\begin{equation*}\label{eq:auxiliary_nu(x_0)}
|\nabla \psi_{x_0,r}(x_0)-\nabla \psi_{x_1,r}(x_0)|\leq C r^\alpha,
\end{equation*}
and the proof is complete.
\end{proof}

\begin{remark}\label{rem:other_definition_of_nu(x_0)}
For future reference, we observe that from the previous proof it is also clear that, given $x_r\in B_{r/2}(x_0)$, we also have
\[
\nabla \psi_{x_0,r}(x_r)\to \nu(x_0) \qquad \text{ as } r\to 0.
\]
\end{remark}

\medbreak

The following lemma relates $\nu(x_0)$ some blowup limits, providing at the same time that $\nu(x_0)\neq 0$.
\begin{lemma}
Given $x_0\in \Gamma_{(u,v)}$. Then for every $(\bar u,\bar v)\in \Bcal\Ucal_{x_0}$ with fixed center $x_0$ we have that
\[
\bar u_i=\alpha_i \left(x\cdot \frac{\nu(x_0)}{|\nu(x_0)|} \right)^+\quad \text{ and }\quad  \bar v_i=\beta_i \left(x\cdot \frac{\nu(x_0)}{|\nu(x_0)|}\right)^-,
\] 
for some $\alpha_i,\beta_i\in \R$ (not all zero). In particular, $\nu(x_0)\neq 0$.
\end{lemma}

\begin{proof}
Take $\bar u_i,\bar v_i$ coming from a limit of $t_n\to 0$ with fixed center $x_0$, and define
\[
\varphi_{x_0,n}(x):=\frac{u_{x_0}(x_0+t_n x)}{\sqrt{H(x_0,(u,v),t_n)}}-\frac{v_{x_0}(x_n+t_n x)}{\sqrt{H(x_0,(u,v),t_n)}} \qquad \text{ in } B_1(0),
\]
which converges to 
\[
\Ucal(x_0)\cdot (\sqrt{a_1} \bar u_1,\ldots,\sqrt{a_k} \bar u_k)-\Vcal(x_0)\cdot (b_1 \bar v_1,\ldots, b_k \bar v_k)=\sqrt{\sum_{i=1}^k a_i \alpha_i^2}(x\cdot \tilde \nu)=\sqrt{\sum_{i=1}^k b_i \beta_i^2}(x\cdot \tilde \nu)
\]
for some $\alpha_i,\beta_i$ ($i=1,\ldots, k$) and $\tilde \nu\in S^{N-1}$. On the other hand, we have that
\[
\widetilde \psi_{x_0,t_n}(x):=\frac{\psi_{x_0,t_n}(x_0+t_n x)}{\sqrt{H(x_0,(u,v),t_n)}},\qquad \text{defined in } B_1(0),
\]
coincides at the boundary $\partial B_1(0)$ with $\varphi_{x_0,n}(x)$, and hence it converges strongly $C^{2,\gamma}$ (for every $\gamma\in (0,1)$) to $\sqrt{\sum_{i=1}^k a_i \alpha_i^2}(x\cdot \tilde \nu)$. In particular,
\[
\sqrt{\sum_{i=1}^k a_i \alpha_i^2} \tilde \nu=\lim_n \nabla \widetilde \psi_{x_0,t_n}(0)=\lim_n \frac{t_n}{\sqrt{H(x_0,(u,v),t_n)}}\nabla \psi_{x_0,t_n}(x_0)=A \nu(x_0),
\]
with $A\neq 0$ being the limit of $t_n/\sqrt{H(x_0,(u,v),t_n)}$. Hence $\nu(x_0) \parallel \tilde \nu$.
\end{proof}

We are now ready to prove the regularity of the set $\Reh_{(u,v)}$.

\begin{theorem}\label{thm:Regularity!} The map 
\[
|U(x)|-|V(x)|=\sqrt{a_1 u_1^2+\ldots+a_k u_k^2}-\sqrt{b_1 v_1^2+\ldots+b_k v_k^2}
\] 
is differentiable at each $x_0 \in \Reh_{(u,v)}$, with 
\begin{equation}\label{eq:final_extremality_condition}
\nabla \left(|U|-|V|\right)(x_0)=\nu(x_0).
\end{equation}
In particular, the set $\Reh_{(u,v)}$ is locally a $C^{1,\alpha}$--hypersurface, for some $\alpha\in (0,1)$.
\end{theorem}

\begin{proof}
We prove this by definition, namely we show that
\begin{equation*}\label{eq:auxiliary_final_proof}
\frac{|U(x)|-|V(x)|-\nu(x_0)\cdot (x-x_0)}{|x-x_0|}\to 0 \qquad \text{ as } x\to x_0.
\end{equation*}
In the following, take $r:=2|x-x_0|$. We have
\[
\frac{|\varphi_{x_0}(x)-\psi_{x_0,r}(x)|}{|x-x_0|}\leq C r^{\alpha}\to 0,
\]
and
\[
\frac{|\psi_{x_0,r}(x_0)|}{|x-x_0|}=\frac{|\psi_{x_0,r}(x_0)-\varphi_{x_0}(x_0)|}{|x-x_0|}\leq C r^{\alpha}\to 0.
\]
by \eqref{eq:|phi-psi|_leq_diff_of_psis}. Moreover,
\[
\frac{\psi_{x_0,r}(x)-\psi_{x_0,r}(x_0)-\nu(x_0)\cdot (x-x_0)}{|x-x_0|}=(\nabla \psi_{x_0,r}(x_r)-\nu(x_0))\cdot \frac{x-x_0}{|x-x_0|}\to 0
\]
by Remark \ref{rem:other_definition_of_nu(x_0)}.
Finally, \eqref{eq:Ucal_alphaHolder} and the Lipschitz continuity of $U$ and $V$ yield
\[
\frac{(\Ucal(x)-\Ucal(x_0))\cdot U(x)}{|x-x_0|}\to 0 \quad \frac{(\Vcal(x)-\Vcal(x_0))\cdot V(x)}{|x-x_0|}\to 0,
\]
and so
\begin{multline*}
\frac{\left|\, |U(x)|-|V(x)|-\nu(x_0)\cdot (x-x_0)\right|}{|x-x_0|} =\frac{|\,\Ucal(x)\cdot U(x)-\Vcal(x)\cdot V(x)-\nu(x_0)\cdot (x-x_0)|}{|x-x_0|}\\
\leq \frac{|\varphi_{x_0}(x)-\nu(x_0)\cdot (x-x_0)|}{|x-x_0|}+\textrm{o}(1)\\
\leq \frac{|\varphi_{x_0}(x)-\psi_{x_0,r}(x)|}{|x-x_0|}+ \frac{|\psi_{x_0,r}(x_0)|}{|x-x_0|}+\frac{|\psi_{x_0,r}(x)-\psi_{x_0,r}(x_0)-\nu(x_0)\cdot (x-x_0)|}{|x-x_0|}+\textrm{o}(1)
\end{multline*}
converges to 0. The last statement of the theorem is now a simple consequence of the Implicit Function Theorem, since $\nu(x_0)\neq 0$.
\end{proof}

\begin{proof}[Proof of Theorem \ref{thm:3rdmain} completed.] Items (i), (ii) and (iii) are a consequence of Theorem \ref{thm:minimizer_for_c_beta} and Corollaries \ref{coro:characterization_of_c_infty}, \ref{coro:Lipschiz} and \ref{coro:One_component_positive}. In item (iv), the fact that $(\bar \omega_1,\ldots, \bar \omega_m)$ solves $c_\infty$ is a consequence of Corollary \ref{coro:characterization_of_c_infty}, while its regularity is provided by Theorem \ref{thm:Regularity!} (for $N=2$, the more precise statement for $\Gamma\setminus \Reh$ follows also from Theorem \ref{thm:hausdorff_measures_of_nodalsets}, while the equal angle property, at this point, follows exactly as in and \cite[p. 305]{TavaresTerracini1}). Finally, the extremality condition in (v) comes out from \eqref{eq:final_extremality_condition}.
\end{proof}


\section{The limit of the approximating solutions as $p\to \infty$. Existence and regularity results for the initial problem}\label{sec:limit_as_p_to_infty}

\subsection{Properties of the approximating solutions $u_p$, $v_p$. Proof of Theorems \ref{thm:first_main_result} and \ref{thm:2ndmain}}\label{subsec:first_subsec_of_Chapter4}

We are now ready to prove our first two main results, Theorems \ref{thm:first_main_result} and \ref{thm:2ndmain}. Until now, we have proved existence and regularity for the approximate problem \eqref{eq:general_OPP_with_psi} and so, in particular, for \eqref{eq:OPPsecondary}. As explained in the introduction, a solution of our initial model problem \eqref{eq:OPPmain} will be obtained by taking $p\to +\infty$ in \eqref{eq:OPPsecondary}. Let us summarize what we have proved so far regarding the latter problem.

Given $p\in \N$, we have shown the existence of $u_p=(u_{1p}, \ldots, u_{kp})$, $v_p=(v_{1p},\ldots, v_{kp})$, with components in $H^1_0(\Omega)\cap C^{0,1}(\overline \Omega)$, such that $u_p$ and $v_p$ are segregated (in the sense that $u_{ip}\cdot v_{jp}\equiv 0$ for every $i,j$), they are $L^2$--normalized
\[
\int_\Omega u_{ip}^2\, dx=\int_\Omega v_{ip}^2\, dx=1 \qquad \forall i=1,\ldots, k,
\]
and mutually $H^1_0$--orthogonal, 
\[
\int_\Omega \nabla u_{ip} \cdot \nabla u_{jp}\, dx=\int_\Omega \nabla v_{ip}\cdot \nabla v_{jp}\, dx=0,\qquad \forall i \neq j.
\]
We may assume that the components are ordered in the following way
\begin{equation}\label{eq:ordering_of_u1,...,up}
\int_\Omega |\nabla u_{1p}|^2\, dx\leq \ldots \leq \int_\Omega |\nabla u_{kp}|^2\, dx, \qquad \int_\Omega |\nabla v_{1p}|^2\, dx\leq \ldots \leq \int_\Omega |\nabla v_{kp}|^2\, dx.
\end{equation}
Moreover, there exist coefficients $a_{1p},\ldots, a_{kp}, b_{1p},\ldots, b_{kp}>0$ and $ \mu_{1p},\ldots, \mu_{kp},\nu_{1p},\ldots, \nu_{kp}>0$ such that 
\[
-a_{ip}\Delta u_{ip}=\mu_{ip}u_{ip} \quad \text{ in } \omega_{u_p},\qquad -b_{ip}\Delta v_{ip}=\nu_{ip}v_{ip} \quad \text{ in } \omega_{v_p}
\]
with 
\[
\omega_{u_p}=\left\{|u_p|>0\right\}=\Omega \setminus \overline{\left\{ |v_p|>0\right\}}, \quad \omega_{v_p}=\left\{|v_p|>0\right\}=\Omega \setminus \overline{\left\{ |u_p| >0\right\}},
\]
while the coefficients $a_{ip}$ and $b_{ip}$ are given by the following formulas: \footnote{Recall that $a_{ip}=\frac{\partial \psi}{\partial \xi_{i}}(\|u_{1p}\|^2,\ldots, \|u_{kp}\|^2)$ and $b_{ip}=\frac{\partial \psi}{\partial \xi_{i}}(\|v_{1p}\|^2,\ldots, \|v_{kp}\|^2)$, for $\psi(\xi_1,\ldots, \xi_k)=\left(\sum_{i=1}^k \xi_i^p\right)^{1/p}$.}
\[
a_{ip}=\frac{\left(\int_\Omega |\nabla u_{ip}|^2\, dx\right)^{p-1}}{\left(\sum_{j=1}^k \left(\int_\Omega |\nabla u_{jp}|^2\, dx\right)^p\right)^\frac{p-1}{p}},\qquad  b_{ip}=\frac{\left(\int_\Omega |\nabla v_{ip}|^2\, dx\right)^{p-1}}{\left(\sum_{j=1}^k \left(\int_\Omega |\nabla v_{jp}|^2\, dx\right)^p\right)^\frac{p-1}{p}}
\]
From \eqref{eq:ordering_of_u1,...,up}, we have the ordering
\[
0<a_{1p}\leq \ldots \leq a_{kp},\qquad 0<b_{1p} \leq \ldots \leq b_{kp}.
\]
The partition $(\omega_{u_p},\omega_{v_p})$ solves \eqref{eq:OPPsecondary}, and 
\begin{equation}\label{eq:functions_achieving_c_p}
\begin{split}
\left(\sum_{i=1}^k \left(\int_\Omega |\nabla u_{ip}|^2\, dx\right)^p\right)^{1/p}+&\left(\sum_{i=1}^k \left(\int_\Omega |\nabla v_{ip}|^2\, dx\right)^p\right)^{1/p}\\ 
&=  \Bigl(\sum_{j=1}^{k} (\lambda_j(\omega_{u_p}))^p\Bigr)^{1/p}+\Bigl(\sum_{j=1}^{k} (\lambda_j(\omega_{v_p}))^p\Bigr)^{1/p}\\
& = \inf_{(\omega_1,\omega_2)\in \Peh_2(\Omega)} \sum_{i=1}^2 \Bigl(\sum_{j=1}^{k} (\lambda_j(\omega_i))^p\Bigr)^{1/p}=:c_p. 
\end{split}
\end{equation}
Finally, the following Pohozaev--type formula holds: given $x_0\in \Omega$ and $r\in (0,\dist(x_0,\partial \Omega))$, we have that
\begin{align}\label{eq:Pohozaev--type_for_p}
(2-N) \sum_{i=1}^k \int_{B_r(x_0)} (a_{ip}|\nabla u_{ip}|^2+b_{ip} |\nabla v_{ip}|^2)\, dx\\
=\sum_{i=1}^k \int_{\partial B_r(x_0)} a_{ip} r(2(\partial_n u_{ip})^2-|\nabla u_{ip}|^2)\, d\sigma +\sum_{i=1}^k \int_{\partial B_r(x_0)} b_{ip} r (2(\partial_n v_{ip})^2-|\nabla v_{ip}|^2)\, d\sigma\\
+\sum_{i=1}^k \int_{\partial B_r(x_0)} r(\mu_{ip} u_{ip}^2+\nu_{ip} v_{ip}^2)\, d\sigma-\sum_{i=1}^k \int_{B_r(x_0)} N(\mu_{ip} u_{ip}^2+\nu_{ip} v_{ip}^2)\, dx.
\end{align}

We have also proved that, up to a set of Hausdorff dimension at most $N-2$ (corresponding to $\Sigma_{(u_p,v_p)}$) the common boundary $\partial \omega_{u_p}\cap \Omega= \partial \omega_{v_p}\cap \Omega$ is locally a hypersurface of class $C^{1,\alpha}$, for some $0<\alpha<1$.

\begin{lemma}\label{lemma:upper_inequality c_p leq c_infty}
We have $c_p\leq k^{1/p} c_\infty$. In particular, $\|u_p\|_{H^1_0}, \|v_p\|_{H^1_0}$ are bounded independently of $p\in \N$.
\end{lemma}

\begin{proof}
The comparison between the levels follows easily from the estimate
\[
\left(\sum_{j=1}^k \xi_j^p\right)^{1/p}\leq k^{1/p} \max\{\xi_1,\ldots, \xi_k\} \qquad \text{ for } \xi_1,\ldots, \xi_k\geq 0.
\]
The  $H^1_0$--boundedness of $u_p$ and $v_p$ then follows directly from \eqref{eq:functions_achieving_c_p}.
\end{proof}

\medbreak

Thus, we have guaranteed the existence of $\utilde=(\utilde_1,\ldots, \utilde_k)$ and $\vtilde=(\vtilde_1,\ldots, \vtilde_k)$ such that (up to a subsequence)
\[
u_p\ \rightharpoonup \utilde,\quad v_p \rightharpoonup \vtilde \qquad \text{ in } H^1_0(\Omega;\R^k),
\]
and
\[
u_p\ \to \utilde,\quad v_p \to \vtilde \qquad \text{ in } L^2(\Omega;\R^k).
\]
In particular, $\utilde_i\cdot \vtilde_j\equiv 0$ a.e. in $\Omega$ $\forall i,j$, and
\begin{equation}\label{eq:the limiting components are L^2 orthogonal}
\int_\Omega \utilde_i \utilde_j\, dx =\int_\Omega \vtilde_i\vtilde_j\, dx=\delta_{ij}.
\end{equation}
Moreover, it is easy to check that

\begin{equation}\label{eq:weightsbounds}
\left[ \frac{\int_\Omega |\nabla u_{ip}|^2\, dx}{k^{1/p}\int_\Omega |\nabla u_{kp}|^2\, dx}\right]^{p-1}\leq a_{ip}\leq 1,\qquad \left[ \frac{\int_\Omega |\nabla v_{ip}|^2\, dx}{k^{1/p}\int_\Omega |\nabla v_{kp}|^2\, dx}\right]^{p-1}\leq b_{ip}\leq 1.
\end{equation}

Therefore, there exists $1\leq l\leq k$ such that (up to a subsequence)
\begin{equation}\label{a_ip_to_0_and_non_zero}
a_{ip},b_{ip}\to 0,\quad \forall i=1,\ldots, l-1; \qquad  a_{ip}\to \atilde_i\neq 0, \ \ \ b_{ip} \to \btilde_i\neq 0 \qquad \forall i=l,\ldots, k 
\end{equation}
(observe that it is not necessarily true that the same number of components $a_i$ and $b_i$ vanish, and here we just suppose it without loss of generality to make the notations simpler).
As 
\[
a_{ip}\int_\Omega |\nabla u_{ip}|^2\, dx=\mu_{ip},\qquad b_{ip}\int_\Omega |\nabla v_{ip}|^2\, dx=\nu_{ip},
\]
also
\[
\mu_{ip},\nu_{ip}\to 0,\quad \forall i=1,\ldots, l-1; \qquad  \mu_{ip}\to \mutilde_i\neq 0, \ \ \ \nu_{ip} \to \nutilde_i\neq 0 \qquad \forall i=l,\ldots, k.
\]

In the rest of this section, our aim is to show that:
\begin{enumerate}
\item[(H1)] \label{item:goal1} $u_{ip}\to \utilde_i$ and $v_{ip}\to \vtilde_i$ in $C^{0,\gamma}\cap H^1_0(\Omega)$ for every $0<\gamma<1$, $\forall i=l,\ldots, k$;
\item[(H2)] \label{item:goal2} $\utilde_i$ and $\vtilde_i$ are Lipschitz continuous, for every $i=l,\ldots, k$;
\item[(H3)] \label{item:goal3} The nodal set $\Gamma^l_{(\utilde,\vtilde)}:=\{x\in \Omega: \utilde_i(x)=\vtilde_i(x)=0, \ \ \forall i=l,\ldots,k\}$ has Hausdorff dimension at most $N-1$, the pair  
\begin{equation}\label{eq:optimalpartition_final!}
(\widetilde \omega_1,\widetilde \omega_2):=\left( \interior\left(\, \overline{ \left\{\sum_{i=l}^k \utilde_i^2>0 \right\} } \, \right), \interior\left(\, \overline{ \left\{\sum_{i=l}^k \vtilde_i^2>0\right\} }\, \right)\right),
\end{equation}
belongs to $\Peh_2(\Omega)$, and it is a regular partition in the sense of Definition \ref{definition:regular_partition}. Moreover, $\partial \widetilde \omega_1\cap \Omega=\partial \widetilde \omega_2\cap \Omega$ and for each regular point $x_0\in \Reh$,
\[
\mathop{\lim_{x\to x_0}}_{x\in \widetilde \omega_1} \sum_{i=l}^k \tilde a_i |\nabla \utilde_i(x)|^2=\mathop{\lim_{x\to x_0}}_{x\in \widetilde \omega_2} \sum_{i=l}^k \tilde b_i |\nabla \vtilde_i(x)|^2\neq 0.
\]
\end{enumerate}
The main idea is that these properties only hold for $i=1,\ldots, k$ since all the other components do not appear in the limiting (as $p\to \infty$) Pohozaev identities \eqref{eq:Pohozaev--type_for_p}. This brings additional challenges, specially in the proof of (H3) - for instance, here in general we are dealing with vector solutions whose components are \emph{all} sign-changing.

Once these properties are proved, we can proceed as follows and end the proofs of Theorems \ref{thm:first_main_result} and \ref{thm:2ndmain}, proving that the optimal partition problem \eqref{eq:OPPmain} is achieved by the open regular partition $(\widetilde \omega_1,\widetilde \omega_2)$.

\begin{proof}[Proof of Theorems \ref{thm:first_main_result} and \ref{thm:2ndmain}]

The statements for $p$ fixed in Theorem \ref{thm:2ndmain} are a direct consequence of Theorem \ref{thm:3rdmain}, while the convergences of $a_{np}^i$ and $\utilde_{np}$ as $p\to \infty$ follow from \eqref{a_ip_to_0_and_non_zero} and (H1).

It is hard to check directly that $(\widetilde \omega_1,\widetilde \omega_2)$ solves \eqref{eq:OPPmain}, since we cannot ensure that $\{\tilde u_1,\ldots, \tilde u_k\}$ and $\{\tilde v_1,\ldots, \tilde v_k\}$ are mutually $H^1_0$--orthogonal (recall that (H1) only holds for the last $k-l+1$ components). So, in order to conclude our proof, we will use some auxiliary sets which are obtained via the theory of $\gamma$--convergence for quasi-open domains (cf. \cite[Section 5.3]{BucurButtazzoBook}). From \cite[Proposition 5.3.4]{BucurButtazzoBook}, we know that there exist quasi-open sets $\omega_1,\omega_2\subset \Omega$ such that
\[
\omega_{u_p}\stackrel{\gamma}{\rightharpoonup} \omega_1,\quad \omega_{v_p}\stackrel{\gamma}{\rightharpoonup} \omega_2,\qquad \text{ as } p\to \infty.
\]
Recall that this means that, taking $z_{1p}, z_{2p}$ to be the solutions of
\[
\begin{cases}
-\Delta z_{1p}=1\\
z_{1p}\in H^1_0(\omega_{u_p})
\end{cases}
\qquad 
\begin{cases}
-\Delta z_{2p}=1\\
z_{2p}\in H^1_0(\omega_{v_p})
\end{cases}
\]
then there exist $z_1,z_2\in H^1_0(\Omega)$ quasi-continuous such that (up to a subsequence)
\[
z_{1p}\rightharpoonup z_1,\quad z_{2p}\rightharpoonup z_2 \qquad \text{ in } H^1_0(\Omega),
\]
and $(\omega_1,\omega_2)=(\{ z_{1}>0\},\{ z_{2}>0\})$. Define $\widetilde \Gamma=\partial \widetilde \omega_1\cap \Omega=\partial \widetilde \omega_2\cap \Omega$.

\medbreak 

{\sc Claim 1.} We have
\begin{align*}
&\omega_1\subseteq \widetilde \omega_1\cup (\Seh_{(\tilde u,\tilde v)}\cap \widetilde \Gamma)= \widetilde \omega_1 \cup \{x\in \widetilde \Gamma:\ N(x,(\tilde u,\tilde v),0^+)>1\};\\
&\omega_2\subseteq \widetilde \omega_2\cup (\Seh_{(\tilde u,\tilde v)}\cap \widetilde \Gamma).
\end{align*}
Given $x_0\in \left\{\sum_{i=l}^k \vtilde_i^2>0\right\}$, from (H1) there exists $\delta>0$ such that $z_{2p}\geq \delta$ in $B_\delta(x_0)\subset \omega_{v_p}$. Then $z_2(x_0)>0$, and $z_1(x_0)=0$; whence $z_1=0$ a.e. in $\omega_v$, and also q.e. (quasi-everywhere). In conclusion, $\left\{\sum_{i=l}^k \vtilde_i^2>0\right\}\subset \{z_1=0\}$ in the sense of capacity. Moreover, since
\[
z_1(x)=\lim_{r\to 0} \frac{1}{|B_R|}\int_{B_r(x)} z_1 \qquad \text{q.e. in }\Omega
\]
(\cite[p. 160]{EvansGariepy}) and the common nodal $\Gamma_{(\tilde u,\tilde v)}^l$ has zero Lebesgue measure, then $\widetilde \omega_1\subseteq \{z_1=0\}$. Finally, since $z_1$ is approximately continuous q.e.\footnote{A function $u$ is approximately continuous at $x_0$ if there exists $A_{x_0}$ measurable such that $\lim_{r\to 0} |B_r(x_0)\cap A_{x_0}|/|B_r(x_0)|=1$ and $u|_{A_{x_0}}$ is continuous at $x_0$. From \cite[Remark 3.3.5]{Ziemer}, every $W^{1,2}(\R^N)$ function is approximately continuous q.e. $x\in \R^N$.} and $\widetilde \Gamma\cap \Reh_{(\tilde u,\tilde v)}$ is a regular hypersurface by (H3), then $\widetilde \omega_1\cup (\widetilde \Gamma\cap \Reh_{(\tilde u,\tilde v)})\subseteq \{z_1=0\}$, and Claim 1 follows.

\medbreak 

{\sc Claim 2.} The set $\Seh_{(\tilde u,\tilde v)}\cap \widetilde \Gamma$ has zero capacity.

\smallbreak

\noindent This is a consequence of a stratification result in the line of a result by Caffarelli and Lin, namely \cite[Theorem 4.2]{CaffarelliLinStrat}. There it is proved, for the case $k_i=1$ in \eqref{eq:OPPmain}, that the singular set $\Seh_{(\tilde u,\tilde v)}$ is stratified:
\[
S_0\subset S_1\subset \ldots \subset S_{N-2}=\Seh_{(\tilde u,\tilde v)},
\]
with $\mathscr{H}_{\rm dim}(S_i)\leq i$ and $S_{N-2}\setminus S_{N-3}$ is contained in a finite union of $(N-2)$--dimensional $C^1$ manifolds. In particular, from this we deduce that $\mathscr{H}^{N-2}(\Seh_{(\tilde u,\tilde v)}\cap \widetilde \Gamma)<\infty$, and the claim follows. The case $k_i>1$ is more delicate, but the arguments of \cite{CaffarelliLinStrat} can be adapted: the most relevant fact is that we already have a compatibility condition in the regular part of the free boundary - see condition (H3). For more details of this we refer to \cite{TavaresTerraciniStrat}.

\medbreak

Combining both claims, we have that $\lambda_k(\widetilde \omega_1)=\lambda_k(\widetilde \omega_1\cup ( \Seh_{(\tilde u,\tilde v)} \cap \widetilde \Gamma) )\leq \lambda_k(\omega_1)$ and 
$\lambda_k(\widetilde \omega_2)=\lambda_k(\widetilde \omega_2\cup ( \Seh_{(\tilde u,\tilde v)} \cap \widetilde \Gamma) )\leq \lambda_k(\omega_2)$. Thus, recalling Lemma \ref{lemma:upper_inequality c_p leq c_infty} and the fact that $\lambda_k(\cdot)$ is lower semi-continuous for the weak $\gamma$-convergence (\cite[Example 5.4.2]{BucurButtazzoBook} and \cite[Theorem 3.1]{BucurButtazzoHenrot}),

\begin{align*}
c_\infty &=  \lim_{p\to \infty} (k^{1/p} c_\infty) \geq \limsup_{p\to \infty} c_p\geq \liminf c_p \\
		&=\liminf_{p\to \infty}\left( \left(\sum_{i=1}^k \left(\lambda_i(\omega_{u_p})\right)^p\right)^{1/p}+\left(\sum_{i=1}^k \left(\lambda_i(\omega_{v_p})\right)^p\right)^{1/p}\right)\\
	         & \geq	 \liminf (\lambda_k(\omega_{u_p})+\lambda_k(\omega_{v_p}))\\
		&\geq \lambda_k(\omega_1)+\lambda_k(\omega_2)\\
		&\geq \lambda_k\left(\widetilde \omega_1\right)+ \lambda_k\left(\widetilde \omega_2\right)\geq c_\infty,
\end{align*}
which implies that $c_\infty$ is achieved by the open regular partition \eqref{eq:optimalpartition_final!}. The remaining statements now follow from (H2) and (H3).
\end{proof}

\begin{remark}\label{rem:last_remark}
Observe that the proof for the more general case \eqref{eq:OPPmain_general_with_F} follows exactly in the same way as the one of \eqref{eq:OPPmain}. In fact, in the proof of the latter case we have only used properties (F1) and (F2) - cf. the Introduction - of the function $F(x_1,\ldots, x_m)=\sum_{i=1}^m x_i$, and never its specific form. Dealing with the general case, one would have to approximate it with
\[
\inf_{(\omega_1,\ldots,\omega_m)\in \Peh_m(\Omega)} F\left( \Bigl(\sum_{j=1}^{k_1} (\lambda_j(\omega_1)^p\Bigr)^{1/p},\ldots, \Bigl(\sum_{j=1}^{k_m} (\lambda_j(\omega_1)^p\Bigr)^{1/p}\right),
\]
instead of \eqref{eq:OPPsecondary}, which would only lead to heavier notations.
\end{remark}
So, to conclude, we will show (H1) in the following subsection, and (H2)--(H3) in Subsection \ref{sec:final_regularity}.


\subsection{Proof of property (H1): Strong convergence of the last components of $u_{p}$, $v_p$}\label{sec:strong_convergence_for_some_i}

We will follow the ideas of Subsection \ref{subsec:Holderbounds}, which here apply in an easier way. First, we deduce $L^\infty$ bounds.

\begin{lemma}
There exists $C>0$, independent of $p$, such that
\[
\|u_{ip}\|_{L^\infty(\Omega)}, \|v_{ip}\|_{L^\infty(\Omega)}\leq C,\qquad \forall i=1,\ldots, k.
\]
\end{lemma}

\begin{proof}
Observe that 
\[
-\Delta |u_{ip}|\leq \frac{\mu_{ip}}{a_{ip}}|u_{ip}| \qquad \text{ in } \Omega,
\]
and
\[
\left| \frac{\mu_{ip}}{a_{ip}} \right|=\int_\Omega |\nabla u_{ip}|^2\, dx\leq c_p \leq k^{1/p} c_\infty \leq k c_\infty.
\]
Thus the result follows from a standard Brezis--Kato type argument.
\end{proof}

This is the first step to prove uniform H\"older bounds for the last $k-l+1$ components.

\begin{lemma}
For each $0<\gamma<1$ there exists $D>0$ such that
\[
\| u_{ip}\|_{C^{0,\gamma}(\overline \Omega)}, \| v_{ip}\|_{C^{0,\gamma}(\overline \Omega)}\leq D,\qquad \forall i=l,\ldots, k.
\]
\end{lemma}

\begin{proof}
Suppose, by contradiction, that
\[
L_p:= \max_{i=l,\ldots, k} \left\{ \mathop{\sup_{x,y\in \overline \Omega}}_{x\neq y}\frac{|u_{ip}(x)-u_{ip}(y)|}{|x-y|^\gamma}, \mathop{\sup_{x,y\in \overline \Omega}}_{x\neq y}\frac{|v_{ip}(x)-v_{ip}(y)|}{|x-y|^\gamma}  \right\}\to \infty
\]
as $p\to \infty$, and suppose without loss of generality that $L_p$ is achieved for $u_{kp}$ at $x_p,y_p$, and that $x_p\in \omega_{u_p}$. Since we have uniform $L^\infty$ bounds, then $r_p:=|x_p-y_p|\to 0$. Take, for $i=l,\ldots, k$, the rescaled functions 
\[
\bar u_{ip}(x)=\frac{1}{L_p r_p^\gamma}u_{ip}(x_p+r_px),\qquad \bar v_{ip}(x)=\frac{1}{L_p r_p^\gamma} v_{ip}(x_p+r_p x), \qquad x\in \Omega_p=\frac{\Omega-x_p}{r_p},
\]
which solve
\[
\begin{cases}
- a_{ip} \Delta \bar u_{ip}=r_p^2 \mu_{ip} \bar u_{ip} \qquad  &\text{ in }\frac{\omega_{u_p}-x_p}{r_p} \\
-b_{ip} \Delta \bar v_{ip} =r_p^2 \nu_{ip} \bar v_{ip} \qquad  &\text{ in }\frac{\omega_{v_p}-x_p}{r_p}.
\end{cases}
\]
Moreover, due to the normalization,
\begin{align}\label{eq:normalized=1}
\max_{i=l,\ldots, k} \left\{ \mathop{\sup_{x,y\in \overline{\Omega}_p}}_{x\neq y}\frac{|\bar u_{ip}(x)- \bar u_{ip}(y)|}{|x-y|^\gamma}, \mathop{\sup_{x,y\in \overline{\Omega}_p}}_{x\neq y}\frac{|\bar v_{ip}(x)-\bar v_{ip}(y)|}{|x-y|^\gamma}  \right\}&\\
=\left|\bar u_{kp}(0)-\bar u_{kp}\left(\frac{y_p-x_p}{r_p}\right)\right|=1&
\end{align}

Let $z_p\in \Gamma_{(u_p,v_p)}$ be such that $\dist(x_p,\Gamma_{(u_p,v_p)})=|x_p-z_p|$. We now have to split the proof in two cases:

\medbreak

\noindent \emph{Case 1. } Suppose that $|x_p-z_p|/r_p\to \infty$.

Then, given $R>0$, $B_{R\,r_p}(x_p)\subset \omega_{u_p}$, and $-a_{kp}\Delta \bar u_{kp}=r_p^2 \mu_{kp}\bar u_{kp}$ in $B_R(0)$
for large $p$. By taking the auxiliary function $\tilde u_{kp}(x):= \bar u_{kp}(x)-\bar u_{kp}(0)$, we have that it is bounded in $C^{0,\gamma}(B_R(0))$, and there exists $\tilde v_k$ such that
\[
\tilde v_{kp}\to \tilde v_k \qquad \text{ in } C^{0,\gamma'}(B_R(0)),\ \forall 0<\gamma'<\gamma.
\]
As
\[
-a_{kp}\Delta \tilde u_{kp}=r_p^2 \mu_{kp}\bar u_{kp}=\frac{\mu_{kp}r_p^{2-\gamma}}{L_p}u_{kp}(x_p+r_p\cdot)\to 0 \qquad \text{ in } B_R(0),\ \forall R>0,
\]
and $a_{kp}\to \atilde_k\neq 0$, then $\tilde v_k$ is harmonic in $\R^N$, and
\begin{equation}\label{eq:normalized_tilde_v (aux)}
\mathop{\sup_{x,y\in \R^N}}_{x\neq y} \frac{|\tilde v_k(x)-\tilde v_k(x)|}{|x-y|^\gamma}=1<\infty.
\end{equation}
Thus \cite[Corollary 2.3]{NTTV1} yields that $\tilde v_k$ is constant, which contradicts \eqref{eq:normalized_tilde_v (aux)}.

\medbreak 

\noindent \emph{Case 2. } Suppose that $|x_p-z_p|/r_p$ is bounded.

In this case, observe that this information combined with \eqref{eq:normalized=1} and
\[
\bar u_{ip}\left(\frac{z_p-x_p}{r_p}\right)=\bar v_{ip}\left(\frac{z_p-x_p}{r_p}\right)=0
\]
yields boundedness of $\bar u_{ip}$ and $\bar v_{ip}$ in $C^{0,\gamma}(B_R(0))$--norm for every $R>0$ sufficiently large, $i=l,\ldots, k$. Thus there exists $\bar u_i$, $\bar v_i$ ($i=l,\ldots, k$) such that $\bar u_{ip}\to \bar u_i$, $\bar v_{ip}\to \bar v_i$ in $C^{0,\gamma'}(B_R(0))$ for every $0<\gamma'<\gamma$ and $R>0$. As
\[
-\Delta \left(\sum_{i=l}^k |\bar u_{ip}|\right)\leq r_p^2 \sum_{i=1}^k \frac{\mu_{ik}}{a_{ik}} |\bar u_{ip}|,\quad -\Delta \left(\sum_{i=l}^k |\bar v_{ip}|\right) \leq r_p^2 \sum_{i=l}^k \frac{\nu_{ip}}{b_{ik}} |\bar v_{ip}| \qquad \text{ in } \Omega,
\]
then the limiting functions $\sum_{i=l}^k |\bar u_{i}|$ and $\sum_{i=l}^k |\bar v_{i}|$ are nonnegative in $\R^N$, have a common zero, are subharmonic in $\R^N$, have finite $C^{0,\gamma}(\R^N)$ semi-norm  and satisfy $\left( \sum_{i=l}^k |\bar u_{i}|\right)\cdot \left( \sum_{i=l}^k |\bar v_{i}|\right)\equiv 0$. Thus by the Liouville--type result \cite[Proposition 2.2]{NTTV1} we have that one of these functions is zero. By \eqref{eq:normalized=1}, we necessarily have $\sum_{i=l}^k |\bar v_{i}|\equiv 0$. Now we can finish this proof \emph{exactly} as in the final part of Step B of the proof of Theorem \ref{thm:convergences}-(i) (or as in \cite[p. 293]{NTTV1}), as soon as we prove an Almgren monotonicity formula for the remaining components.

From \eqref{eq:Pohozaev--type_for_p}, we have
\begin{align*}
(2-N) \sum_{i=1}^k \int_{B_r(x_0)} (a_{ip}|\nabla \bar u_{ip}|^2+b_{ip} |\nabla \bar v_{ip}|^2)\, dx\\
=\sum_{i=1}^k \int_{\partial B_r(x_0)} a_{ip} r(2(\partial_n \bar u_{ip})^2-|\nabla \bar u_{ip}|^2)\, d\sigma +\sum_{i=1}^k \int_{\partial B_r(x_0)} b_{ip} r (2(\partial_n \bar v_{ip})^2-|\nabla \bar v_{ip}|^2)\, d\sigma\\
+\sum_{i=1}^k \int_{\partial B_r(x_0)} r_p^2 r(\mu_{ip} \bar u_{ip}^2+\nu_{ip} \bar v_{ip}^2)\, d\sigma-\sum_{i=1}^k \int_{B_r(x_0)} r_p^2N(\mu_{ip} \bar u_{ip}^2+\nu_{ip} \bar v_{ip}^2)\, dx.
\end{align*}
and, passing to the limit as $p\to \infty$,
\begin{align*}
(2-N) \sum_{i=l}^k \int_{B_r(x_0)} \atilde_{i}|\nabla \bar u_{i}|^2\, dx=\sum_{i=l}^k \int_{\partial B_r(x_0)} \atilde_{i} r(2(\partial_n \bar u_{i})^2-|\nabla \bar u_{i}|^2)\, d\sigma
\end{align*} 
for every $x_0\in \R^N$, $r>0$ (observe that, as $a_{ip}, b_{ip}\to 0$ from $i=1,\ldots, l-1$, the previous sums start from $i=l$). Then, as in Theorem \ref{thm:Almgren}, for 
\[
N(x_0,(\bar u_l,\ldots, \bar u_k), r):=\frac{E(x_0,(\bar u_l,\ldots, \bar u_k),r)}{H(x_0,(\bar u_l,\ldots, \bar u_k),r)},
\]
with
\[
E(x_0,(\bar u_l,\ldots, \bar u_k),r)=\frac{1}{r^{N-2}} \sum_{i=l}^k \int_{B_r(x_0)} \atilde_i |\nabla \bar u_i|^2\, dx
\]
and
\[
H(x_0,(\bar u_l,\ldots, \bar u_k),r)=\frac{1}{r^{N-1}} \sum_{i=l}^k \int_{\partial B_r(x_0)} \atilde_i \bar u_i^2\, d\sigma,
\]
we can prove a monotonicity result, namely that $N(x_0,(\bar u_l,\ldots, \bar u_k),r)$ is a nondecreasing function in $r$, and 
\[
\frac{d}{dr}\log H(x_0,(\bar u_l,\ldots, \bar u_k),r)=\frac{2 N(x_0,(\bar u_l,\ldots, \bar u_k),r)}{r},\qquad \forall x_0\in \R^N,\ r>0.
\]
\end{proof}
We can now end the subsection with the proof of our first aim.

\begin{proof}[Proof of property (H1)]
The previous lemma provides the existence of $\utilde_i$, $\vtilde_i$ ($i=l,\ldots,k$) such that
\[
u_{ip}\to \utilde_i,\quad v_{ip}\to \vtilde_i \qquad \text{ in } C^{0,\gamma}(\overline \Omega),\ \forall 0<\gamma<1.
\]
Moreover, it is now clear that
\[
\begin{cases}
-\atilde_i \Delta \utilde_i =\mutilde_i \utilde_i \qquad &\text{ in } \{\sum_{i=l}^k \utilde_i^2>0\}, \\
-\btilde_i \Delta \vtilde_i=\nutilde_i \vtilde_i \qquad    &\text{ in } \{\sum_{i=l}^k \vtilde_i^2>0\}.
\end{cases}
\]
The strong $H^1_0$--convergence now follows in a standard way, combining the previous system with the one for $u_{ip}$, $v_{ip}$.
\end{proof}


\subsection{Proof of properties (H2) and (H3): Regularity of the last components of $\utilde,\vtilde$, and of the partition $(\widetilde \omega_1,\widetilde \omega_2)$}\label{sec:final_regularity}

From property (H1), proved before, we garantee the existence of $\utilde_i,\vtilde_i\in H^1_0(\Omega)\cap C^{0,\gamma}(\overline \Omega)$ for every $i=l,\ldots, k$, which satisfy
\[
-\atilde_i \Delta \utilde_i = \mutilde_i \utilde_i  \text{ in the open set } \left\{\sum_{i=l}^k \utilde_i^2 >0\right\},  
\]
while 
\[
-\btilde_i \Delta \vtilde_i =\nutilde_i \vtilde_i  \text{ in the open set } \left\{\sum_{i=l}^k \vtilde_i^2 >0\right\}.  
\]
Moreover, by passing to the limit $p\to \infty$ in \eqref{eq:Pohozaev--type_for_p}, we obtain the limiting local Pohozaev--type formula:
\begin{align}\label{eq:Pohozaev--type_at_the_limit_of_p}
(2-N) \sum_{i=l}^k \int_{B_r(x_0)} (\atilde_{i}|\nabla \utilde_{i}|^2+\btilde_{i} |\nabla \vtilde_{i}|^2)\, dx\\
=\sum_{i=l}^k \int_{\partial B_r(x_0)} \atilde_{i} r(2(\partial_n \utilde_{i})^2-|\nabla \utilde_{i}|^2)\, d\sigma +\sum_{i=l}^k \int_{\partial B_r(x_0)} \btilde_{i} r (2(\partial_n \vtilde_{i})^2-|\nabla \vtilde_{i}|^2)\, d\sigma\\
+\sum_{i=l}^k \int_{\partial B_r(x_0)} r( \mutilde_{i} \utilde_{i}^2+\nutilde_{i} \vtilde_{i}^2)\, d\sigma-\sum_{i=l}^k \int_{B_r(x_0)} N(\mutilde_{i} \utilde_{i}^2+ \nutilde_{i} \vtilde_{i}^2)\, dx,
\end{align}
for every $x_0\in \Omega$ and $r\in (0,\dist(x_0,\partial \Omega))$.

Having these properties at hand, the regularity of the pair $(\widetilde \omega_1,\widetilde \omega_2)$ follows almost exactly as in Section \ref{sec:Regularity_of_nodal_set}. In this subsection we will simply go through the proof, highlighting the similarities and the differences with respect to what we already presented. One difference is that one needs to work with the components $(\utilde_l,\ldots,\utilde_k,\vtilde_l,\ldots, \vtilde_k)$ instead of with the whole vector $(\utilde,\vtilde)$. Another delicate point is that although some results apply to the common nodal set
\[
\Gamma^l_{\utilde,\vtilde}=\{x\in \Omega:\ \utilde_i(x)=\vtilde_i(x)=0, \ \ \forall i=l,\ldots,k\},
\]
the regularity is proved only for its subsets $\partial \widetilde \omega_1$, $\partial \widetilde \omega_2$. A first simple by crucial observation is the following.

\begin{lemma}\label{lemma:common_boundary_widetilde}
We have $\partial \widetilde \omega_1 \cap \Omega=\partial \widetilde \omega_2\cap \Omega \subseteq \Gamma_{(\utilde,\vtilde)}^l$.
\end{lemma}
\begin{proof}
Take $x_0\in \partial \widetilde \omega_1\cap \Omega$, so that in particular $\sum_{i=l}^k \utilde_i^2(x_0)=0$. Suppose, in view of a contradiction, that $x_0\not\in \partial \widetilde \omega_2\cap \Omega$. Then there exists $\delta>0$ so that $\sum_{i=l}^k \vtilde_i^2\equiv 0$ in $B_\delta(x_0)$. For every $x\in B_\delta(x_0)$, as $\Gamma^l_{(\utilde,\vtilde)}$ has non empty interior, then in each neighborhood of $x$ there exists a point where $\sum_{i=l}^k \utilde_i^2>0$, whence $x\in \overline{\{\sum_{i=l}^k \utilde_i^2>0\}}$. We have concluded that $B_\delta(x_0)\subseteq \overline{\{\sum_{i=l}^k \utilde_i^2>0\}}$, therefore $x_0\in \widetilde \omega_1$, a contradiction. The other inclusion is proved in a similar way.
\end{proof}

\medbreak

Let us describe the proofs of (H2) and (H3). From the system and the local Pohozaev--type identities, we have a version of the Almgren's Monotonicity Formula (Theorem \ref{thm:Almgren}) for $(\bar u_l,\ldots, \bar u_k), (\bar v_l,\ldots, \bar v_k)$.

\begin{theorem}
Take
\[
\widetilde E(x_0,r)=\frac{1}{r^{N-2}}  \sum_{i=l}^k \int_{B_r(x_0)} (\atilde_i |\nabla \utilde_i|^2+\btilde_i |\nabla \vtilde_i|^2)\, dx,
\]
\[
\widetilde H(x_0,r)=\frac{1}{r^{N-1}}\sum_{i=l}^k \int_{\partial B_r(x_0)} (\atilde_i \utilde_i^2+\btilde_i \vtilde_i^2)\, d\sigma,
\]
and 
\[
\widetilde N(x_0,r)=\frac{\widetilde E(x_0,r)}{\widetilde H(x_0,r)}.
\]
Then given $\tilde \Omega \Subset \Omega$ there exists $C>0$ and $\tilde r>0$ such that, for every $x_0\in \tilde \Omega$, $e^{C r^2}(N(x_0,(u,v),r)+1)$ is a non decreasing function for $r\in (0,\tilde r]$, and the limit $N(x_0,(u,v),0^+):=\lim_{r\to 0^+} N(x_0,(u,v),r)$ exists and is finite. Also,
\[
\frac{d}{dr}\log(H(x_0,(u,v),r))=\frac{2}{r}N(x_0,(u,v),r)\qquad \forall r\in (0,\tilde r)
\]
and $\Gamma^l_{(\tilde u,\tilde v)}$ has no interior points.
\end{theorem}

This implies versions of Corollaries \ref{Corollaries 2.6_2.7_2.8_TaTe} and \ref{coro:Lipschiz}; in particular, $\utilde_i,\vtilde_i$ are Lipschitz continuous for $i=l,\ldots, k$ (hence (H2) is proved), and $N(x_0,0^+)\geq 1$ whenever $x_0\in \Gamma^l_{(\utilde,\vtilde)}$. Similarly to Proposition \ref{eq:propositionmeasures}, we prove that for $i\geq l$ there exist measures $\widetilde \Mcal_i$, $\widetilde \Ncal_i$ supported on $\Gamma_{(\tilde u,\tilde v)}^l$, such that
\[
-\atilde_i \Delta \utilde_i = \mutilde_i \utilde_i-\widetilde \Mcal_i,\quad -\btilde \Delta \vtilde_i=\nutilde_i \vtilde_i-\widetilde \Ncal_i \qquad \text{ in }\mathscr{D}'(\Omega).
\]

Let us now pass through the results of Section \ref{sec:Regularity_of_nodal_set}. From the monotonicity formula and the equations for $\utilde_i,\vtilde_i$, we can repeat the arguments of Subsection \ref{subset:Blowupsequences}, obtaining that, given $\tilde \Omega \Subset \Omega$ and sequences $x_n\in \widetilde \Omega$, $t_n\to 0$, then
\[
\utilde_{i,n}(x):=\frac{\bar u_i(x_n+t_n x)}{\sqrt{\widetilde H(x_n,t_n)}}\to \bar u_i,\quad \vtilde_{i,n}(x):=\frac{\bar v_i(x_n+t_n x)}{\sqrt{\widetilde H(x_n,t_n)}}\to \bar v_i
\]
in $C^{0,\gamma}_\textrm{loc}\cap H^1_\textrm{loc}(\R^N)$, for every $l \leq i\leq k$, $0<\gamma<1$. Moreover, $(\bar u_l,\ldots, \bar u_k,\bar v_l,\ldots, \bar v_k)\neq 0$, $\bar u_i\cdot \bar v_j=0$ $\forall i,j=l,\ldots,k$, and
\[
-\atilde_i \Delta \bar u_i=-\bar \Mcal_i,\quad -\btilde_i \Delta \bar v_i=-\bar \Ncal_i, 
\]
where $\bar \Mcal_i,\bar\Ncal_i$ are measures concentrated on
\[
\Gamma_{(\bar u,\bar v)}^l:=\{x\in \Omega:\ \bar u_i(x)=\bar v_i(x)=0\ \forall i=l,\ldots,k\}.
\]
Finally, we also have a local Pohozaev type identity, which yields an Almgren monotonicity formula similar to Corollary \ref{coro:Almgren_for_baru_barv}, this time for the limiting Almgren's quotient
\[
\bar N(x_0,r):=\frac{\bar E(x_0,r)}{\bar H(x_0,r)}=\frac{r \sum_{i=l}^k \int_{B_r(x_0)} (\atilde_i|\nabla \bar u_i|^2+\btilde_i |\nabla \bar v_i|^2)\, dx}{\sum_{i=l}^k \int_{\partial B_r(x_0)} (\atilde_i \bar u_i^2+\btilde_i \bar v_i^2)\, d\sigma}.
\]
In particular, $\Gamma_{(\utilde,\vtilde)}^l$ has an empty interior, and if $x_n$ is either a constant sequence, or $x_n\in \Gamma_{(\utilde,\vtilde)}^l$ satisfies $x_n\to x_0$ with $\bar N(x_0,0^+)=1$, then 
\[
\bar N(0,r)\equiv \widetilde N(x_0,r)=:\alpha \qquad \text{ for every $r>0$ }
\]
and the limiting functions in polar coordinates write as
\[
\bar u_i=r^\alpha g_i(\theta),\qquad \bar v_i=r^\alpha h_i(\theta),
\]
with 
\[
-\Delta_{S^{N-1}}g_i=\lambda g_i \quad -\Delta_{S^{N-1}}h_i=\lambda h_i \quad \text{ in } S^{N-1}\setminus \Gamma_{(\bar u,\bar v)}^l,\qquad \ \lambda=\alpha(\alpha+N-2).
\]

Following now Subsection \ref{subset:HausdorffDim}, we define
\[
\Reh_{(\utilde,\vtilde)}^l=\{x\in \Gamma^l_{(\utilde,\vtilde)}:\ \widetilde N(x,0^+)=1\}
\] 
and
\[
\Seh^l_{(\utilde,\vtilde)}=\{x\in \Gamma^l_{(\utilde,\vtilde)}:\ \widetilde N(x,0^+)>1\}.
\]
Following word by word the proofs presented there, we obtain the existence of $\delta_N>0$ such that
\[
\forall x_0\in \Gamma^l_{(\utilde,\vtilde)}, \quad \text{ either}\quad  \widetilde N(x_0,0^+)=1 \ \text{ or } \widetilde N(x_0,0^+)\geq 1+\delta_N
\]
(as in Proposition \ref{prop:jump_condition_on_N}). This implies (Theorem \ref{thm:hausdorff_measures_of_nodalsets}) that
\[
\Hh_\text{dim}(\Gamma^l_{(\utilde,\vtilde)})\leq N-1,\qquad \Hh_{dim}(\Seh^l_{(\utilde,\vtilde)})\leq N-2,
\]
which yields the first part of property (G2). Moreover, as observed in Remark \ref{rem:at_most_hyperplane}, given a blowup limit $(\bar u,\bar v)\in \Bcal\Ucal_{x_0}$ with $x_0\in \Reh^l_{(\utilde,\vtilde)}$, the nodal set $\Gamma^l_{(\utilde,\vtilde)}$ is a vector space which is at most a hyperplane; if $\bar u,\bar v\not \equiv 0$, then it is actually an hyperplane. 

By using a Clean-Up result (Proposition \ref{prop:cleanup}), we know that if $\bar v\equiv 0$ then also $\vtilde \equiv 0$ in a neighborhood of $x_0$. However, this can well be the case for some $x_0\in \Gamma^l_{(\utilde,\vtilde)}$. This is the reason why, when going through Subsections \ref{subsec:Cleanup} and \ref{subsec:regularity}, we need to restrict our attention to 
\[
\widetilde \Gamma:= \partial \widetilde \omega_1\cap \Omega=\partial \widetilde \omega_2\cap \Omega, \quad \text{ and } \quad \widetilde \Reh:=\widetilde \Gamma\cap \Reh^l_{(\utilde,\vtilde)}.
\]
Given $x_0\in \widetilde \Reh$ and $(\bar u,\bar v)\in \Bcal\Ucal_{x_0}$, exactly as in Theorem \ref{coro:hyperplane} one can now prove that $\bar u,\bar v\not \equiv 0$, and show the existence of $\nu\in S^{N-1}$ and $ \alpha_i,\beta_i\in \R$ ($l\leq i\leq k$) such that
\begin{equation}\label{blowup_full_charact_for_tilde}
\bar u_i= \alpha_i (x\cdot \nu)^+,\ \bar v_i=\beta_i (x\cdot \nu)^-, \qquad \text{ and }\qquad \sum_{i=l}^k \atilde_i \alpha_i^2=\sum_{i=l}^k \btilde_i \beta_i^2.
\end{equation}
As a corollary, we obtain that given $x_0\in\widetilde \Reh$ there exists $R_0>0$ such that
\[
B_{R_0}(x_0)\setminus \Gamma^l_{(\utilde,\vtilde)}=\Omega_1 \cup \Omega_2,
\]
with $\Omega_1\cap \Omega_2=\emptyset$, and $\Omega_1,\Omega_2$ are $(\delta,R_0)$--Reifenberg flat and NTA domains. Moreover,
\[
\sum_{i=l}^k \utilde_i^2>0,\ \sum_{i=1}^k \vtilde_i^2\equiv 0 \text{ on } \Omega_1\qquad \text{ and }\qquad \sum_{i=1}^k \utilde_i^2=0,\ \sum_{i=1}^k \vtilde_i^2>0 \text{ on } \Omega_2.
\]
Define
\[
\widetilde U(x)=(\sqrt{\atilde_l}\utilde_l(x),\ldots, \sqrt{\atilde_k}\utilde_k(x)),\qquad \Ucal(x):=\frac{\widetilde U(x)}{|\widetilde U(x)|},\qquad \text{ for } x\in \Omega_1
\]
and
\[
\widetilde V(x)=(\sqrt{\btilde_l}\vtilde_l(x),\ldots, \sqrt{\btilde_k}\vtilde_k(x)),\qquad \Vcal(x):=\frac{\widetilde V(x)}{|\widetilde V(x)|},\qquad \text{ for } x\in \Omega_2.
\]
Proposition \ref{lemma:survives_blowup_then_positive} together with \eqref{blowup_full_charact_for_tilde} implies that at least one component among $(\utilde_l,\ldots, \utilde_k)$ and among $(\vtilde_l,\ldots, \vtilde_k)$ is signed in a neighborhood of $x_0$. Thus, as before, we can extend $\Ucal,\Vcal$ up to $\Gamma:=\partial \Omega_1\cap \partial \Omega_2\cap B_{r_0}(x_0)$, and \eqref{eq:Ucal_alphaHolder} holds. The rest of Subsection \ref{subsec:regularity} can now be applied to our situation.

In conclusion, we have shown that $\widetilde R_{(\utilde,\vtilde)}$ is locally a $C^{1,\alpha}$ hypersurface. As the remaining part of $\widetilde \Gamma$ has Hausdorff dimension at most $N-2$, we have thus proved property (H2), as wanted. Having proved by now (H1), (H2) and (H3),  as seen in the end of Subsection \ref{subsec:first_subsec_of_Chapter4} we have everything we need in order to prove Theorems \ref{thm:first_main_result} and \ref{thm:2ndmain}.


\appendix

\section{Boundary behavior  on NTA and Reifenberg flat domains}\label{appendix:NTA}

\subsection{Results on NTA domains}
For convenience of the reader, we recall here the notion of \emph{non-tangencially accessible domain} (NTA), introduced in \cite{JerisonKenig} (the version we present here is taken from \cite{Lewis}).

\begin{definition}\label{def:NTA}
A bounded domain $\Omega \subset \R^N$ is called an NTA if there exist $M,r_0>0$ (the NTA constants) such that
\begin{enumerate}
\item \emph{Corkscrew condition:} Given $x\in \partial \Omega$ and $0<r<r_0$ there exist $x_0\in \Omega$ such that
\[
M^{-1} r<\dist(x_0,\partial \Omega)<|x-x_0|<M r.
\]
\item $\R^N\setminus \Omega$ satisfies the corkscrew condition
\item If $w\in \partial \Omega$ and $w_1,w_2 \in B_{r_0}(w)\cap \Omega$, then there is a rectifiable curve $\gamma : [0, 1]\to  \Omega$ with $\gamma(0)=w_1$, $\gamma(1)=w_2$, and:
\begin{itemize}
\item[(a)] $\Hh^1(\gamma) \leq M |w_1 - w_2|$,
\item[(b)] $\min\{\Hh^1(\gamma([0, t])), \Hh^1(\gamma([t, 1]))\} \leq M \dist(\gamma(t), \partial \Omega)$.
\end{itemize}
\end{enumerate} 
\end{definition}

We need the following slight extensions of the classical boundary Harnack principle by Jerison and Kenig \cite{JerisonKenig} (see also the book \cite{Kenig}):

\begin{lemma}\label{lemma:boundaryHarnack_same_lambda}
Let $\Omega$ be an NTA domain, $\lambda>0$, and $x_0\in \partial \Omega$. Then there exist $R_0,C>0$ (depending only on $\lambda$ and the NTA constants) such that for every $0<2r<R_0$ and for every $u,v$ solutions of $-\Delta u=\lambda u$ in $\Omega\cap B_{2r}(x_0)$ with $u=v$ on $\partial \Omega\cap B_{2r}(x_0)$, and $u,v> 0$ in $\Omega$, then
\[
C^{-1}\frac{v(y)}{u(y)}\leq \frac{v(x)}{u(x)}\leq C\frac{v(y)}{u(y)}
\]

for every $x\in \overline \Omega \cap B_r(x_0)$, $y\in \Omega \cap B_r(x_0)$.
\end{lemma}

\begin{proof}
 Let $R_0>0$ be so that there exists $\varphi_0\geq 0$ nontrivial solution of
\[
-\Delta \varphi_0=\lambda \varphi_0 \text{ in } B_{3R_0}(x_0),\qquad \varphi_0=0 \text{ on } \partial B_{3R_0}(x_0).
\]
Then $\varphi_0>0$ on $\overline{B_{2R_0}(x_0)}$, and
\begin{align*}
\div(\varphi_0^2\nabla (\frac{w}{\varphi_0}))&=\div(\nabla w \varphi_0-w\nabla \varphi_0)\\
								  &=\Delta w \varphi_0-\nabla w \nabla \varphi_0-\nabla w\nabla \varphi_0-w\Delta \varphi_0=0 \quad \text{ in } B_{2R_0}(x_0)
\end{align*}
for $w\equiv u$ and $w\equiv v$. Thus we can apply \cite[Lemma 1.3.7]{Kenig} to $u/\varphi_0, v/\varphi_0$, which provides the result.
\end{proof}

\begin{lemma}\label{prop:u_i/u_1_alpha_Holder_generalcase}
Let $\Omega$ be an NTA domain, $I\subset \R^+$ a bounded interval, and $x_0\in \partial \Omega$. Then there exists $R_0>0$ and $\alpha\in (0,1)$ (depending on $I$ and on the NTA constants) such that:
for every $\lambda,\mu\in I$ and $u,v$ solution of
\[
-\Delta u=\lambda u,\quad -\Delta v=\mu v \qquad \text{ in } \Omega,
\]
with $u>0$ in $\Omega$, $u,v=0$ on $\partial \Omega \cap B_{2R_0}(x_0)$ then
\[
\frac{v}{u} \text{ is $\alpha$--H\"older continuous in } \overline \Omega\cap B_{R_0}(x_0).
\]
\end{lemma}

\begin{proof} 1. Assume first that both $u,v>0$ in $\Omega$. Define $U,V:\Omega \times [-2,2]\to \R$ by
\[
U(x,x_{N+1})=e^{-x_{N+1} \sqrt{\lambda}} u(x), \qquad V(x,x_{N+1})=e^{-x_{N+1} \sqrt{\mu}} v(x).
\]
These new functions are harmonic in $\widetilde \Omega=\Omega \times (-2,2)\subset \R^{N+1}$, which is an NTA domain with the same constants of $\Omega$. Moreover, $U,V>0$ in $\widetilde \Omega$ and $U=V=0$ on $(\partial \Omega \cap B_{2R_0}(x_0))\times(-2,2)$. Then, by \cite[Theorem 7.9]{JerisonKenig} we have that
\[
\frac{V}{U} \quad \text{ is $\alpha$--H\"older in } (\overline \Omega \cap B_{R_0}(x_0))\times [-1,1]
\]
and hence also $v/u$ is $\alpha$--H\"older continuous in $\overline \Omega \cap B_{R_0}(x_0)$.

\noindent 2. If, on the other hand, $v$ changes sign, then we take the following $\mu$--harmonic extensions
\[
-\Delta w_1=\mu w_1 \quad -\Delta w_2=\mu w_2 \quad \text{ in } B_{2R_0}(x_0)\cap \Omega,
\]
\[
w_1=v^+,\quad w_2=v^- \qquad \text{ on } \partial (B_{2R_0}(x_0)\cap \Omega).
\]
These are unique for $R_0$ small (more precisely, small so that the first Laplacian eigenvalue on $B_{2R_0}(x_0)\cap \Omega$ is larger than $\mu$). Then by the previous point both $w_1/u$ and $w_2/u$ are $\alpha$--H\"older continuous, and so is $v/u=w_1/u-w_2/u$.

\end{proof}

It should be acknowledged that the transformation $U(x,x_{N+1})=e^{-\sqrt{\lambda}x_{N+1}} u(x)$ was suggested to us by S. Salsa. For the interest of the reader we would like to say that there is actually a different and more involved proof that enables to treat more general problems of type $-\Delta u=f(x,u)$, with $|f(x,s)|\leq C|s|$ for $s\sim 0$, but for the purpose of this paper the previous lemma is sufficient.

\subsection{Results on Reifenberg flat domains}

Next, we prove that if some blowup limit of a $\lambda$--harmonic function $u$ converges to a nonnegative, nontrivial function, then $u>0$ in a neighborhood of $x_0$. This result holds for Reifenberg flat domains. Let us recall its definition.

\begin{definition}\label{def:Reifenberg}
Let $\Omega\subset \R^N$ be an open set and take $\delta,R>0$. We say that $\Omega$ is a $(\delta,R)$--Reifenberg flat domain if:
\begin{itemize}
\item[-] $\forall x\in \partial \Omega$ there exists a hyperplane $H=H_{x,R}$ containing $x$, and $\nu=\nu(x,R)\in S^{N-1}$ orthogonal to $H$ such that
\begin{align*}
&\{x+t\nu\in B_R(x):\ x\in H,\ t\geq 2\delta R\}\subset \Omega,\\
&\{x- t\nu\in B_R(x):\ x\in H,\ t\geq 2\delta R\}\subset \R^N \setminus \Omega.
\end{align*}
\item[-] $\forall x\in \partial \Omega$, $0<r<R$ there exists a hyperplane $H=H_{x,r}$ containing $x$ such that
\[
d_\Hh (\partial \Omega \cap B_r(x),H\cap B_r(x))\leq \delta r.
\]
\end{itemize}
\end{definition}

 Recall that for $\delta=\delta(N)>0$ sufficiently small a $(\delta,R)$--Reifenberg flat domain is NTA \cite[Theorem 3.1]{KenigToro}. We assume from now take we have such $\delta$.

\begin{proposition}\label{lemma:survives_blowup_then_positive}
Let $\Omega$ be a Reifenberg flat domain, and take $u$ such that $-\Delta u=\lambda u$ in $\Omega$, for some $\lambda>0$. Take $x_0\in \partial \Omega$ and suppose that $u=0$ on $\partial \Omega\cap B_{R_0}(x_0)$. If, for some sequences $t_n,\rho_n\to 0$, we have 
\[
u_{n}(x):=\frac{u(x_0+t_n x)}{\rho_n}\to \bar u\geq 0 \text{ in $C^{0,\alpha}(\overline \Omega)$},
\]
with $\bar u\not \equiv 0$, then 
\[
u>0 \qquad \text{ in } \Omega\cap B_r(x_0)
\]
for some sufficiently small $r>0$.
\end{proposition}
\begin{proof}
Define by $w_n,z_n$ the $\lambda$--harmonic extensions of $u^+$ and $u^-$ on $B_{2t_n}(x_0)\cap \Omega$, namely:
\[
-\Delta w_n=\lambda w_n \quad -\Delta z_n=\lambda z_n \quad \text{ in } B_{2t_n}(x_0)\cap \Omega,
\]
\[
w_n=u^+,\quad z_n=u^- \qquad \text{ on } \partial (B_{2t_n}(x_0)\cap \Omega).
\]
and observe that $u=w_n-z_n$ in $B_{2 t_n}(x_0)\cap \Omega$. Let $\widetilde w_n$, $\widetilde z_n$ denote the blowups
\[
\widetilde w_n=\frac{w_n(x_0+t_n x)}{\rho_n}\ \text{ and }\ \widetilde z_n=\frac{z_n(x_0+t_n x)}{\rho_n}, \qquad \text{defined in $B_2(0)\cap\left( \frac{\Omega-x_0}{t_n}\right)$}. 
\]
At the limit, we find a harmonic equation in an half sphere (since $\Omega$ is a Reifenberg flat domain), and the boundary data converges to $\bar u$ and $0$ respectively; hence we have
\[
\widetilde w_n\to \bar u_i,\qquad \widetilde z_n\to 0 \qquad \text{ in a half-sphere of radius 2}.
\]
Thus there exists $\bar y\in \partial B_{1/2}(0)\cap \left(\frac{\Omega-x_0}{t_n}\right)$ such that
\[
\frac{z_n(\bar y)}{w_n(\bar y)}=\frac{\tilde z_n(\bar y)}{w_n(\bar y)}<\frac{1}{C}
\]
for $n$ large, where $C$ is the constant introduced in Lemma \ref{lemma:boundaryHarnack_same_lambda}, depending only on $\lambda$ and on the NTA constants of $\Omega$. Then by this very same lemma applied to $z_n,w_n$, we have 
\[
\frac{z_n(x_0+t_n x)}{w_n(x_0+t_n x)}\leq C\frac{\tilde z_n(\bar y)}{w_n(\bar y)}<1 \qquad \forall x\in B_1(0)\cap \frac{\Omega-x_0}{t_n},
\]
and so $u=w_n-z_n>0$ in $B_{r_n}(x_0)$ for sufficiently large $n$.
\end{proof}

\end{document}